\documentclass[a4paper,10pt]{amsart}
\usepackage[utf8]{inputenc}
\usepackage{amsmath}
\usepackage{amsfonts}
\usepackage{amssymb}
\usepackage{graphicx}
\usepackage[left=2.4cm,right=2.4cm,top=2.5cm,bottom=2.5cm]{geometry}
\numberwithin{equation}{section}
\graphicspath{ {images/} }
\usepackage{amsmath,amssymb, amsthm} 
\usepackage{lipsum}

\usepackage{amsmath}
\usepackage{amsfonts}
\usepackage{amssymb}
\usepackage{graphicx}
\numberwithin{equation}{section}
\graphicspath{ {images/} }
\usepackage{amsmath,amssymb, amsthm} 
\usepackage{lipsum}
\usepackage{stmaryrd}
\usepackage{dsfont}
\usepackage[all]{xy}

\makeatletter
\def\@settitle{\begin{center}%
  \baselineskip14\p@\relax
  \bfseries
  \uppercasenonmath\@title
  \@title
  \ifx\@subtitle\@empty\else
     \\[1ex]\uppercasenonmath\@subtitle
     \footnotesize\mdseries\@subtitle
  \fi
  \end{center}%
}
\def\subtitle#1{\gdef\@subtitle{#1}}
\def\@subtitle{}
\makeatother

\newcommand\blfootnote[1]{%
  \begingroup
  \renewcommand\thefootnote{}\footnote{#1}%
  \addtocounter{footnote}{-1}%
  \endgroup
}

\usepackage{tikz-cd}
\usepackage{hyperref}
\usepackage{bm}
\hypersetup{colorlinks, linkcolor=blue, citecolor=black, urlcolor=black}

\makeatletter
\def\@settitle{\begin{center}%
  \baselineskip14\p@\relax
  \bfseries
  \uppercasenonmath\@title
  \@title
  \ifx\@subtitle\@empty\else
     \\[1ex]\uppercasenonmath\@subtitle
     \footnotesize\mdseries\@subtitle
  \fi
  \end{center}%
}
\def\subtitle#1{\gdef\@subtitle{#1}}
\def\@subtitle{}
\makeatother

\usepackage{tikz-cd}
\usepackage{hyperref}
\usepackage{bm}
\hypersetup{colorlinks, linkcolor=blue, citecolor=black, urlcolor=black}

\usepackage[english]{babel}
\usepackage[colorinlistoftodos]{todonotes}

\usepackage[english]{babel}
\usepackage[colorinlistoftodos]{todonotes}

\theoremstyle{plain}
\newtheorem{thm}{Theorem}[subsection] 

\usepackage{mathrsfs}

\theoremstyle{definition}
\newtheorem{defi}[thm]{Definition}
\newtheorem{conjecture}[thm]{Conjecture}

\newtheorem{rmk}[thm]{Remark}

\newtheorem{rem}[thm]{Remark}
\theoremstyle{definition}

\theoremstyle{plain}
\newtheorem{prop}[thm]{Proposition}
\theoremstyle{plain}
\newtheorem{lemma}[thm]{Lemma}
\theoremstyle{plain}
\newtheorem{cor}[thm]{Corollary}
\theoremstyle{plain}
\newtheorem{thmintro}{Theorem}

\newtheorem{conj}[thmintro]{Conjecture}


\newcounter{parentnumber}

\DeclareMathOperator{\ord}{ord}

\newcommand{\Hom}{\operatorname{Hom}}
\newcommand{\Frac}{\operatorname{Frac}}

\newcommand{\ind}{\operatorname{ind}}
\newcommand{\Gal}{\operatorname{Gal}}

\newcommand{\bQ}{\mathbb{Q}}

\makeatletter  
\newcommand{\colim@}[2]{%
  \vtop{\m@th\ialign{##\cr
    \hfil$#1\operator@font colim$\hfil\cr
    \noalign{\nointerlineskip\kern1.5\ex@}#2\cr
    \noalign{\nointerlineskip\kern-\ex@}\cr}}%
}
\newcommand{\colim}{%
  \mathop{\mathpalette\colim@{\rightarrowfill@\scriptscriptstyle}}\nmlimits@
}
\renewcommand{\varprojlim}{%
  \mathop{\mathpalette\varlim@{\leftarrowfill@\scriptscriptstyle}}\nmlimits@
}
\renewcommand{\varinjlim}{%
  \mathop{\mathpalette\varlim@{\rightarrowfill@\scriptscriptstyle}}\nmlimits@
}
\makeatother

\newcommand{\M}{\mathfrak{M}}

\newcommand{\Z}{\mathbb{Z}}
\newcommand{\Q}{\mathbb{Q}}

\newcommand{\Qp}{\mathbb{Q}_{p}}
\newcommand{\Zp}{\mathbb{Z}_{p}}

\newcommand{\R}{\mathbb{R}}
\newcommand{\C}{\mathbb{C}}

\newcommand{\F}{\mathcal{F}}

\newcommand{\X}{\mathfrak{X}}

\newcommand{\et}{\rm \acute{e}t}

\newcommand{\bg}{\boldsymbol{g}}

\newcommand{\sk}{\vspace{0.1in}}

\font\wncyr=wncyr9.8
\newcommand{\sha}{\text{\wncyr{W}}}
\newcommand{\CE}{\mathcal{E}}
\newcommand{\CM}{\mathcal{M}}

\DeclareMathOperator{\SL}{SL}
\DeclareMathOperator{\ch}{ch}

\DeclareMathOperator{\GL}{GL}
\DeclareMathOperator{\frob}{Frob}

\DeclareMathOperator{\Gr}{Gr}

\DeclareMathOperator{\End}{End}

\DeclareMathOperator{\loc}{loc}

\DeclareMathOperator{\im}{Im}
\DeclareMathOperator{\Aut}{Aut}
\DeclareMathOperator{\length}{length}

\newcommand{\Lcal}{\mathcal{L}}
\newcommand{\cF}{\mathcal{F}}
\newcommand{\CF}{\mathcal{F}}
\newcommand{\Fcal}{\mathcal{F}}

\newcommand{\cL}{\mathscr{L}}
\newcommand{\cN}{\mathscr{N}}
\newcommand{\kap}{\kappa}

\newcommand{\fm}{\mathfrak{m}}
\newcommand{\rH}{\mathrm{H}}

\begin{document}
\title{Mazur's main conjecture 
at Eisenstein primes}
\author{Francesc Castella}
\address[F.~Castella]{University of California Santa Barbara, South Hall, Santa Barbara, CA 93106, USA}
\email{castella@ucsb.edu}

\author{Giada Grossi}
\address[G.~Grossi]{CNRS, Institut Galilée, Université Sorbonne Paris Nord, 93430 Villetaneuse, FRANCE}
\email{grossi@math.univ-paris13.fr}

\author{Christopher Skinner} 
\address[C.~Skinner]{Princeton University, Fine Hall, Washington Road, Princeton, NJ 08544-1000, USA}
\email{cmcls@princeton.edu}

\begin{abstract}
Let $E/\Q$ be an elliptic curve and let $p$ be an odd prime of good reduction for $E$. Assume that $E$ admits a rational $p$-isogeny $\varphi:E\rightarrow E'$, and let $\phi:G_{\bQ}\rightarrow\mathbb{F}_p^\times$ be the character by which $G_{\bQ}$ acts on ${\rm ker}(\varphi)$. In this paper, we prove the Iwasawa main conjecture for $E$, as formulated by B.\,Mazur in 1972, when $\phi\vert_{G_p}\neq 1,\omega$, where $G_p\subset G_{\bQ}$ is a decomposition group at $p$ and $\omega$ is the Teichm\"uller character. 

Two key innovations in our proof are a Kolyvagin system argument for the Selmer group of $E$ twisted by anticyclotomic Hecke characters arbitrarily close to the trivial character, and a congruence argument exploiting Beilinson--Flach classes and their explicit reciprocity laws. 
%
\end{abstract}

\maketitle
\tableofcontents

\section{Introduction}
\addtocontents{toc}{\protect\setcounter{tocdepth}{2}}

Let $E$ be an elliptic curve over $\Q$, and let $p$ be an odd prime of good reduction for $E$. In the early 1970s, motivated by Iwasawa's theory for the $p$-part of class groups of number fields in $\Z_p$-extensions, Mazur 
initiated a parallel study for the arithmetic of $E$ over the cyclotomic $\Z_p$-extension $\Q_\infty/\Q$. In particular, in \cite{mazur-towers} he proved a foundational ``control theorem'' for the $p$-primary Selmer group ${\rm Sel}_{p^\infty}(E/\Q_\infty)$, viewed as a module over the Iwasawa algebra $\Lambda_\Q=\Z_p\llbracket{{\rm Gal}(\Q_\infty/\Q)\rrbracket}$,  
and formulated an analogue of Iwasawa's main conjecture, expressing the characteristic ideal of the Pontryagin dual of $\X_{\rm ord}(E/\Q_\infty)={\rm Sel}_{p^\infty}(E/\Q_\infty)^\vee$ in terms of the $p$-adic $L$-function attached to $E$ by the work Mazur--Swinnerton-Dyer \cite{M-SwD}:
\begin{equation}\label{MC}
\ch_{\Lambda_\Q}\bigl(\X_{\rm ord}(E/\Q_\infty)\bigr)\overset{}=\bigl(\Lcal_p^{\rm MSD}(E/\Q)\bigr)\tag{MC}
\end{equation}
(see [\emph{op.\,cit.}, Conj.~3]). By the Weierstrass preparation theorem, conjecture (\ref{MC}) can be viewed as the equality between two integral $p$-adic polynomials attached to $E$, one by means of the arithmetic of $E$ (i.e., its Mordell--Weil and Tate--Shafarevich groups) 
over $\Q_\infty$ and the other by 
means of the modular symbols associated with $E$ (available thanks to its modularity \cite{wiles,TW,BCDT}), encoding the special values of the Hasse--Weil $L$-function of $E$ twisted by finite order characters of ${\rm Gal}(\Q_\infty/\Q)$.

The main result in this paper is the proof (under a mild hypothesis) of Mazur's main conjecture (\ref{MC}) when $p$ is an 
\emph{Eisenstein prime} for $E$, meaning that $E$ admits a rational $p$-isogeny. (In the non-Eisenstein case, the conjecture is largely known \cite{kato-euler-systems,skinner-urban,wan-hilbert}.) In other words, we consider the case where $E[p]$ is \emph{reducible} as a module over the absolute Galois group $G_\Q={\rm Gal}(\overline{\Q}/\Q)$, so that
\[
E[p]^{ss}=\mathbb{F}(\phi)\oplus\mathbb{F}(\psi)
\] 
as $G_\Q$-modules, where $\phi,\psi:G_\Q\rightarrow\mathbb{F}^\times$ are characters whose product $\phi\psi=\omega$ is the mod $p$ cyclotomic character. 

\begin{thmintro}\label{thm:CYC}
Let $E/\Q$ be an elliptic curve, and let $p>2$ be a prime of good reduction  for $E$. Suppose that $p$ is Eisenstein with 
$\phi\vert_{G_p}\neq 1,\omega$, 
where $G_p\subset G_\Q$ is a decomposition group at $p$. Then $\X_{\rm ord}(E/\Q_\infty)$ is $\Lambda_\Q$-torsion with
\[
\ch_{\Lambda_\Q}\bigl(\X_{\rm ord}(E/\Q_\infty)\bigr)\overset{}=\bigl(\Lcal_p^{\rm MSD}(E/\Q)\bigr),
\]
and hence Mazur's main conjecture holds.
\end{thmintro}

In the setting of Theorem~\ref{thm:CYC}, Kato proved \cite{kato-euler-systems} that $\X_{\rm ord}(E/\Q_\infty)$ is $\Lambda_\Q$-torsion and that its characteristic ideal contains $\Lcal_p^{\rm MSD}(E/\Q)$ after inverting $p$. This ambiguity of powers of $p$ was subsequently removed by W\"uthrich \cite{wuthrich-int}. On the other hand, in a foundational paper \cite{greenvats}, Greenberg--Vatsal proved conjecture (\ref{MC}) for Eisenstein primes $p$ under the assumption that
\begin{equation}\label{GV}
\textrm{$\phi$ is either}\,
\begin{cases}
\textrm{unramified at $p$ and odd, or} \\
\textrm{ramified at $p$ and even.}\tag{GV}
\end{cases}
\end{equation}
Under this hypothesis, they could show that both $\X_{\rm ord}(E/\Q_\infty)$ and $\Lcal_p^{\rm MSD}(E/\Q)$ have vanishing $\mu$-invariant by building on the work of Ferrero--Washington \cite{FW} and 
Mazur--Wiles \cite{MW}, thereby reducing their result on (\ref{MC}) to a delicate comparison of Iwasawa $\lambda$-invariants.

Without hypothesis (\ref{GV}), the situation is known to be more complicated. Indeed,  Greenberg showed that if $E[p^\infty]$ contains a $G_\Q$-invariant cyclic subgroup $\Phi$ of order $p^m$ which is ramified at $p$ and odd, then $\X_{\rm ord}(E/\Q_\infty)$ has $\mu$-invariant $\geq m$ (see \cite[Prop.~5.7]{greenberg-cetraro}). On the analytic side, a conjecture by Stevens \cite{stevens-invmath} predicts a similar phenomenon for $\Lcal_p^{\rm MSD}(E/\Q)$.
The methods in this paper allow us to prove Mazur's main conjecture (\ref{MC}) for Eisenstein primes regardless of the value of the $\mu$-invariant. In particular, our results include giving a new proof for the case previously handled by Greenberg--Vatsal.

\subsection{Some ideas from the proof} 

Our proof of Theorem~\ref{thm:CYC} goes through  \emph{anticyclotomic} Iwasawa theory for $E$ over an auxiliary imaginary quadratic field $K/\Q$ in which $p$ splits. Roughly speaking, this is used to show that $\mathfrak{X}_{\rm ord}(E/\Q_\infty)$ and $\Lcal_p^{\rm MSD}(E/\Q)$ have the same Iwasawa invariants\footnote{Strictly speaking, our results only show these equalities for \emph{the sum} of the Iwasawa invariants of $E$ and the quadratic twist $E^K$, but this suffices for the proof of Theorem~\ref{thm:CYC} thanks to Kato's work.}: 
\[
\mu(\mathfrak{X}_{\rm ord}(E/\Q_\infty))=\mu(\Lcal_p^{\rm MSD}(E/\Q))
\quad\textrm{and}\quad
\lambda(\mathfrak{X}_{\rm ord}(E/\Q_\infty))=\lambda(\Lcal_p^{\rm MSD}(E/\Q)),
\]
even in situations of positive $\mu$-invariant. 


More precisely, our argument rests on the proof of an Iwasawa main conjecture for $E$ over the anticyclotomic $\Z_p$-extension of $K$ (see Theorem~\ref{thm:AC} below), and a congruence argument building on the cyclotomic Euler system of Beilinson--Flach classes of Lei--Loeffler--Zerbes \cite{RSCMF} and Kings--Loeffler--Zerbes \cite{explicit}. 

 
%

\subsubsection*{Anticyclotomic main conjectures}

Denote by $N$ the conductor of $E$, and let $K$ be an imaginary quadratic field such that
\begin{equation}\label{eq:intro-disc}
\textrm{the discriminant $D_K$ is odd and $D_K\neq -3$,}\tag{disc}
\end{equation}
and such that the following \emph{Heegner hypothesis} holds:
\begin{equation}\label{eq:intro-Heeg}
\textrm{every prime $\ell\vert N$ splits in $K$.}\tag{Heeg}
\end{equation}
Let $\Gamma_K^-={\rm Gal}(K^{-}_{\infty}/K)$ be the Galois group of the anticyclotomic $\Z_p$-extension of $K$, and 
for each $n$ denote by $K_n^-$ the subfield of $K^{-}_{\infty}$ with $[K_n^-:K]=p^n$. In sharp contrast with the case of $\Q_\infty/\Q$, one can show that ${\rm rank}_\Z E(K_n^-)$ is unbounded as $n\to\infty$, and therefore the Pontryagin dual $\X_{\rm ord}(E/K_\infty^-)$ 
of the Selmer group ${\rm Sel}_{p^\infty}(E/K_\infty^-)$ has positive rank as a module over the anticyclotomic Iwasawa algebra $\Lambda_K^-=\Z_p\llbracket{\Gamma_K^-}\rrbracket$. This unbounded growth is accounted for by the existence of Heegner points on $E$ associated with a given modular parametrization
\[
\pi:X_0(N)\rightarrow E.
\]
This system of points gives rise to a $\Lambda_K^-$-adic class $\kappa_1^{\rm Hg}$ which was shown to be non-torsion by Cornut \cite{cornut} and Vatsal \cite{vatsal} in their proof of ``Mazur's conjecture'' on higher Heegner points \cite{mazur-ICM83}. 
As recalled below, a formulation of the Iwasawa main conjecture in this context was given by Perrin-Riou \cite{perrinriou}. Write
\[
\mathfrak{S}_{\rm ord}(E/K_\infty^-):=\varprojlim_n\varprojlim_m{\rm Sel}_{p^m}(E/K_n^-),
\]
which is a compact $\Lambda_K^-$-module containing $\kappa_1^{\rm Hg}$. 

%
\begin{conj}[Perrin-Riou]\label{conj:PR}
Let $E/\Q$ be an elliptic curve, $p>2$ a prime of good ordinary reduction for $E$, and let $K$ be an imaginary quadratic field satisfying {\rm(\ref{eq:intro-Heeg})} and {\rm(\ref{eq:intro-disc})}. Then $\mathfrak{S}_{\rm ord}(E/K_\infty^-)$ and $\X_{\rm ord}(E/K_\infty^-)$ both
have $\Lambda^-_K$-rank one, and
\[
{\rm char}_{\Lambda^-_{K}}\bigl(\X_{\rm ord}(E/K_\infty^-)_{\rm tors}\bigr)={\rm char}_{\Lambda^-_{K}}\bigl(\mathfrak{S}_{\rm ord}(E/K_\infty^-)/(\kappa_1^{\rm Hg})\bigr)^2, 
\] 
where the subscript {\rm tors} denotes the $\Lambda^-_{K}$-torsion submodule. 
\end{conj}

The first general result towards Conjecture~\ref{conj:PR} for Eisenstein primes $p$ was obtained in \cite{eisenstein}. Namely, it was shown 
that Perrin-Riou's main conjecture holds for Eisenstein primes $p$ under the additional hypotheses that \begin{equation}\label{eq:intro-spl}
\textrm{$(p)=v\bar{v}$ splits in $K$,}\tag{spl} 
\end{equation} 
that $\phi\vert_{G_p}\neq 1,\omega$, and that 
\begin{equation}\label{eq:intro-Sel}
\textrm{the $\Z_p$-corank of ${\rm Sel}_{p^\infty}(E/K)$ is $1$}.\tag{Sel}
\end{equation}

In particular, hypothesis (\ref{eq:intro-Sel}), 
which obviously excludes elliptic curves with ${\rm rank}_\Z E(\Q)\geq 2$, was imposed to account for the inability of the methods in [\emph{op.\,cit.}, \S{3}] to control certain error terms appearing at height one primes of $\Lambda_K^-$ approaching the augmentation ideal $\mathfrak{P}_0\subset \Lambda_K^-$. 

A key technical innovation in this paper is the proof of a Kolyvagin system argument for the Selmer group of $E/K$ twisted by characters of $\Gamma_K^-$ arbitrarily close to $1$ yielding the ``upper bound''  divisibility in Conjecture~\ref{conj:PR} after localization at height one primes of $\Lambda_K^-$ approaching $\mathfrak{P}_0$. 
Together with complementary results obtained in \cite{eisenstein}, this argument yields the following.

\begin{thmintro}\label{thm:AC}
Let $E/\mathbb{Q}$ be an elliptic curve, let $p\nmid 2N$ be an Eisenstein prime for $E$, and let $K$ be an imaginary quadratic field satisfying {\rm (\ref{eq:intro-Heeg})}, {\rm (\ref{eq:intro-disc})}, and {\rm (\ref{eq:intro-spl})}. Suppose that   
$\phi\vert_{G_p}\neq 1,\omega$. Then Conjecture~\ref{conj:PR} holds.
\end{thmintro}

To go from Theorem~\ref{thm:AC} to Theorem~\ref{thm:CYC}, we use a reformulation of the former in terms of $p$-adic $L$-functions. Let $\Z_p^{\rm ur}$ be the completion of the ring of integers of the maximal unramified extension of $\Q_p$, and put
\[
\Lambda^{-,\rm ur}_K:=\Lambda_K^-\widehat\otimes_{\Z_p}\Z_p^{\rm ur}.
\]
It follows from the explicit reciprocity law for $\kappa_1^{\rm Hg}$, that Conjecture~\ref{conj:PR} is equivalent to the Iwasawa--Greenberg main conjecture for the $p$-adic $L$-function $\Lcal_p^{\rm BDP}(f/K)$ introduced in \cite{BDP}. 
Under the same hypotheses of Theorem~\ref{thm:AC}, we thus deduce that a different Greenberg Selmer group denoted $\X_{\rm Gr}(E/K_\infty^-)$ 
is $\Lambda_K^-$-torsion, with
\[
{\rm char}_{\Lambda_K^-}\bigl(\mathfrak{X}_{\rm Gr}(E/K_{\infty}^-)\bigr)\Lambda_K^{-,\rm ur}=\bigl(\Lcal_p^{\rm BDP}(f/K)\bigr)
\]
as ideals in $\Lambda_K^{-,\rm ur}$. This equality of characterisitc ideals is the first key ingredient in the proof of Theorem~\ref{thm:CYC}.

\subsubsection*{Comparing Iwasawa invariants}

In \cite{RSCMF} and \cite{explicit}, 
Lei--Loeffler--Zerbes and Kings--Loeffler--Zerbes constructed a cyclotomic Euler system 
(over $\Q$) for the Rankin--Selberg convolution of two modular forms moving in Hida families. To be able to use these classes as a bridge between the anticyclotomic $\Z_p$-extension $K_\infty^-/K$ and the cyclotomic $\Z_p$-extension $\Q_\infty/\Q$ (or rather its translate by $K$),  
here we are led to consider a variant of their construction attached to the pair $(f,\boldsymbol{g})$, where $f$ 
is the weight $2$ newform attached to $E$, and $\boldsymbol{g}$ is a suitable CM Hida family. Because our $\boldsymbol{g}$ specializes in weight $1$ to the $p$-irregular Eisenstein series ${\rm Eis}_{1,\eta}$, 
where $\eta=\eta_{K/\Q}$ is the quadratic character associated to $K/\Q$, in fact we use a refinement of the construction in \cite{explicit} studied in \cite{BST}.

Let $\Lambda_K$ (resp. $\Lambda_K^+$) be the Iwasawa algebra for the $\Z_p^2$-extension of $K$ (resp. the cyclotomic $\Z_p$-extension $K_\infty^+/K$). In particular, from these works we obtain a two-variable Iwasawa cohomology class 
\[
BF\in\rH^1_{\rm Iw}(K_\infty,T_pE_\bullet),
\]
where $E_\bullet/\Q$ is the distinguished elliptic curve in the isogeny class of $E$ constructed by W\"uthrich \cite{wuthrich-int}. Combined with the  relations between different $p$-adic $L$-functions established in Sect.~\ref{sec:Lp}, we also deduce two explicit reciprocity laws:
\begin{enumerate}
\item One relating ${\rm loc}_{\bar{v}}(BF)$ to a $p$-adic $L$-function $\Lcal_p^{\rm PR}(E_\bullet/K)\in\Lambda_K$ whose image $\Lcal_p^{\rm PR}(E_\bullet/K)^+$ under the natural projection $\Lambda_K\rightarrow\Lambda_K^+$ satisfies
\[
\Lcal_p^{\rm PR}(E_\bullet/K)^+=\Lcal_p^{\rm MSD}(E_\bullet/\Q)\cdot\Lcal_p^{\rm MSD}(E_\bullet^K/\Q)
\]
up to a unit, where $E_\bullet^K$ is the twist of $E_\bullet$ by the quadratic field $K$.
\item A second one relating ${\rm loc}_v(BF)$ to a $p$-adic $L$-function $\Lcal_p^{\rm Gr}(f/K)\in\Lambda_K^{\rm ur}=\Lambda_K\widehat\otimes\Z_p^{\rm ur}$ whose image $\Lcal_p^{\rm Gr}(f/K)^-$ under the natural projection $\Lambda_K^{\rm ur}\rightarrow\Lambda_K^{-,\rm ur}$ satisfies
\[
\Lcal_p^{\rm Gr}(f/K)^-=\Lcal_p^{\rm BDP}(f/K)
\]
up to a unit. 
\end{enumerate}


The next key idea to have Theorem~\ref{thm:AC} to bring to bear on the proof of Theorem~\ref{thm:CYC} is to consider the restriction of $BF$ to $\rH^1_{\rm Iw}(K_\infty^+,T_pE_\bullet)$, and exploit the fact that the anticyclotomic and the cyclotomic $\Z_p$-extensions meet at the trivial character (in other words, $K_\infty^-\cap K_\infty^+=K$). When $\Lcal_p^{\rm BDP}(f/K)(1)\neq 0$ (which by the main result of \cite{BDP} can only happen when ${\rm rank}_\Z E(\Q)\leq 1$), the argument for the implication Theorem~\ref{thm:AC}$\Longrightarrow$Theorem~\ref{thm:CYC} is relatively simple, and to help orient the reader, this simpler case is explained in detail in Sect.~\ref{sec:one}.  

To make the argument work in arbitrary rank, we take a character $\alpha$ of $\Gamma_K^-$ with
\[
\alpha\equiv 1\;({\rm mod}\,p^m)
\]
for some $m\gg 0$ such that $\Lcal_p^{\rm BDP}(f/K)(\alpha)\neq 0$ (as always possible by the nonvanishing of $\Lcal_p^{\rm BDP}(f/K)$), and consider $\alpha$-twisted versions of the above Selmer groups and $p$-adic $L$-functions projected to $\Lambda_K^+$. With the aid of a cyclotomic Euler system extending the $\alpha$-twist of the class $BF$ projected to $\rH^1_{\rm Iw}(K_\infty^+,T_pE_\bullet)$, building on Theorem~\ref{thm:AC} we deduce that the twisted Selmer group $\mathfrak{X}_{\rm ord}(E_\bullet(\alpha)/K_\infty^+)$ is $\Lambda_K^+$-torsion, with
\begin{equation}\label{eq:alpha-MC}
\ch_{\Lambda_K^+}\bigl(\X_{\rm ord}(E_\bullet(\alpha)/K_\infty^+)\bigr)=\bigl(\Lcal_p^{\rm PR}(E_\bullet(\alpha)/K)^+\bigr)
\end{equation} 
as ideals in $\Lambda_K^+$. On the other hand, Kato's divisibility \cite{kato-euler-systems} (as refined by W\"uthrich \cite{wuthrich-int}) applied to $E_\bullet$ and $E_\bullet^K$ implies that the untwisted Selmer group $\mathfrak{X}_{\rm ord}(E_\bullet/K_\infty^+)$ is $\Lambda_K^+$-torsion, with the divisibility
\begin{equation}\label{eq:one-MC}
\ch_{\Lambda_K^+}\bigl(\X_{\rm ord}(E_\bullet/K_\infty^+)\bigr)\supset\bigl(\Lcal_p^{\rm PR}(E_\bullet/K)^+\bigr)
\end{equation} 
as ideals in $\Lambda_K^+$. By a new congruence argument using the study of the variation of both sides of \eqref{eq:alpha-MC} as $\alpha$ varies carried out in the earlier parts of the paper, 
we deduce from this equality (for $\alpha$ sufficiently close to $1$) that $\X_{\rm ord}(E_\bullet/K_\infty^+)$ and $\Lcal_p^{\rm PR}(E_\bullet/K)^+$ have the same Iwasawa invariants, and so equality holds in (\ref{eq:one-MC}). The proof of Theorem~\ref{thm:CYC} for both $E_\bullet$ and the original elliptic curve $E$, can then be deduced from Kato's work and the invariance of Mazur's main conjecture under isogenies.  


\subsection{Application to the $p$-part of the Birch--Swinnerton Dyer formula}

As a standard consequence of Theorem~\ref{thm:CYC}, 
 we deduce most cases of the $p$-part of the Birch--Swinnerton Dyer formula for elliptic curves $E/\Q$ of analytic rank $\leq 1$ at Eisenstein primes $p$. 

\begin{thmintro}\label{thm:pBSD}
Let $E/\Q$ and $p>2$ be as in Theorem~\ref{thm:CYC}, and let $r\in\{0,1\}$. If ${\rm ord}_{s=1}L(E,s)=r$, then
\begin{displaymath}
\ord_{p}\biggl(\frac{L^{(r)}(E, 1)}{\operatorname{Reg}(E / \Q) \cdot \Omega_{E}}\biggr)=\ord_{p}\biggl(\# \sha(E / \Q) \prod_{\ell \nmid \infty} c_{\ell}(E / \Q)\biggr),
\end{displaymath}
where 
\begin{itemize}
\item ${\rm Reg}(E/\Q)$ is the regulator of $E(\Q)$, 
\item $\Omega_E=\int_{E(\mathbb{R})}\vert\omega_E\vert$ is the N\'{e}ron  period associated to the N\'{e}ron differential $\omega_E$, and 
\item $c_\ell(E/\Q)$ is the Tamagawa number of $E$ at the prime $\ell$.
\end{itemize}
In other words, the $p$-part of the Birch--Swinnerton-Dyer formula for $E$ holds.
\end{thmintro}

\begin{proof}
In the case $r=0$, the argument is the same as in  \cite[Thm.~5.1.4]{eisenstein}, replacing the appeal to \cite[Thm.~1.3]{greenvats} by an appeal to our Theorem~\ref{thm:CYC}. In the case $r=1$, we argue as in \cite[Thm.~5.3.1]{eisenstein}, choosing a suitable $K$ with $L(E^K,1)\neq 0$ and applying our result in the rank $0$ case  to $E^K$.
\end{proof}

\begin{rem}
Previously, by the work of Greenberg--Vatsal \cite{greenvats} in analytic rank $0$, and of the authors with Lee \cite{eisenstein} in analytic rank $1$, Theorem~\ref{thm:pBSD} was only known for roughly ``half'' of the $p$-Eisenstein cases, as both results required a certain parity hypothesis on $\phi$.
\end{rem}


It is easy to produce examples of elliptic curves to which Theorem \ref{thm:pBSD} can be applied. Let $p=5$ and $J$ be any of the three non-isomorphic elliptic curves of conductor 11; one could take for example 
the elliptic curve of Weierstrass equation
\[
y^2+y=x^3-x^2-10x-20,
\]
which is the strong Weil curve of conductor $11$.  
It is known that $J[5]$ is reducible and its semi-simplification is isomorphic to $\mu_5\oplus \Z/5\Z$. Moreover, the rank of $J$ is equal to zero. We can then consider $\psi$ any quadratic character such that $\psi(5)=-1$ and take $E=J_\psi$ the quadratic twist of $J$ by $\psi$. Since the root number of $J$ is 1, applying \cite[Thm.~B.1]{twist}, we can produce infinitely many $\psi$ such that $E=J_\psi$ has analytic (and hence algebraic) rank zero and for which the hypothesis of Theorem \ref{thm:pBSD} are satisfied for $p=5$. Note that prior to this paper, the only elliptic curves in this family for which the theorem was known are of the form $E=J_\psi$ where $\psi$ is odd.

Similarly, take $p=5$ and consider the elliptic curve
\begin{displaymath}
J: y^2+y=x^3+x^2-10x+10.
\end{displaymath} 
The curve $J$ has conductor $123$ and analytic rank $1$, and satisfies $J[5]^{ss}=\mu_5\oplus \Z/5\Z$ as $G_\Q$-modules. Let $\psi$ any quadratic character such that $\psi(5)=-1$ and take $E=J_\psi$ the quadratic twist of $J$ by $\psi$. Since the root number of $J$ is $-1$ (being of analytic rank one), by \cite[Thm.~B.2]{twist} we can find infinitely many $\psi$ as above for which $E=J_\psi$ also has analytic rank one and hence to which Theorem~\ref{thm:pBSD} can be applied for $p=5$. 
Earlier result in this direction would only apply to the elliptic curve in this family of the form $E=J_\psi$ with $\psi$ even. 

One can proceed similarly for $p=3,7,13,37$, taking quadratic twists of elliptic curves of rank zero or one and with a rational $p$-isogeny. If the character describing the kernel of the isogeny is not trivial (which has to be the case for $p=13$ or $p=37$), one might have to impose further conditions to the quadratic character at some bad primes in order to apply \cite[Thm.~B]{twist}.

\subsection{Further applications and relation to previous works}

In addition to being a key ingredient in the proof of our main results, the Kolyvagin system argument with error terms for the $p$-adic Tate module of $E$ twisted by characters close to the trivial one 
contained in Sect.~\ref{sec:anticyc} (see Theorem~\ref{thm:Zp-twisted}) is used in a forthcoming work of the authors with A.\,Burungale \cite{BCGS} to establish Kolyvagin's conjecture on the nonvanishing of derived Heegner classes under  mild hypotheses on $E[p]$, including the first cases for  Eisenstein primes $p$ (many cases in the non-Eisenstein case were first proved by W.\,Zhang \cite{wei-zhang}).

Finally, we note that the idea of using Beilinson--Flach classes to relate different Iwasawa main conjectures has appeared in earlier works, notably in \cite{wan-ss,CW-adv}, but the congruence argument mechanism introduced in this paper to deduce Theorem~\ref{thm:CYC} from Theorem~\ref{thm:AC} is new, and should be useful in other settings, 
as we plan to exploit in forthcoming work. 




\subsection{Outline of the paper}

We begin in Sect.~\ref{sec:Lp} by introducing the different one- and two-variable 
$p$-adic $L$-functions, and the various relations between them that will be needed for our arguments. Then in Sect.~\ref{sec:Sel} we introduce corresponding Selmer groups, both primitive and imprimitive, and prove some ancillary results for the latter ones. In Sect.~\ref{BF} we discuss Beilinson--Flach classes and some direct applications of their relations with $p$-adic $L$-functions. As a guide to the general case, 
in Sect.~\ref{sec:one} we give a simplified proof of Theorem~\ref{thm:CYC} in the case of rank one.  
In Sect.~\ref{sec:anticyc}
we prove our new Kolyvagin system bound for twists by characters arbitrarily close to $1$, resulting in Theorem~\ref{thm:AC}.
Finally, in Sect.~\ref{sec:MC} we run our congruence argument and conclude the proof of Theorem~\ref{thm:CYC}.

\subsection{Acknowledgements}
We would like to thank Ashay Burungale, Shinichi Kobayashi, and Romyar Sharifi for useful exchanges related to this work. We also thank the referees for their useful comments and suggestions.  
During the preparation of this paper, F.C. was supported by the NSF grants DMS-2101458 and DMS-2401321; G.G. was partially supported by the postdoctoral fellowship of the Fondation Sciences Math\'{e}matiques de Paris; C.S. was partially supported by the Simons Investigator Grant \#376203 from the Simons Foundation and by the NSF grant DMS-1901985.

\section{$p$-adic $L$-functions}
\addtocontents{toc}{\protect\setcounter{tocdepth}{2}}
\label{sec:Lp}

Let $E/\Q$ be an elliptic curve of conductor $N$, and let $p\nmid 2N$ be a prime of good ordinary reduction for $E$. Fix embeddings $\iota_p:\overline{\Q}\hookrightarrow\overline{\Q}_p$ and $\iota_\infty:\overline{\Q}\hookrightarrow\C$. Let $K$ be an imaginary quadratic field of discriminant $D_K<0$ prime to $N$, and assume that
\begin{equation}\label{eq:spl}
\textrm{$(p)=v\bar{v}$ splits in $K$,}\tag{spl}
\end{equation}
with $v$ the prime above $p$ induced by $\iota_p$. 
We denote by $\Gamma_{\Q}$  (resp. $\Gamma_K$) the Galois group of the cyclotomic $\Z_p$-extension $\Q_{\infty}/\Q$ (resp. the $\Z_p^2$-extension $K_{\infty}/K$). Since $p>2$ the action of complex conjugation gives an eigenspace decomposition 
\[
\Gamma_K=\Gamma_K^+\times \Gamma_K^-.
\] 
Note that $\Gamma_K^+$ (resp. $\Gamma_K^-$) is identified with the Galois group of the cyclotomic $\Z_p$-extension $K_\infty^+/K$ (resp. the anticyclotomic $\Z_p$-extension of $K^-_{\infty}/K$), and 
hence $\Gamma^+_K$ is naturally identified with $\Gamma_\Q$. Let 
\[
\Lambda_\Q=\Z_p\llbracket{\Gamma_\Q}\rrbracket,\quad
\Lambda_K=\Z_p\llbracket{\Gamma_K}\rrbracket,\quad\Lambda_K^\pm=\Z_p\llbracket\Gamma_K^\pm\rrbracket
\]
be the corresponding Iwasawa algebras, so in particular $\Lambda_K^+$ is naturally identified with $\Lambda_\Q$.

\subsection{Cyclotomic $p$-adic $L$-function}
\label{subsec:L-cyc} 

Fix a modular parametrization $\pi_E:X_0(N)\rightarrow E$, and denote by $c_E\in\Z$ the corresponding Manin constant, so that 
\[
\pi_E^*(\omega_E)=c_E\cdot 2\pi i f(\tau)d\tau,
\] 
where $\omega_E$ is a minimal differential on $E$ and $f=\sum_{n=1}^\infty a_nq^n\in S_2(\Gamma_0(N))$ is the  newform associated with $E$.  Pick generators $\delta^\pm$ of $H_1(E,\Z)^\pm$, and define the N\'{e}ron periods $\Omega_E^\pm$ by 
\[
\Omega_E^\pm=\int_{\delta^\pm}\omega_E.
\]
We normalize the $\delta^\pm$
so that $\Omega_E^+\in\R_{>0}$ and $\Omega_E^-\in i\R_{>0}$.

The Fourier coefficient $a_p$ is a $p$-adic unit by hypothesis, and we let $\alpha_p$ be the $p$-adic unit root of $x^2-a_px+p$. 

\begin{thm}
\label{thm:MSD}
There exists an element $\Lcal^{\rm MSD}_p(E/\Q)\in\Lambda_\Q$ such that for any finite order character $\chi$ of $\Gamma_\Q$ of conductor $p^r$ with $r>0$, we have
\[
\Lcal^{\rm MSD}_p(E/\Q)(\chi)=\frac{p^r}{\tau(\overline{\chi})\alpha_p^r}\cdot\frac{L(E,\overline{\chi},1)}{\Omega_E^+},
\]
where $\tau(\overline{\chi})=\sum_{a\;{\rm mod}\;p^r}\overline{\chi}(a)e^{2\pi ia/p^r}$ is the Gauss sum, and
\[
\Lcal_p^{\rm MSD}(E/\Q)(1)=\bigl(1-\alpha_p^{-1}\bigr)^2\cdot\frac{ L(E,1)}{\Omega_E^+}.
\]
\end{thm}

\begin{proof}
The construction of $\Lcal_p^{\rm MSD}(E/\Q)$ as an element in $\Lambda_{\Q}\otimes\Q_p$ with the stated interpolation property is given in \cite[\S{9}]{M-SwD}. 
The integrality of $\Lcal_p^{\rm MSD}(E/\Q)$ is shown in \cite[Prop.~3.7]{greenvats} in the case where $E[p]$ is irreducible as a $G_\Q$-module, and in \cite[Cor.~18]{wuthrich-int} in the reducible case. Since the latter integrality result and some of the ingredients in the proof will be important later, we briefly outline the argument. First, W\"uthrich shows the existence of an elliptic curve $E_\bullet/\Q$ in the isogeny class of $E$ with $\mathcal{L}_p^{\rm MSD}(E_\bullet/\Q)\in\Lambda_{\Q}$ and whose $p$-adic Tate module satisfies 
\[
T_pE_{\bullet}\simeq V_{\Z_p}(f)(1),
\] 
where $V_{\Z_p}(f)$ is the geometric lattice in the $p$-adic Galois representation $V_{\Q_p}(f)$ associated to $f$ considered in \cite[\S{8.3}]{kato-euler-systems}. Building on the theorem of Ferrero--Washington \cite{FW}, Kato's divisibility in the Iwasawa main conjecture for $E_\bullet$ ``without zeta functions''  \cite[Thm.~12.5(4)]{kato-euler-systems} is shown to be integral, from which the integrality of $\mathcal{L}_p^{\rm MSD}(E_\bullet/\Q)$ and of $\mathcal{L}_p^{\rm MSD}(E/\Q)$ itself follows by global duality and the invariance of Mazur's main conjecture under isogenies (see Prop. \ref{prop:isog-inv}).
\end{proof}


\subsection{Two-variable $p$-adic $L$-function, I}
\label{subsec:L-2var-I} 

In this section we recall the $p$-adic $L$-function constructed in \cite{PR-JLMS} following Hida's $p$-adic Rankin method \cite{hida-measure-I}, and explain the relation with the $p$-adic $L$-function of Mazur--Swinnerton-Dyer. 

Let $\Sigma$ be the set of infinity types of Hecke characters $\psi$ of $K$ for which $s=1$ is a critical value of the Rankin $L$-function $L(f/K,\psi,s)$. Following the notations and conventions in \cite[\S{6.1}]{LLZ-K}, the set $\Sigma$ decomposes as 
\[
\Sigma=\Sigma^{(1)}\cup\Sigma^{(2)}\cup\Sigma^{(2')},
\]
where $\Sigma^{(1)}=\{(0,0)\}$ (i.e., corresponding to characters $\psi$ of finite order), $\Sigma^{(2)}=\{(a,b)\colon a\leq -1, b\geq 1\}$, and $\Sigma^{(2')}=\{(b,a)\colon a\leq -1, b\geq 1\}$. Note that the regions $\Sigma^{(2)}$ and $\Sigma^{(2')}$ are interchanged by the involution $\psi\mapsto\psi^\tau$, where $\psi^\tau$ denotes the composition of $\psi$ with the non-trivial automorphism $\tau\in{\rm Gal}(K/\Q)$.

For $\psi$ a finite order Hecke character of $K$ of conductor $\mathfrak{c}$, denote by $\theta(\psi)$ the associated theta series, which is an eigenform of weight one and level $M=D_K\mathbf{N}(\mathfrak{c})$. As in \cite[p.\,457]{PR-GZ}, define the ``Artin root number'' of $\psi$ to be the complex number $W(\psi)$ with $\vert W(\psi)\vert=1$ given by
\[
\theta(\psi)\vert_1\left(\begin{smallmatrix}&-1\\M&\end{smallmatrix}\right)=-iW(\psi)\cdot\theta(\psi);
\]
then the completed $L$-function $\Lambda(\psi,s)=M^{s/2}(2\pi)^{-s}\Gamma(s)\sum_{\mathfrak{a}}\psi(\mathfrak{a})\mathbf{N}(\mathfrak{a})^{-s}$ satisfies the functional equation $\Lambda(\psi,s)=W(\psi)\Lambda(\overline{\psi},1-s)$. As before, let $f\in S_2(\Gamma_0(N))$ be the newform associated with $E$, and put
\[
H_p(f)=\Bigl(1-\frac{p}{\alpha_p^2}\Bigr)\Bigl(1-\frac{1}{\alpha_p^2}\Bigr).
\] 
Denote by $c_f\in\Z_p$ the \emph{congruence number} of $f$, as defined in \cite[\S{7}]{hida-congruences,ribet-modp}; this is such that $c_f^{-1}$ is a $\Z_p$-generator of the congruence ideal $I_f\subset\Q_p$ defined in $\S\ref{subsec:ERL}$.  

\begin{thm}
\label{thm:PR}
There exists an element $\Lcal_p(f/K,\Sigma^{(1)})\in c_f^{-1}\Lambda_K$ such that for every finite order character $\psi$ of $\Gamma_K$ of conductor $v^m\bar{v}^n$ with $m+n>0$, we have
\[
\Lcal_p(f/K,\Sigma^{(1)})(\psi)=\frac{W(\psi)p^{(m+n)/2}}{\alpha_p^{m+n}H_p(f)}\cdot\frac{L(f/K,\overline{\psi},1)}{8\pi^2\langle f,f\rangle_N},
\]
where $W(\psi)$ is the Artin root number of $\psi$, and
\[
\Lcal_p(f/K,\Sigma^{(1)})(1)=
\frac{(1-\alpha_p^{-1})^{4}}{H_p(f)}
\cdot\frac{L(f/K,1)}{8\pi^2\langle f,f\rangle_N},
\]
where $\langle f,g\rangle_N=\int_{\Gamma_0(N)\backslash\mathcal{H}}\overline{f(\tau)}g(\tau)dxdy$ is the Petersson inner product on $S_2(\Gamma_0(N))$.
\end{thm}

\begin{proof}
This is a reformulation of \cite[Thm.~1.1]{PR-GZ}. For our later use, we note that this is the same as the two-variable $p$-adic $L$-function 
\[
L_p(f,\boldsymbol{g})\in c_f^{-1}\Z_p\widehat\otimes\Lambda_{\boldsymbol{g}}\llbracket\Gamma_\Q\rrbracket
\]
obtained by specializing to $f$ the three-variable $p$-adic Rankin $L$-series in \cite[Thm.~7.7.2]{explicit}, where $\boldsymbol{g}$ is the CM Hida family introduced in the proof of Lemma~\ref{lem:integral} below. 
\end{proof}

\begin{defi}[``Perrin-Riou's $p$-adic $L$-function'']\label{def:PR}
Let $f\in S_2(\Gamma_0(N))$ be the newform associated to $E$, and let $\pi_E:X_0(N)\rightarrow E$ be a modular parametrization. Then we
put
\[
\Lcal_p^{\rm PR}(E/K):=\frac{\deg(\pi_E)}{c_E^2}\cdot H_p(f)\cdot \Lcal_p(f/K,\Sigma^{(1)}),\nonumber
\]
where $c_E$ is the Manin constant associated to $\pi_E$.
\end{defi}


\begin{rem}\label{rem:periods}
Note that the factor $H_p(f)$ can be interpreted as an Euler factor coming from the adjoint $L$-function of $f$ (see \cite[Rem.~6.5.10]{KLZ-AJM}). Moreover, setting 
\[
\Omega_{E/K}:=\frac{1}{\sqrt{\vert D_K\vert}}\int_{E(\C)}\omega_E\wedge i\,\overline{\omega_E}
\]
we find
\[
8\pi^2\langle f,f\rangle_N=\int_{X_0(N)}\omega_f\wedge i\,\overline{\omega_f}
=\frac{\deg(\pi_E)}{c_E^2}\cdot\Omega_{E/K},
\]
and so Theorem~\ref{thm:PR} says that $\Lcal_p^{\rm PR}(E/K)$ interpolates the finite order twists of $L(f/K,1)$ normalized by the complex period $\Omega_{E/K}$.
\end{rem}

Denote by $\Lcal^{\rm PR}_p(E/K)^+\in\Lambda_\Q$ the image of $\Lcal^{\rm PR}_p(E/K)$ under the map induced by the projection $\Gamma_K\twoheadrightarrow\Gamma_K^+\simeq\Gamma_\Q$.

\begin{prop}
\label{prop:comp-Lcyc}
Up to a unit in $\Lambda_\Q^\times$, we have 
\[
\Lcal_p^{\rm PR}(E/K)^+=\Lcal_p^{\rm MSD}(E/\Q)\cdot\Lcal^{\rm MSD}_p(E^K/\Q),
\] 
where $E^K$ is the twist of $E$ by the quadratic character corresponding to $K/\Q$.
\end{prop}

\begin{proof}
Let $\chi$ be a Dirichlet character of conductor $p^r$. Using the standard formula 
\[
W(\chi\circ\mathbf{N})=\chi(\vert D_K\vert)\frac{\tau(\chi)^2}{p^r}=\chi(\vert D_K\vert)\frac{p^r}{\tau(\overline{\chi})^2}
\]
(see e.g. \cite[Lem.~4.8.1]{miyake}) and the factorization
\[
L(f/K,\overline{\chi}\circ\mathbf{N},1)=L(E,\overline{\chi},1)\cdot L(E^K,\overline{\chi},1),
\]
from Theorem~\ref{thm:MSD} and Theorem~\ref{thm:PR} we find
\[
\begin{aligned}
\Lcal_p(f/K,\Sigma^{(1)})(\chi\circ\mathbf{N})\cdot H_p(f)&=W(\chi\circ\mathbf{N})\cdot\frac{p^r}{\alpha_p^{2r}}\cdot L(f/K,\overline{\chi}\circ\mathbf{N},1)\\
&=\chi(\vert D_K\vert)\cdot\frac{p^{2r}}{\tau(\overline{\chi})\alpha_p^{2r}}\cdot L(E,\overline{\chi},1)\cdot L(E^K,\overline{\chi},1)\\
&=\chi(\vert D_K\vert)\cdot\Lcal_p^{\rm MSD}(E/\Q)(\chi)\cdot\Lcal_p^{\rm MSD}(E^K/\Q)(\chi)\cdot\frac{\Omega_E^+\cdot\Omega_{E^K}^+}{8\pi^2\langle f,f\rangle_N}.
\end{aligned}
\]
Since the complex period $\Omega_{E/K}$ in Remark~\ref{rem:periods} satisfies
\[
\Omega_{E/K}=[E(\R):E^0(\R)]\cdot\Omega_E^+\cdot\Omega_{E^K}^+
\]
(see \cite[\S{V.2}]{grosszagier}), the result in now clear from the definition of $\Lcal_p^{\rm PR}(E/K)^+$.
%
\end{proof}

\subsection{Anticyclotomic $p$-adic $L$-function}
\label{subsec:L-ac} 

We next recall the ``square-root'' $p$-adic $L$-function of Bertolini--Darmon--Prasanna \cite{BDP}, explicitly shown to be a $p$-adic measure in \cite{cas-hsieh1}. Assume that 
\begin{equation}\label{eq:Heeg}
\textrm{every prime $\ell\vert N$ splits in $K$,}\tag{Heeg}
\end{equation}
and fix an integral ideal $\mathfrak{N}\subset\mathcal{O}_K$ with $\mathcal{O}_K/\mathfrak{N}=\Z/N\Z$.  For simplicity, we also assume that
\begin{equation}\label{eq:disc}
\textrm{the discriminant $D_K<0$ is odd and $D_K\neq -3$}.\tag{disc}
\end{equation}

In the following we let $\Omega_p$ and $\Omega_K$ be CM periods attached to $K$ as in \cite[\S{2.5}]{cas-hsieh1} and put  $\Lambda_{K}^{-,\rm ur}=\Lambda_K\widehat\otimes\Z_p^{\rm ur}$, where $\Z_p^{\rm ur}$ is the completion of the ring of integers of the maximal unramified extension of $\Q_p$.

\begin{thm}
\label{thm:BDP}
There exists an element $\Lcal_p^{\rm BDP}(f/K)\in\Lambda_K^{-,{\rm ur}}$ characterized by the following interpolation property: for every character $\xi$ of $\Gamma_K^-$ crystalline at both $v$ and $\bar{v}$ and corresponding to a Hecke character of $K$ of infinity type $(n,-n)$ with $n\in\Z_{>0}$ and $n\equiv 0\pmod{p-1}$, we have
\[
\Lcal_p^{\rm BDP}(f/K)(\xi)=\frac{\Omega_p^{4n}}{\Omega_K^{4n}}\cdot\frac{\Gamma(n)\Gamma(n+1)\xi(\mathfrak{N}^{-1})}{4(2\pi)^{2n+1}\sqrt{D_K}^{2n-1}}\cdot\bigl(1-a_p\xi(\bar{v})p^{-1}+\xi(\bar{v})^2p^{-1}\bigr)^2\cdot L(f/K,\xi,1).
\]  
Moreover, $\Lcal_p^{\rm BDP}(f/K)$ is a nonzero element of $\Lambda_K^{-,{\rm ur}}$.
\end{thm}

\begin{proof}
See \cite[Thm.~2.1.1]{eisenstein} for the construction, 
which is deduced from \cite[\S{3}]{cas-hsieh1}. Since (\ref{eq:Heeg}) implies that $f$ does not have CM by $K$, the nonvanishing of $\Lcal_p^{\rm BDP}(f/K)$ follows from \cite[Thm.~3.9]{cas-hsieh1}.
\end{proof}

\begin{rmk}\label{rem:periods}
The CM period $\Omega_{K}\in\C^\times$ in Theorem~\ref{thm:BDP} agrees with that in \cite[(5.1.16)]{BDP}, but is \emph{different} from the period $\Omega_\infty$ defined in \cite[p.\,66]{deshalit} and \cite[(4.4b)]{hidatilouineI}. In fact, one has 
\[
\Omega_{\infty}=2\pi i\cdot\Omega_K.
\]
In terms of $\Omega_\infty$, the interpolation formula in Theorem~\ref{thm:BDP} reads
\[
\Lcal_{p}^{\rm BDP}(f/K)(\xi)=\frac{\Omega_p^{4n}}{\Omega_\infty^{4n}}\cdot\frac{\Gamma(n)\Gamma(n+1)\xi(\mathfrak{N}^{-1})}{4(2\pi)^{1-2n}\sqrt{D_K}^{2n-1}}\cdot\bigl(1-a_p\xi(\overline{v})p^{-1}+\xi(\bar{v})^2p^{-1}\bigr)^2\cdot L(f/K,\xi,1).
\]
This is the form of the interpolation that we shall use later.
\end{rmk}

\subsection{Two-variable $p$-adic $L$-function, II}
\label{subsec:L-2var-II} 

The main result of this section is Proposition~\ref{prop:comp-Lac}, relating the 
$p$-adic $L$-function $\Lcal_p^{\rm BDP}(f/K)$ of Theorem~\ref{thm:BDP} to the anticyclotomic projection of the following two-variable $p$-adic Rankin $L$-series.

\begin{thm}\label{thm:Hida-2}
\label{thm:Gr}
There exists an element $\Lcal_p(f/K,\Sigma^{(2')})\in{\rm Frac}\,\Lambda_K$ such that for every character $\xi$ of $\Gamma_K$ 
crystalline at both $v$ and $\bar{v}$, and of infinity type $(b,a)$ with $a\leq -1$ and $b\geq 1$, we have
\[
\Lcal_p(f/K,\Sigma^{(2')})(\psi)=\frac{2^{a-b}i^{b-a-1}\Gamma(b+1)\Gamma(b)N^{a+b+1}}{(2\pi)^{2b+1}\langle\theta_{\psi_b},\theta_{\psi_b}\rangle_N}\cdot\frac{\mathcal{E}(\psi,f,1)}{(1-\psi^{1-\tau}(\bar{v}))(1-p^{-1}\psi^{1-\tau}(\bar{v}))}\cdot L(f/K,\psi,1),
\]
where $\theta_{\psi_b}$ is the theta series of weight $b-a+1\geq 3$ associated to the Hecke character $\psi_b=\psi\vert\;\vert^{-b}$ of $\infty$-type $(0,a-b)$, and 
\[
\mathcal{E}(\psi,f,1)=(1-p^{-1}\psi(\bar{v})\alpha_p)(1-\psi(\bar{v})\alpha^{-1}_p)(1-\psi^{-1}({v})\alpha_p^{-1})(1-p^{-1}\psi^{-1}({v})\alpha_p).
\]
\end{thm}

\begin{proof}
This is another instance of Hida's $p$-adic Rankin $L$-series, as explained in \cite[Thm.~6.1.3]{LLZ-K} (note, however, that we have reversed the roles of $v$ and $\bar{v}$ with respect to \emph{loc.\,cit.}).
\end{proof}

We also need to recall the interpolation property of the Katz $p$-adic $L$-functions \cite{katz}, following the formulation in \cite{deshalit}. Put $\Lambda_K^{\rm ur}=\Lambda_K\widehat\otimes_{\Z_p}\Z_p^{\rm ur}$.

\begin{thm}\label{thm:katz}
There exists an element $\Lcal_{\bar{v}}(K)\in\Lambda_K^{\rm ur}$ such that for every character $\xi$ of $\Gamma_K$ of infinity type $(k,j)$ with $0\leq-j<k$ satisfies 
\[
\Lcal_{\bar{v}}(K)(\xi)=\frac{\Omega_p^{k-j}}{\Omega_\infty^{k-j}}\cdot\Gamma(k)\cdot\biggl(\frac{\sqrt{D_K}}{2\pi}\biggr)^j\cdot(1-\xi^{-1}(\bar{v})p^{-1})(1-\xi({v}))\cdot L(\xi,0).
\]
Similarly, there exists an element $\Lcal_v(K)\in\Lambda_K^{\rm ur}$ such that for every character $\xi$ of $\Gamma_K$ of infinity type $(j,k)$ with $0\leq-j<k$, we have 
\[
\Lcal_{v}(K)(\xi)=\frac{\Omega_p^{k-j}}{\Omega_\infty^{k-j}}\cdot\Gamma(k)\cdot\biggl(\frac{\sqrt{D_K}}{2\pi}\biggr)^j\cdot(1-\xi^{-1}({v})p^{-1})(1-\xi(\bar{v}))\cdot L(\xi,0).
\]
Moreover, we have the functional equation 
\[
\Lcal_{\bar{v}}(\xi)=\Lcal_{v}(\xi^{-1}\mathbf{N}^{-1}),
\] 
where the equality is up to a $p$-adic unit.
\end{thm}

\begin{proof}
This is \cite[Thm.~II.4.14]{deshalit}, with $\Lcal_{\bar{v}}(K)$ (resp. $\Lcal_{{v}}(K)$) corresponding to the measure $\mu(\bar{v}^\infty)$ (resp. $\mu(v^\infty)$) in \emph{loc.\,cit.}. On the other hand, as explained in \cite[Lem.~3.3.2(b)]{7author}, the stated functional equation is a reformulation of \cite[Thm.~II.6.4]{deshalit}.
\end{proof}

\begin{defi}[``Greenberg's $p$-adic $L$-function'']\label{def:Gr}
Put
\[
\Lcal_p^{\rm Gr}(f/K):=h_K\cdot\Lcal_v(K)^-\cdot\Lcal_p(f/K,\Sigma^{(2')}),
\]
where $h_K$ is the class number of $K$ and $\Lcal_v(K)^-$ the image of $\Lcal_v(K)$ under the map $\Lambda_K^{\rm ur}\rightarrow\Lambda_K^{\rm ur}$ given by $\gamma\mapsto\gamma^{1-\tau}$ for $\gamma\in\Gamma_K$. 
\end{defi}

Note that \emph{a priori} we have $\Lcal_p^{\rm Gr}(f/K)\in{\rm Frac}\,\Lambda_K^{\rm ur}$. 

\begin{lemma}\label{lem:integral}
The $p$-adic $L$-function $\Lcal_p^{\rm Gr}(f/K)$ is integral, i.e., $\Lcal_p^{\rm Gr}(f/K)\in\Lambda_K^{\rm ur}$.
\end{lemma}

\begin{proof}
Denote by $\eta=\eta_{K/\Q}$ the quadratic character corresponding to $K/\Q$, and let ${\rm Eis}_{1,\eta}(q)$ be the weight one Eisenstein series
\[
{\rm Eis}_{1,\eta}(q)=\sum_{n\geq 1}q^n\sum_{d\vert n}\eta_{}(d).
\]
Since $p$ splits in $K$, ${\rm Eis}_{1,\eta}(q)$ is $p$-irregular. Letting $g$ denote the unique $p$-stabilization of ${\rm Eis}_{1,\eta}$, by \cite[Thm.~A(i)]{Betina-Dimitrov-Pozzi} there is a unique cuspidal Hida family $\boldsymbol{g}$ passing through $g$. Moreover, $\boldsymbol{g}$ is of CM type, given explicitly as the $q$-series
\[
\boldsymbol{g}=\sum_{(\mathfrak{a},\bar{v})=1}[\mathfrak{a}]q^{\mathbf{N}(\mathfrak{a})}\in\Lambda_{\boldsymbol{g}}\llbracket q\rrbracket,
\]
where $\Lambda_{\boldsymbol{g}}=\Z_p\llbracket\Gamma_v\rrbracket$, for $\Gamma_v$ the Galois group of the maximal $\Z_p$-extension inside $K_\infty/K$ unramified at $v$, and where $[\mathfrak{a}]$ denotes the natural image of the ideal $\mathfrak{a}\subset\mathcal{O}_K$ in $\Gamma_v$ under the Artin reciprocity map. 

Now, the $p$-adic $L$-function $\mathcal{L}_p(f/K,\Sigma^{(2')})$ in Theorem~\ref{thm:Hida-2} arises from Hida's $p$-adic Rankin $L$-series 
\[
L_p(\boldsymbol{g},f)\in I_{\boldsymbol{g}}^{\rm cusp}\widehat\otimes_{\Z_p}\Z_p\llbracket\Gamma_\Q\rrbracket,
\] 
where $I_{\boldsymbol{g}}^{\rm cusp}\subset{\rm Frac}\,\Lambda_{\boldsymbol{g}}$ is the cuspidal congruence ideal of $\boldsymbol{g}$ (see $\S\ref{subsec:ERL}$ for the definition).  
Thus if $H_{\bg}^{\rm cusp}\in\Lambda_{\boldsymbol{g}}$ denotes a characteristic power series for the denominator of  $I_{\boldsymbol{g}}$, then the product $H_{\bg}^{\rm cusp}\cdot\Lcal_p(f/K,\Sigma^{(2)})$ is integral, so it suffices to show that $h_K\cdot\Lcal_v(K)^-$ is divisible by $H_{\bg}^{\rm cusp}$. 

By \cite[Thm.~0.3]{hidatilouineII} and Rubin's proof of the Iwasawa main conjecture for $K$, \cite{rubinmainconj},  one has that such divisibility holds up to powers of the augmentation ideal $(\gamma_v-1)\subset\Z_p\llbracket\Gamma_v\rrbracket$; since by \cite[Thm.~A(i)]{Betina-Dimitrov-Pozzi} one knows that $H_{\bg}^{\rm cusp}$ is not divisible by $\gamma_v-1$, the result follows.
\end{proof}

Denote by $\Lcal_p^{\rm Gr}(f/K)^-$ the image of $\Lcal_p^{\rm Gr}(f/K)$ under the natural projection $\Lambda_K^{\rm ur}\rightarrow\Lambda_K^{-,{\rm ur}}$.

\begin{prop}\label{prop:factor-L-Gr}
\label{prop:comp-Lac}
We have the equality
\[
\Lcal_p^{\rm Gr}(f/K)^-\cdot \Lambda_K^{-,\rm{ur}}=\Lcal_p^{\rm BDP}(f/K)\cdot \Lambda_K^{-,\rm{ur}}.
\]
\end{prop}

\begin{proof}
This follows from a direct comparison of the interpolation formulas in Theorem~\ref{thm:Gr}, Theorem~\ref{thm:BDP} and Theorem~\ref{thm:katz}, together with an application of Dirichlet's class number formula (cf. \cite[Thm.~1.7]{castellaheights}). 

Indeed, let $\xi$ be a Hecke character of infinity type $(n,-n)$, $n\in\Z_{\geq 0}$, as in the statement of Theorem~\ref{thm:BDP}. Then the character $\xi^{1-\tau}\mathbf{N}^{-1}$, of infinity type $(2n+1,1-2n)$, is in the range of interpolation of $\Lcal_{\bar{v}}(K)$, and noting that $L(\xi^{1-\tau}\mathbf{N}^{-1},0)=L(\xi^{1-\tau},1)$, by Theorem~\ref{thm:katz} we have
\begin{equation}\label{eq:ev-Katz}
\Lcal_{\bar{v}}(K)(\xi^{1-\tau}\mathbf{N}^{-1})=\frac{\Omega_p^{4n}}{\Omega_\infty^{4n}}\cdot\Gamma(2n+1)\cdot\biggl(\frac{2\pi}{\sqrt{D_K}}\biggr)^{2n-1}\cdot(1-\xi^{1-\tau}(\bar{v}))(1-\xi^{1-\tau}({v})p^{-1})\cdot L(\xi^{1-\tau},1).
\end{equation}
On the other hand, from Hida's formula for the adjoint $L$-value (see \cite[Thm~.7.1]{hidatilouineI}) and Dirichlet's class number formula we obtain 
\begin{equation}\label{eq:classnumber}
\langle \theta_{\xi_n},\theta_{\xi_n}\rangle_N\;\sim_p\;
\Gamma(2n+1)\cdot\frac{1}{2^{4n-1}\pi^{2n+1}}\cdot h_K\cdot L(\xi^{1-\tau},1),
\end{equation}
where $\sim_p$ denotes equality up to $p$-adic unit independent of $n$, and similarly as in Theorem~\ref{thm:Hida-2}, $\xi_n$ is the theta series of weight $2n+1\geq 3$ associated to the Hecke character $\xi_n=\xi\vert\;\vert^{-n}$ of infinity type $(0,-2n)$. Combining (\ref{eq:ev-Katz}), (\ref{eq:classnumber}) and the functional equation in Theorem~\ref{thm:katz} this gives
\[
h_K\cdot\Lcal_{v}(K)(\xi^{1-\tau})\;\sim_p\;\frac{\Omega_p^{4n}}{\Omega_\infty^{4n}}\cdot\biggl(\frac{2\pi}{\sqrt{D_K}}\biggr)^{2n-1}\cdot(1-\xi^{1-\tau}(\bar{v}))(1-\xi^{1-\tau}({v})p^{-1})\cdot 2^{4n-1}\pi^{2n+1}\cdot \langle\theta_{\xi_n},\theta_{\xi_n}\rangle_N.
\]
Noting that the $p$-Euler factor $\mathcal{E}(\psi,f,1)$ in Theorem~\ref{thm:Hida-2} satisfies
\[
\mathcal{E}(\xi,f,1)=\bigl(1-a_p\xi(\overline{v})p^{-1}+\xi(\bar{v})^2p^{-1}\bigr)^2,
\]
from Definition~\ref{def:Gr}, Theorem~\ref{thm:Hida-2}, and Theorem~\ref{thm:BDP} we thus find that
\[
\Lcal_p^{\rm Gr}(f/K)(\xi)\;\sim_p\;\xi(\mathfrak{N})\cdot 2^{3n-2}i^{2n-1}\cdot\Lcal_p^{\rm BDP}(f/K)(\xi).
\]
Since $\xi(\mathfrak{N})\cdot 2^{3n-2}i^{2n-1}$ is interpolated by a unit in $\Lambda_K^{-,\rm{ur}}$, this completes the proof.
\end{proof}

\subsection{Twists and imprimitive $p$-adic $L$-functions}\label{subsec:tw-L}


Let $\alpha:\Gamma_K\rightarrow R^\times$ be a character with values in the ring of integers $R$ of a finite extension $\Phi/\Q_p$ with uniformiser $\varpi\in R$. Let $\Lambda_{K,R} = R\widehat\otimes_{\Z_p}\Lambda_K = R[\![\Gamma_K]\!]$, and define
\[
{\rm Tw}_\alpha:\Lambda_{K,R}\rightarrow\Lambda_{K,R}
\]
to be the $R$-linear isomorphism given by $\gamma\mapsto\alpha(\gamma)\gamma$ for $\gamma\in\Gamma_K$. Denote by 
$\Lcal_p^{\rm PR}(E(\alpha)/K), \Lcal_p^{\rm Gr}(f(\alpha)/K)$ the image of $\Lcal_p^{\rm PR}(E/K), \Lcal_p^{\rm Gr}(f/K)$, respectively, under ${\rm Tw}_\alpha$. 

\begin{lemma}\label{lem:cong-L}
Suppose $\alpha\equiv 1\pmod{\varpi^m}$. Then 
\[
\Lcal_p^{\rm PR}(E(\alpha)/K)^+\equiv\Lcal_p^{\rm PR}(E/K)^+\;({\rm mod}\,\varpi^m).
\]
\end{lemma}

\begin{proof}
This is clear from the definitions.
%
\end{proof}


For a prime $w$ split in $K$, lying over the rational prime $\ell\neq p$, we let $\Gamma_w^\pm$ be the corresponding decomposition group in $\Gamma_K^\pm$, and $\gamma_w^\pm\in\Gamma_w^\pm$ be the image of an arithmetic Frobenius ${\rm Frob}_w$ under the projection $G_K\rightarrow\Gamma_K^\pm$. 

\begin{defi}\label{def:w-Euler}
Put 
\[
\mathcal{P}^\pm_w(\alpha):=P_w(\ell^{-1}\gamma_w^\pm)\in\Lambda_{K,R}^{\pm},
\]
where $P_w(X)=\det(1-{\rm Frob}_wX\,\vert\,V(\alpha)_{I_w})$ is the Euler factor at $w$ of the $L$-function of $V(\alpha)=T_pE(\alpha)\otimes\Q_p$. For $S'$ a finite set of primes $w$ as above, define
\begin{align*}
\Lcal_{p}^{\rm PR}(E(\alpha)/K)^{+,S'}&:=\Lcal_{p}^{\rm PR}(E(\alpha)/K)^+\cdot\prod_{w\in S'}\mathcal{P}_w^+(\alpha),\\\Lcal_p^{\rm Gr}(f(\alpha)/K)^{\pm,S'}&:=\Lcal_p^{\rm Gr}(f(\alpha)/K)^\pm\cdot\prod_{w\in S'}\mathcal{P}^\pm_w(\alpha),
\end{align*}
and similarly $\Lcal_p^{\rm BDP}(f(\alpha)/K)^{S'}:=\Lcal_p^{\rm BDP}(f(\alpha)/K)\cdot\prod_{w\in S'}\mathcal{P}^-_w(\alpha)$.
\end{defi}

Of course, the results of Proposition~\ref{prop:comp-Lcyc} and Proposition~\ref{prop:comp-Lac} directly extend to their analogues for these $S'$-imprimitive $p$-adic $L$-functions.

\section{Selmer groups}\label{sec:Sel}

In this section, we let $E/\Q$ be an elliptic curve of conductor $N$, $p$ an odd  prime of good ordinary reduction for $E$, and $K$ an imaginary quadratic field satisfying (\ref{eq:Heeg}) and (\ref{eq:spl}). 

\subsection{Selmer structures}\label{subsec:Sel-str}

Let $\Sigma$ be a finite set of places of $\Q$ containing the prime $p$, $\infty$, and the prime factors of $N$. We assume throughout that 
\[
\textrm{all finite primes in $\Sigma$ split in $K$.}
\]
With a slight abuse of notation, we 
also write $\Sigma$ 
for the set of places of $K$ lying above the places in $\Sigma$. 

\subsubsection{Discrete coefficients}\label{subsec:discrete-Sel}

For a discrete $\Z_p$-module $M$, we let
\[
M^\vee={\rm Hom}_{\rm cts}(M,\Q_p/\Z_p)
\]
be the Pontryagin dual. The module $\Lambda_{\Q}^{\vee}$ is equipped with a $G_\Q$-action via $\Psi^{-1}$, where $\Psi:G_\Q\rightarrow\Lambda_{\Q}^\times$ is the character arising from the projection $G_\Q\twoheadrightarrow\Gamma_\Q$. Similarly, $\Lambda_K^{\vee}$ and $(\Lambda^\pm_K)^{\vee}$ are equipped with $G_K$-actions.

\begin{defi}\label{defilocalconds}
Let $F$ be $\Q$ or $K$ and $w$ a prime above $p$. For $\Lambda$ any of the Iwasawa algebras $\Lambda_\Q$, $\Lambda_K$, or $\Lambda_K^\pm$, 
we put
\begin{align*}
\rH^1_{\rm rel}(F_w,T_pE\otimes \Lambda^{\vee})&=\rH^1(F_w,T_pE\otimes \Lambda^{\vee}),\\
\rH^1_{\ord}(F_w,T_pE\otimes \Lambda^{\vee})&= {\rm im}\bigl\{\rH^1(F_w,{\rm Fil}_w^+(T_pE)\otimes \Lambda^{\vee})\rightarrow\rH^1(F_w,T_pE\otimes \Lambda^{\vee})\bigr\},\\
\rH^1_{\rm str}(F_w,T_pE\otimes \Lambda^{\vee})&=\{0\}, 
\end{align*}
where ${\rm Fil}_w^+(T_pE):={\rm ker}\bigl\{T_pE\rightarrow T_p\tilde{E}\bigr\}$ with $\tilde{E}$ the reduction of $E$ at $w$.
\end{defi}
Let $G_{\Q,\Sigma}$ and $G_{K,\Sigma}$ denote the Galois group of the maximal extension of $\Q$ and $K$ respectively unramified outside $\Sigma$. 
For $\bullet \in \{\ord, \rm str, \rm rel \}$ and $M=T_pE\otimes \Lambda_\Q^{\vee}$, we define the Selmer group
\begin{equation}
\rH^1_{\Fcal_{\bullet}}(\Q, M)=\ker \biggl(\rH^1(G_{\Q,\Sigma},M) \to \prod_{w\in\Sigma,w\nmid p}\rH^1(\Q_w, M)\times\frac{\rH^1(\Q_{p},M)}{\rH^1_{\bullet}(\Q_{p},M)}\biggr).
\end{equation}
Similarly, for $\star,\bullet \in \{\ord, \rm str, \rm rel \}$ and $M=T_pE\otimes\Lambda^{\vee}$, where $\Lambda$ is any of the Iwasawa algebras $\Lambda_K$ or $\Lambda_K^\pm$, we let
\begin{equation}\label{eq:defSelK}
\rH^1_{\Fcal_{\star,\bullet}}(K, M)=\ker \biggl(\rH^1(G_{K,\Sigma},M) \to \prod_{w\in\Sigma,w\nmid p}\rH^1(K_w, M)\times\frac{\rH^1(K_v,M)}{\rH^1_{\star}(K_v,M)}\times\frac{\rH^1(K_{\bar{v}},M)}{\rH^1_{\bullet}(K_{\bar{v}},M)}\biggr).
\end{equation}
To ease notation, we write $\rH^1_{\Fcal_{\ord}}(K, M)=\rH^1_{\Fcal_{\ord,\ord}}(K, M)$ and $\rH^1_{\Fcal_{\rm Gr}}(K, M)=\rH^1_{\Fcal_{\rm rel, str}}(K, M)$, and put
\begin{align*}
\X_{\rm ord}(E/\Q_{\infty})&=\rH^1_{\Fcal_{\ord}}(\Q, T_pE\otimes\Lambda_\Q^{\vee})^{\vee},\\
\X_{\rm ord}(E/K_\infty^\pm)&=\rH^1_{\Fcal_{\ord}}(K, T_pE\otimes (\Lambda^\pm_K)^{\vee})^{\vee},\\
\X_{\rm Gr}(E/K_\infty^\pm)&=\rH^1_{\Fcal_{\rm Gr}}(K, T_pE\otimes (\Lambda^\pm_K)^{\vee})^{\vee}.
\end{align*}
It is a standard fact that these are finitely generated modules over the corresponding Iwasawa algebras. 

\subsubsection{Compact coefficients} 

For $\Lambda$ any of the Iwasawa algebras $\Lambda_\Q$, $\Lambda_K$, or $\Lambda_K^\pm$, consider the compact module
\[
T_pE\widehat\otimes_{\Z_p}\Lambda,
\]
where $\Lambda$ is equipped with a $G_K$-action via $\Psi:G_K\rightarrow\Lambda^\times$. For $\bullet\in\{\rm ord, \rm str, \rm rel\}$ and $w$ a prime of $K$ above $p$, we define the local conditions $\rH^1_{\bullet}(K_w,T_pE\widehat\otimes_{\Z_p}\Lambda)\subset\rH^1(K_w,T_pE\widehat\otimes_{\Z_p}\Lambda)$ similarly as in Definition~\ref{defilocalconds}, and for $\star,\bullet\in\{\rm ord,\rm str,\rm rel\}$ we define the Selmer group $\rH^1_{\Fcal_{*,\bullet}}(K,T_pE\widehat\otimes_{\Z_p}\Lambda)$ by the same recipe as in (\ref{eq:defSelK}). Put
\[
\mathfrak{S}_{\rm ord,\rm rel}(E/K_\infty)=\rH^1_{\Fcal_{\rm ord,\rm rel}}(K,T_pE\widehat\otimes_{\Z_p}\Lambda_K),\quad
\mathfrak{S}_{\rm ord}(E/K_\infty)=\rH^1_{\Fcal_{\rm ord,\rm ord}}(K,T_pE\widehat\otimes_{\Z_p}\Lambda_K),
\]
etc., and similarly for $E/K_\infty^\pm$ with $\Lambda_K$ in place of $\Lambda_K^\pm$, respectively.

\subsection{Iwasawa main conjectures}

The Iwasawa main conjecture for $E$, as formulated by Mazur \cite{mazur-towers}, \cite[\S{9.5}, Conj.~3]{M-SwD}, is the following. 

\begin{conjecture}[Mazur]\label{conj:IMC}
The module $\X_{\rm ord}(E/\Q_{\infty})$ is $\Lambda_\Q$-torsion, with 
\[
\ch_{\Lambda_\Q}\bigl(\X_{\rm ord}(E/\Q_{\infty})\bigr)=\bigl(\Lcal_p^{\rm MSD}(E/\Q)\bigr).
\]
\end{conjecture}

More generally, a vast generalization of Mazur's main conjecture to $p$-adic deformations of motives formulated by Greenberg \cite{greenberg-reps,greenberg-motives},
predicts the following.

\begin{conjecture}[Greenberg]\label{conj:IMC-K}
The modules $\X_{\rm ord}(E/K^+_\infty)$ and $\X_{\rm Gr}(E/K_\infty^-)$ are torsion over $\Lambda_K^+$ and $\Lambda_K^-$ respectively, with
\begin{align*}
\ch_{\Lambda_K^+}\bigl(\X_{\rm ord}(E/K_{\infty}^+)\bigr)&=\bigl(\Lcal_p^{\rm PR}(E/K)^+\bigr),\\
\ch_{\Lambda_K^-}\bigl(\X_{\rm Gr}(E/K_{\infty}^-)\bigr)\Lambda_K^{-,\rm ur}&=\bigl(\Lcal_p^{\rm Gr}(f/K)^-\bigr).
\end{align*}
\end{conjecture}


In this paper we shall prove Mazur's Main Conjecture~\ref{conj:IMC} (in the case where $p$ is a good Eisenstein prime for $E$) by first proving  Conjecture~\ref{conj:IMC-K} for a suitable $K$. 

\subsubsection{Isogeny invariance}

Conjectures~\ref{conj:IMC} and \ref{conj:IMC-K} are known to be invariant under isogenies. This follows from a computation in global duality due to Schneider and Perrin-Riou \cite{schneider-isogenies,perrinriou}.


\begin{prop}\label{prop:isog-inv}
Suppose $E_1/\Q$ and $E_2/\Q$ are isogenous elliptic curves with good ordinary reduction at $p$. Assume that $\mathfrak{X}_{\rm ord}(E_i/K_\infty^\pm)$ is $\Lambda_K^+$-torsion ($i=1,2$), and let $\mathcal{F}_{\rm ord}(E_i/K_\infty^+)$ and $\mathcal{F}_{\rm ord}(E_i/\Q_\infty)$ be characteristic power series for $\mathfrak{X}_{\rm ord}(E_i/K_\infty^+)$ and $\mathfrak{X}_{\rm ord}(E/\Q_\infty)$, respectively. Then we have the equalities up to a $p$-adic unit:
\begin{align*}
\Omega_{E_1}\cdot\Fcal_{\rm ord}(E_1/\Q_\infty)&\;\sim_p\;\Omega_{E_2}\cdot\Fcal_{\rm ord}(E_2/\Q_\infty),\\
\Omega_{E_1/K}\cdot\Fcal_{\rm ord}(E_1/K_\infty^+)&\;\sim_p\;\Omega_{E_2/K}\cdot\Fcal_{\rm ord}(E_2/K_\infty^+).
\end{align*} 
In particular, the main conjectures for $\mathfrak{X}_{\rm ord}(E/\Q_\infty)$ and $\mathfrak{X}_{\rm ord}(E/K_{\infty}^+)$ are both invariant under isogenies.
\end{prop}

\begin{proof}
See \cite[Appendice]{perrinriou}.
\end{proof}

Although not directly needed for our arguments, we also note that the isogeny invariance of the main conjecture for $\mathfrak{X}_{\rm Gr}(E/K_\infty^-)$ similarly follows from the main result of \cite{PR-isogenies} (see \cite[Prop.\,2.9]{KO-iwasawa}).

\subsection{Imprimitive Selmer groups}\label{subsec:imp}

For any subset $S'\subset\Sigma$ consisting of primes away from $p$, we define $S'$-imprimitive versions of the  Selmer groups of $\S\ref{subsec:discrete-Sel}$ by relaxing the local conditions at the primes $w\in S'$, e.g. for $M=T_pE\otimes\Lambda_\Q^\vee$:
\[
\rH^1_{\Fcal_{\rm ord}^{S'}}(\Q,M)=\ker \biggl(\rH^1(G_{\Q,\Sigma},M) \to \prod_{w\in\Sigma\smallsetminus S',w\nmid p}\rH^1(\Q_w, M)\times\frac{\rH^1(\Q_{p},M)}{\rH^1_{\rm ord}(\Q_{p},M)}\biggr).
\] 
%
We denote with a superscript $S'$ the Pontryagin duals of these modules:
\begin{align*}
\X_{\rm ord}^{S'}(E/\Q_{\infty})&=\rH^1_{\Fcal_{\ord}^{S'}}(\Q, T_pE\otimes \Lambda_\Q^{\vee})^{\vee},\\
\X_{\rm ord}^{S'}(E/K^\pm_\infty)&=\rH^1_{\Fcal_{\ord}^{S'}}(K, T_pE\otimes (\Lambda^\pm_K)^{\vee})^{\vee},\\
\X_{\rm Gr}^{S'}(E/K_\infty^\pm)&=\rH^1_{\Fcal_{\rm Gr}^{S'}}(K, T_pE\otimes (\Lambda^\pm_K)^{\vee})^{\vee}.
\end{align*}

The next result will be used to descend from $K$ to $\Q$ (\emph{cf.} Proposition~\ref{prop:comp-Lcyc}). 
As done here, in the following we shall often identify the Iwasawa algebras $\Lambda_K^+$ and $\Lambda_\Q$ (via the natural projection $\Lambda_K^+\stackrel{\sim}{\rightarrow}\Gamma_\Q$).

\begin{prop}\label{prop:comp-Selcyc}
Let $S'\subset\Sigma$ be any subset of primes not lying above $p$. Then the restriction map from $G_\Q$ to $G_K$ induces a $\Lambda_\Q$-module isomorphism 
\[
\X_{\rm ord}^{S'}(E/K_{\infty}^+)\simeq \X_{\rm ord}^{S'}(E/\Q_\infty)\oplus\X_{\rm ord}^{S'}(E^K/\Q_\infty).
\]
In particular, 
\[
\ch_{\Lambda^+_K}\bigl(\X_{\rm ord}^{S'}(E/K_{\infty}^+)\bigr)=\ch_{\Lambda_{\Q}}\bigl(\X_{\rm ord}^{S'}(E/\Q_\infty)\bigr)\cdot \ch_{\Lambda_{\Q}}\bigl(\X_{\rm ord}^{S'}(E^K/\Q_\infty)\bigr).
\]
\end{prop}

\begin{proof}
This follows readily from the inflation-restriction exact sequence and Shapiro's lemma (see e.g. \cite[Lem.\,3.6]{skinner-urban}).
%
\end{proof}



\subsubsection{From imprimitive to primitive}

As observed by Greenberg, imprimitive Selmer groups as above tend to have better properties with respect to congruences than their primitive counterparts. For our arguments, we shall also find it convenient to work first with imprimitive Selmer group, and so the next results will be useful.

\begin{prop}\label{prop:prim-imprim-ord}
Assume that $E(K)[p]=0$ and that $\X_{\rm ord}(E/K_\infty^+)$ is $\Lambda_K^+$-torsion. Then for any $S'\subset\Sigma$ consisting of primes away from $p$, the Selmer group $\X_{\rm ord}^{S'}(E/K_\infty^+)$ is also $\Lambda_K^+$-torsion, with
\[
\ch_{\Lambda_K^+}\bigl(\X_{\rm ord}^{S'}(E/K_\infty^+)\bigr)=
\ch_{\Lambda_K^+}\bigl(\X_{\rm ord}(E/K_\infty^+)\bigr)\cdot\prod_{w\in S'}\bigl(\mathcal{P}_w^+(1)\bigr),
\]
where $\mathcal{P}_w^+(1)\in\Lambda_K^+$ is as in Definition~\ref{def:w-Euler}, with $\alpha=1$.
\end{prop}

\begin{proof}
From the assumption that $E(K)[p]=0$, we see that the $G_{K_\infty^+}$-invariants of ${\rm Hom}(T_pE,\mu_{p^\infty})$ are trivial, and so by \cite[Prop.~A.2]{pollack2011anticyclotomic} the global-to-local map defining $\rH^1_{\Fcal_{\rm ord}}(K,T_pE\otimes(\Lambda_K^+)^\vee)$ as in (\ref{eq:defSelK}) is surjective. We therefore find an exact sequence
\begin{equation}\label{eq:seq-imp-ord}
0\to\prod_{w\in S'}\rH^1(K_w,T_pE\otimes(\Lambda_K^+)^\vee)^\vee\to\X_{\rm ord}^{S'}(E/K_\infty^+)\to\X_{\rm ord}(E/K_\infty^+)\to 0.
\end{equation}
Since the primes $w\in S'$ split in $K$, by  \cite[Prop.~2.4]{greenvats} the module $\rH^1(K_w,T_pE\otimes(\Lambda_K^+)^\vee)^\vee$ is $\Lambda_K^+$-torsion, with
\[
\ch_{\Lambda_K^+}\bigl(\rH^1(K_w,T_pE\otimes(\Lambda_K^+)^\vee)^\vee\bigr)=\bigl(\mathcal{P}_w^+(1)\bigr).
\]
The result now follows by taking characteristic ideals in (\ref{eq:seq-imp-ord}).
\end{proof}

\begin{cor}\label{cor:imp-IMC} 
Assume that $E(K)[p]=0$ and let $S'\subset\Sigma$ be any subset consisting of primes away from $p$. Then Conjecture~\ref{conj:IMC-K} for $\mathfrak{X}_{\rm ord}(E/K_\infty^+)$ holds if and only if $\mathfrak{X}_{\rm ord}^{S'}(E/K_\infty^+)$ is $\Lambda_K^+$-torsion, with
\[
\ch_{\Lambda_K^+}\bigl(\X_{\rm ord}^{S'}(E/K_{\infty}^+)\bigr)=\bigl(\Lcal_p^{\rm PR}(E/K)^{+,S'}\bigr).
\]
\end{cor}

\begin{proof}
Since for any prime $w\in S'$, the element $\mathcal{P}_w^+(1)\in\Lambda_K^+$ is nonzero, this is clear from Definition~\ref{def:w-Euler} and Proposition~\ref{prop:prim-imprim-ord}.
\end{proof}



\subsection{Anticyclotomic twists and congruences}\label{sec:seltwist}

We now introduce the twisted variants of the Selmer groups from the preceding sections that we shall need. The results of Proposition~\ref{prop:euler-char} and Proposition~\ref{prop:cong-Sel} will play an important role later.


As in $\S\ref{subsec:tw-L}$, let $\Phi$ be a finite extension of $\Q_p$, and let $R$ be the ring of integers of $\Phi$ with uniformizer $\varpi$. We consider a character $\alpha:\Gamma_K^-\to R^{\times}$ which satisfies
\[
\alpha\equiv 1\;({\rm mod}\,\varpi^m)
\]
for some $m>0$. Let $S'\subset\Sigma$ be any subset consisting of primes away from $p$. Replacing $T_pE$ by the twist 
\[
T_\alpha:=T_pE\otimes_{\Z_p}R(\alpha),
\]  
we define (imprimitive) Selmer groups $\rH^1_{\Fcal_{\rm ord}^{S'}}(K,T_\alpha\otimes(\Lambda_K^\pm)^\vee)$ and $\rH^1_{\Fcal_{\rm Gr}^{S'}}(K_,T_\alpha\otimes(\Lambda_K^\pm)^\vee)$ (with Pontryagin duals $\X^{S'}_{\rm ord}(E(\alpha)/K_\infty^\pm)$ and $\X_{\Gr}^{S'}(E(\alpha)/K_\infty^\pm)$, respectively)  
in the same way as before, replacing ${\rm Fil}_w^+(T_pE)$ by ${\rm Fil}_w^+(T_pE)\otimes_{\Z_p}R(\alpha)$ in Definition \ref{defilocalconds}. We also consider their counterparts with $K_\infty^\pm$ and $\Lambda_K^\pm$ replaced by $K_\infty$ and $\Lambda_K$, respectively. 
Each of these Selmer groups is a module for the Iwasawa algebra $\Lambda_{R}= R\widehat\otimes_{\Z_p}\Lambda$ with $\Lambda$ either $\Lambda_\Q$, $\Lambda_K$, or $\Lambda_K^\pm$ (per the definitions).

Put $V_\alpha=T_\alpha\otimes\Q_p$ and $W_\alpha=T_\alpha\otimes\Q_p/\Z_p=V_\alpha/T_\alpha$. We shall also need to consider the following variant of the above Selmer groups for the module $W_\alpha:=T_\alpha\otimes\Q_p/\Z_p$:
\[
\rH^1_{\Fcal_{\rm Gr}^{S'}}(K,V_\alpha):=\ker\Biggl(\rH^1(G_{K,\Sigma},V_\alpha)\rightarrow
\prod_{w\in\Sigma\smallsetminus S', w\nmid p}\rH^1(K_w,V_\alpha)\times\rH^1(K_{\bar{v}},V_\alpha)\Biggr),
\]
and the resulting $\rH^1_{\Fcal_{\rm Gr}^{S'}}(K,T_\alpha)$ and 
$\rH^1_{\Fcal_{\rm Gr}^{S'}}(K,W_\alpha)$ obtained by propagation via $0\rightarrow T_\alpha\rightarrow V_\alpha\rightarrow W_\alpha\rightarrow 0$. 
We also consider the 
Selmer group $\rH^1_{\Fcal_{\rm ord}}(K,V_\alpha)$ consisting of classes that land in 
\[
%
\rH^1_{\rm ord}(K_w,V_\alpha):={\rm im}(\rH^1(K_w,{\rm Fil}_w^+V_\alpha)\rightarrow\rH^1(K_w,V_\alpha))
\] 
at the primes $w\mid p$ and are trivial at the primes $w\nmid p$, and its counterparts $\rH^1_{\Fcal_{\rm ord}}(K,T_\alpha)$ and $\rH^1_{\Fcal_{\rm ord}}(K,W_\alpha)$ defined by propagating the local conditions.

\subsubsection{Ancillary results for $\X_{\rm Gr}(E(\alpha)/K_\infty^\pm)$}
The main result of this section is a relation between the specializations  of the cyclotomic Selmer group $\X_{\rm Gr}(E(\alpha)/K_\infty^+)$ and the anticyclotomic Selmer group $\X_{\rm Gr}(E(\alpha)/K_\infty^-)$ at the trivial character. 

Define
\[
M_\alpha^\pm=T_\alpha\widehat\otimes_{\Z_p}(\Lambda_K^\pm)^\vee,
\]
and for any $S'\subset\Sigma$ consisting of primes not dividing $p$, put
\[
\mathcal{P}_{\rm Gr}(M_\alpha;S'):=\rH^1(K_{\bar{v}},M_\alpha)\times
\prod_{w\in\Sigma\smallsetminus S', w\nmid p}\rH^1(K_w,M_\alpha),
\]
and similarly,
\[
\mathcal{P}_{\Gr}(W_\alpha;S'):=\frac{\rH^1(K_v,W_\alpha)}{\rH^1(K_v,W_\alpha)_{\rm div}}\times\rH^1(K_{\bar{v}},W_\alpha)\times
\prod_{w\in\Sigma\smallsetminus S', w\nmid p}\rH^1(K_w,W_\alpha),
\]
so in particular $\rH^1_{\Fcal_{\Gr}^{S'}}(K,W_\alpha)$ is the kernel of the global-to-local map $\rH^1(G_{K,\Sigma},W_\alpha)\rightarrow\mathcal{P}_{\Gr}(W_\alpha;S')$.

\begin{lemma}\label{lem:coinv}
Assume $E(K)[p]=0$ and $\alpha:\Gamma_K^-\rightarrow R^\times$ is a crystalline character such that:
\begin{itemize}
\item[(a)] ${\rm corank}_{R}\rH^1_{\Fcal_{\rm BK}}(K,W_{\alpha^{-1}})=1$,
\item[(b)] The restriction map
\[
\rH^1_{\Fcal_{\rm ord}}(K,W_{\alpha^{-1}})\xrightarrow{{\rm loc}_v}\rH^1_{\rm ord}(K_v,W_{\alpha^{-1}})
\]
is nonzero. 
\end{itemize}
Then the following hold:
\begin{itemize}
\item[(i)] $\rH^1_{\Fcal_{\rm Gr}}(K,W_\alpha)$ is finite, and  $\rH^1_{\Fcal_{\rm Gr}}(K,T_\alpha)=0$.
\item[(ii)] $\rH^1(G_{K,\Sigma},M_\alpha^\pm)_{\Gamma_K^\pm}=0$.
\item[(iii)] For any $S'\subset\Sigma$ consisting of primes away from $p$, $\rH^1_{\Fcal_{\rm Gr}^{S'}}(K,M_\alpha^\pm)_{\Gamma_K^\pm}=0$.
\end{itemize}
\end{lemma}

\begin{proof}
It follows from their definition by propagation that $\rH^1_{\Fcal_{\rm Gr}}(K,T_\alpha)$ is the $p$-adic Tate module of $\rH^1_{\Fcal_{\rm Gr}}(K,W_\alpha)$, and so part (i) is shown in Proposition\,3.2.1 of \cite{jsw}. For the proof of parts (ii) and (iii),  we shall adapt the arguments in the proof of [\emph{op.\,cit.}, Lem.\,3.3.5]. 
Let $\gamma^\pm\in\Gamma_K^{\pm}$ be any topological generator. From the relation $W_\alpha=
(M_\alpha^\pm)^{\Gamma_K^\pm}=M_\alpha^\pm[\gamma^\pm-1]$ we get an injection 
\begin{equation}\label{eq:inj-h2}
\rH^1(G_{K,\Sigma},M_\alpha^\pm)_{\Gamma_K^\pm}\hookrightarrow\rH^2(G_{K,\Sigma},W_\alpha).\nonumber
\end{equation}
Since the $p$-adic representation $V_{\alpha^{-1}}$ is pure of weight $-1$, we have $\rH^2(K_w,W_\alpha)=\rH^0(K_w,T_{\alpha^{-1}})^\vee=0$ for all $w\in\Sigma$, and so $\rH^2(G_{K,\Sigma},W_\alpha)=\sha_\Sigma^2(K,W_\alpha)$ 
which by Poitou--Tate duality is dual to $\sha_\Sigma^1(K,T_{\alpha^{-1}})$. 
However, from $E(K)[p]=0$ the group $\sha_\Sigma^1(K,T_{\alpha^{-1}})$ is torsion-free, and from our assumption on $\alpha$ it is also of $\Z_p$-corank $0$. Hence $\sha_\Sigma^1(K,T_{\alpha^{-1}})=0$, so also $\rH^2(G_{K,\Sigma},W_\alpha)=0$, and the vanishing of $\rH^1(G_{K,\Sigma},M_\alpha^\pm)_{\Gamma_K^\pm}$ follows. This shows part (ii), and the vanishing of $\rH^1_{\Fcal_{\rm Gr}^{S'}}(K,M_\alpha^\pm)_{\Gamma_K^\pm}$ for any $S'\subset\Sigma$ as in part (iii) then follows from the exact sequence
\[
\rH^1(G_{K,\Sigma},W)=\rH^1(G_{K,\Sigma},M_\alpha^\pm)^{\Gamma_K^\pm}\xrightarrow{\lambda}\mathcal{P}_{\rm Gr}(M_\alpha^{\pm};S')^{\Gamma_K^\pm}\rightarrow\rH^1_{\Fcal_{\rm Gr}^{S'}}(K,M_\alpha^\pm)_{\Gamma_K^\pm}\rightarrow \rH^1(K^\Sigma/K,M_\alpha^\pm)_{\Gamma_K^\pm},
\]
in which the map $\lambda$ is surjective, being the composition of the restriction map $\rH^1(G_{K,\Sigma},W_\alpha)\rightarrow\mathcal{P}_{\Gr}(W_\alpha;S')$ (whose cokernel naturally injects into the dual of $\rH^1_{\Fcal_{\Gr}^{S'}}(K,T_{\alpha})=0$) and the natural map $\mathcal{P}_{\Gr}(W_\alpha;S')\rightarrow\mathcal{P}_{\Gr}(M_\alpha^{\pm};S')^{\Gamma_K^\pm}$, whose surjectivity follows from the fact that for $w\in\Sigma$, the local Galois group ${\rm Gal}(K_{\infty,\eta}^\pm/K_w)$ (for any $\eta\vert w$ in $K_\infty^\pm$) is either trivial or isomorphic to $\Z_p$, and so has $p$-cohomological dimension $\leq 1$.
\end{proof}

As standard, we shall often identify $\Lambda_K^\pm$ with the one variable power series ring $\Z_p[\![T^\pm]\!]$ via $\gamma^\pm=1+T^\pm$ upon the choice of a topological generator $\gamma^\pm\in\Lambda_K^\pm$. Thus denoting by $\mathcal{L}\mapsto\mathcal{L}^\pm$ the natural projection $\Lambda_K\rightarrow\Lambda_K^\pm$, the equality 
\[
\Lcal_p^{\rm PR}(E(\alpha)/K)^{+,S'}(0)=\Lcal_p^{\rm PR}(E(\alpha)/K)^{-,S'}(0)
\] 
is clear, reflecting in particular the fact that the trivial character is both cyclotomic and anticyclotomic. The next  result (which we shall only need for $S'=\emptyset$) is a parallel equality for characteristic power series. 

\begin{prop}\label{prop:euler-char}
Suppose $E(K)[p]=0$ and $\alpha:\Gamma_K^-\rightarrow R^\times$ is such that the conditions in Lemma~\ref{lem:coinv} hold.
Then the Selmer groups $\mathfrak{X}_{\rm Gr}^{S'}(E(\alpha)/K_\infty^+)$ and $\mathfrak{X}^{S'}_{\rm Gr}(E(\alpha)/K_\infty^-)$ are torsion over $\Lambda_K^+$ and $\Lambda_K^-$, respectively, where $S'\subset\Sigma$ is any subset consisting of primes away from $p$. 
%
Furthermore, we have the equality up to a $p$-adic unit:
\[
\mathcal{F}_{\rm Gr}^{S'}(E(\alpha)/K_\infty^+)(0)\;\sim_p\;\mathcal{F}_{\rm Gr}^{S'}(E(\alpha)/K_\infty^-)(0),
\]
where $\Fcal_{\rm Gr}^{S'}(E(\alpha)/K_\infty^\pm)\in\Lambda^\pm_{K,R}$ is any characteristic power series for $\X_{\rm Gr}^{S'}(E(\alpha)/K_\infty^\pm)$.
\end{prop}

\begin{proof} 
By Poitou--Tate duality, the cokernel of the global-to-local map in the defining exact sequence
\[
0\rightarrow\rH^1_{\Fcal_{\Gr}}(K,W_\alpha)\rightarrow\rH^1(G_{K,\Sigma},W_\alpha)\rightarrow\mathcal{P}_{\Gr}(W_\alpha;\emptyset)
\]
injects into the Pontryagin dual of the Selmer group $\rH^1_{\Fcal_{\rm Gr}^*}(K,T_{\alpha^{-1}})$ dual to $\rH^1_{\Fcal_{\rm Gr}}(K,W_\alpha)$, defined by the local conditions given by the orthogonal complement of those in $\mathcal{P}_{\rm Gr}(W_\alpha;\emptyset)$ under local Tate duality. Since this dual Selmer group is torsion-free by the assumption $E(K)[p]=0$, and it follows from part (i) of Lemma~\ref{lem:coinv} and the finiteness of $\rH^1(K_w,T_{\alpha^{-1}})$ for finite primes $w\nmid p$  that it has finite order, we conclude that the above global-to-local map is surjective. It is then immediate that for any $S'$ as in the statement, we have an analogous exact sequence
\[
0\rightarrow\rH^1_{\Fcal_{\Gr}^{S'}}(K,W_\alpha)\rightarrow\rH^1(G_{K,\Sigma},W_\alpha)\rightarrow\mathcal{P}_{\Gr}(W_\alpha;S')\rightarrow 0.
\] 
A variant of Mazur's control theorem shows that the natural restriction map
\[
\rH^1_{\Fcal_{\Gr}^{S'}}(K,W_\alpha)\rightarrow
\rH^1_{\Fcal_{\rm Gr}^{S'}}(K,M_\alpha^\pm)^{\Gamma_K^\pm}
\]
has finite kernel and cokernel. By Lemma~\ref{lem:coinv} and the above remarks, it follows that $\mathfrak{X}_{\rm Gr}^{S'}(E(\alpha)/K_\infty^\pm)$ is $\Lambda_K^\pm$-torsion (for both choices of sign $\pm$), and together with the general result \cite[Lem.\,4.2]{greenberg-cetraro} for the $\Gamma_K^\pm$-Euler characteristic of $\mathcal{F}_{\rm Gr}^{S'}(E(\alpha)/K_\infty^\pm)$ we obtain the equality up to a $p$-adic unit:
\begin{equation}\label{eq:Euler-char}
\mathcal{F}_{\Gr}^{S'}(E(\alpha)/K_\infty^\pm)(0)\,\sim_p\;\#\rH^1_{\Fcal_{\Gr}^{S'}}(K,M_\alpha^\pm)^{\Gamma_K^\pm}.
\end{equation}
Now, from the Snake Lemma applied to the commutative diagram with exact rows
\[
\xymatrix{
0\ar[r]&\rH^1_{\Fcal_{\Gr}^{S'}}(K,W_\alpha)\ar[r]\ar[d]^{s^\pm}&\rH^1(K^\Sigma/K,W_\alpha)\ar[r]\ar[d]^{h^\pm}&\mathcal{P}_{\Gr}(W_\alpha;S')\ar[r]\ar[d]^{r^\pm}&0\\
0\ar[r]&\rH^1_{\Fcal_{\rm Gr}^{S'}}(K,M_\alpha^\pm)^{\Gamma_K^\pm}\ar[r]&\rH^1(K^\Sigma/K,M_\alpha^\pm)^{\Gamma_K^\pm}\ar[r]&\mathcal{P}_{\rm Gr}(M_\alpha^\pm;S')^{\Gamma_K^\pm},
}
\]
we obtain
\begin{equation}\label{eq:snake}
\#\rH^1_{\Fcal_{\rm Gr}^{S'}}(K,M_\alpha^\pm)^{\Gamma_K^\pm}=\#\rH^1_{\Fcal_{\Gr}^{S'}}(K,W_\alpha)\cdot\frac{\#{\rm ker}(r^\pm)}{\#{\rm ker}(h^\pm)}.
\end{equation}

Clearly, 
\[
{\rm ker}(h^\pm)=\rH^1(\Gamma_K^\pm,(M_\alpha^\pm)^{G_K})=(M_\alpha^\pm)^{G_K}/(\gamma^\pm-1)(M_\alpha^\pm)^{G_K}=\rH^0(K_\infty^\pm,W_\alpha),
\]
and this vanishes by our assumptions. On the other hand, since we assume that every finite prime $w\in\Sigma$ splits in $K$, the argument in the proof of \cite[Prop.\,3.3.7]{jsw} (but noting that here the roles $v$ and $\bar{v}$ are reversed) shows that for both choices of sign $\pm$, the order of ${\rm ker}(r^\pm)$ is given 
\[
\#{\rm ker}(r^\pm)=\#\rH^0(K_v,W_{\alpha^{-1}})^2\cdot\prod_{w\in\Sigma\smallsetminus S',w\nmid p}c_w^{(p)}(W_\alpha),
\]
where $c_w^{(p)}=\#\rH^1_{\rm ur}(K_w,W_\alpha)$ is the $p$-part of the local Tamagawa number of $W_\alpha$ at $w$. Thus the value in (\ref{eq:snake}) is the same for both choices of sign $\pm$, and together with (\ref{eq:Euler-char}) this yields the result.
%
\end{proof}

\subsubsection{Ancillary results for $\X_{\rm ord}(E(\alpha)/K_\infty^+)$}
Finally, in this section we prove a congruence relation modulo $\varpi^m$ for the (characteristic power series of the) Selmer groups $\X_{\rm ord}^S(E(\alpha)/K_\infty^+)$ and $\X_{\rm ord}^S(E/K_\infty^+)$, where $\alpha$ is an anticyclotomic character as above such that $\alpha\equiv 1\mod\varpi^m$. We start with the following preliminary lemma. 


\begin{lemma}\label{lem:almost-div-ord}
Assume that $E(K)[p]=0$ and that $\X_{\rm ord}^{S'}(E(\alpha)/K_\infty^+)$ is $\Lambda_{K,R}^+$-torsion, where $S'\subset\Sigma$ is a subset consisting of primes away from $p$. Then $\X_{\rm ord}^{S'}(E(\alpha)/K_\infty^+)$ has no nonzero finite $\Lambda_{K,R}^+$-submodules.
\end{lemma}

\begin{proof} 
This is shown in \cite[Prop.\,4.14]{greenberg-cetraro} when $\alpha=1$ and $S'=\emptyset$, and the general case follows from a slight variation of the same arguments (see e.g. \cite[Prop.\,2.3.3]{skinner-mult}). Alternatively, this can be seen as a special case of Greenberg's results \cite{greenberg-Sel}.
\end{proof}

Put
\[
S=\Sigma\smallsetminus\{p,\infty\},
\]
which as above we shall view as a set of primes of $\Q$ or of $K$ according to context. The next important result is an algebraic counterpart of Lemma~\ref{lem:cong-L}.

\begin{prop}\label{prop:cong-Sel}
Assume that $E(K)[p]=0$, that $\X_{\rm ord}^S(E(\alpha)/K_\infty^+)$ is $\Lambda^+_{K,R}$-torsion, and that
$\X_{\rm ord}^S(E/K_\infty^+)$ is $\Lambda_K^+$-torsion. 
If $\alpha\equiv 1\;({\rm mod}\,\varpi^m)$, then there are suitable characteristic power series $\Fcal_{\rm ord}^S(E(\alpha)/K_\infty^+)$ and $\Fcal_{\rm ord}^S(E/K_\infty^+)$ for the 
modules $\X_{\rm ord}^S(E(\alpha)/K_\infty^+)$ and $\X_{\rm ord}^S(E/K_\infty^+)$, respectively, such that 
\[
\Fcal_{\rm ord}^S(E(\alpha)/K_\infty^+)\equiv\Fcal_{\rm ord}^S(E/K_\infty^+)\;\,({\rm mod}\,\varpi^m).
\]
\end{prop}

\begin{proof}
Since 
 $\X_{\rm ord}^S(E(\alpha)/K_\infty^+)$ and $\X_{\rm ord}^S(E/K_\infty^+)$ have no nonzero finite $\Lambda_K^+$-submodules by Lemma~\ref{lem:almost-div-ord}, 
their characteristic ideals are the same as their Fitting ideals (see \cite[Lem.~2.3.4(ii)]{skinner-mult}), so to prove the result it suffices to show that
\begin{equation}\label{eq:Sel-m}
\rH^1_{\Fcal_{\rm ord}^S}(K,M)[\varpi^m]\simeq\rH^1_{\Fcal_{\rm ord}^S}(K,M_\alpha)[\varpi^m],
\end{equation}
where $M=T_pE\widehat{\otimes}_{\Zp}(\Lambda_{K,R}^+)^\vee$ and $M_\alpha=T_pE(\alpha)\widehat\otimes_{\Z_p}(\Lambda_K^+)^\vee$. By the assumption  $E(K)[p]=0$, the natural maps  
\[
\rH^1(G_{K,\Sigma},M[\varpi^m])\rightarrow\rH^1(G_{K,\Sigma},M)[\varpi^m],\quad 
\rH^1(G_{K,\Sigma},M_\alpha[\varpi^m])\rightarrow\rH^1(G_{K,\Sigma},M_\alpha)[\varpi^m]
\]
are isomorphisms. Moreover, for any $w\vert p$ the restriction map $\rH^1(K_w,M^-)\rightarrow\rH^1(I_w,M^-)$ is an injection (this uses $p\nmid N$), where we put $M^-=(T_pE/{\rm Fil}_w^+T_pE)\otimes(\Lambda_K^+)^\vee$. Thus $\rH^1_{\Fcal_{\rm ord}^S}(K,M)[\varpi^m]$ is naturally identified with the kernel of the restriction map
\[
\rH^1(G_{K,\Sigma},M[\varpi^m])\rightarrow\prod_{w\vert p}\rH^1(I_w,M^-),
\]
which factors through $\rH^1(G_{K,\Sigma},M[\varpi^m])\rightarrow\prod_{w\vert p}\rH^1(I_w,M^-[\varpi^m])$.  Since the kernel of the natural map $\rH^1(I_w,M^-[\varpi^m])\rightarrow\rH^1(I_w,M^-)$ is given by $(M^-)^{I_w}/p^m(M^-)^{I_w}$, and this is zero since $(M^-)^{I_w}\simeq{\rm Hom}_{\rm cts}(R,\Q_p/\Z_p)$ is $p$-divisible, we conclude that 
\[
\rH^1_{\Fcal_{\rm ord}^S}(K,M)[\varpi^m]={\rm ker}\biggl(\rH^1(G_{K,\Sigma},M[\varpi^m])\rightarrow\prod_{w\vert p}\rH^1(I_w,M^-[\varpi^m])\biggr).
\]
Letting $M_\alpha^-=(T_pE(\alpha)/{\rm Fil}_w^+T_pE(\alpha))\otimes(\Lambda_K^+)^\vee$ we similarly find $(M_\alpha^-)^{I_w}\simeq{\rm Hom}_{\rm cts}(R,\Q_p/\Z_p)$, noting that after restriction to $G_{K_{\infty,w}^+}$ the character $\alpha$ becomes unramified; thus  $(M_\alpha^-)^{I_w}/p^m(M_\alpha^-)^{I_w}=0$ and so $\rH^1_{\Fcal_{\rm ord}^S}(K,M_\alpha)[\varpi^m]$ is identified with the kernel of the restriction map
\[
\rH^1(G_{K,\Sigma},M_\alpha[\varpi^m])\rightarrow\prod_{w\vert p}\rH^1(I_w,M_\alpha^-[\varpi^m]).
\]
Since $M_\alpha[\varpi^m]=M[\varpi^m]$, this proves (\ref{eq:Sel-m}) and yields the result.
\end{proof}

\section{Beilinson--Flach classes}\label{BF}

Let $f=\sum_{n=1}^\infty a_nq^n\in S_2(\Gamma_0(N))$ be a newform with 
Fourier coefficients in $\Q$, and fix a prime $p\nmid 2N$. We denote by ${Y_1(N)}$ the modular curve of level $\Gamma_1(N)$. The $p$-adic Galois representation $V_f$ associated to $f$ can be geometrically realized as the maximal quotient of $\rH^1_{\et}(\overline{Y_1(N)},\Q_p)(1)$ on which the Hecke operators $T_n$ acts as $a_n$. (Note that this is denoted $V_{\Q_p}(f)^*$ in \cite[Def.~6.3]{RSCMF}, and corresponds to  $V_{\Q_p}(f)(1)$ in the notations of \cite{kato-euler-systems}.) Let $T_f$ be the $\Z_p$-submodule $V_f$ generated by the image of $\rH^1_{\et}(\overline{Y_1(N)},\Z_p)(1)$, which 
is a $G_\Q$-stable $\Z_p$-lattice 
in $V_f$. 

We assume that $f$ is ordinary at $p$, i.e., $a_p\in\Z_p^\times$, so there is a $G_{p}$-stable filtration 
\[
0\rightarrow T_f^+\rightarrow T_f\rightarrow T_f^-\rightarrow 0
\] 
with $T_f^{\pm}$ free rank one $\Z_p$-modules with the $G_{p}$-action on $T_f^-$ given by the unramified character sending an arithmetic Frobenius to the $p$-adic unit root of $x^2-a_px+p$. Replacing ${\rm Fil}_w^+(T_pE)$ with $T_f^+$ we can define the ordinary local condition for $T_f$ as in Definition \ref{defilocalconds}.

Let $K$ be an imaginary quadratic field satisfying (\ref{eq:spl}). 
In this section we introduce the special case of the Beilinson--Flach classes of \cite{explicit} that will be needed for our arguments, and deduce some applications.

\subsection{Reciprocity laws}\label{subsec:ERL}

Let $\boldsymbol{g}\in\Lambda_{\boldsymbol{g}}\llbracket{q}\rrbracket$ be the CM Hida family from $\S\ref{subsec:L-2var-II}$, and fix a character $\alpha:\Gamma_K\rightarrow R^\times$ with values in the ring of integers of a finite extension $\Phi/\Q_p$. With a slight abuse of notation, we continue to denote by $\Lambda_K$ the Iwasawa algebra $R\llbracket{\Gamma_K}\rrbracket$ and by $\Lambda_{\bg}$ the extension of scalars $\Lambda_{\bg}\otimes_{\Z_p} R$. 


Let $I_f \subset \Qp$ be the image of the unique Hecke-equivariant map $M_2(\Gamma_0(N),\Z_p)\rightarrow \Q_p$, with the Hecke action on $\Q_p$ 
such that $T_n$ acts as multiplication by $a_n$, that sends $f$ to $1$. Then $I_f$ is a free $\Z_p$-module of rank one. 
We similarly define $I_{\bg}^{\rm cusp} \subset \Frac(\Lambda_{\bg})$ as  the image of the $\Lambda_{\boldsymbol{g}}$-adic cuspforms in ${\rm Frac}(\Lambda_{\boldsymbol{g}})$ under the projection corresponding to ${\boldsymbol{g}}$. Note that $I_{\bg}^{\rm cusp}$ is a finitely-generated $\Lambda_{\bg}$-module.

\begin{thm}
\label{thm:KLZ}
There exists a class
\[
{BF}_\alpha\in\rH^1_{\Fcal_{\rm ord,\rm rel}}(K,T_f(\alpha)\widehat\otimes\Lambda_K)
\]
and injective $\Lambda_K$-linear maps with pseudo-null cokernel
\begin{align*}
\widetilde{\rm Col}_{f}:\rH^1(K_{\bar{v}},T_f^-(\alpha)\widehat\otimes\Lambda_K) \to I_f\widehat\otimes\Lambda_K, \ \ 
\widetilde{\rm Col}_{\boldsymbol{g}}:\rH^1(K_{v},T_f^+(\alpha)\widehat\otimes\Lambda_K) \to I_{\bg}^{\rm cusp}\widehat\otimes\Lambda_K,
\end{align*}
satisfying the following explicit reciprocity laws:
\begin{itemize}
\item[(a)] 
$\widetilde{\rm Col}_f(p^-({\rm loc}_{\bar v}({BF}_\alpha)))=\Lcal_p(f(\alpha)/K,\Sigma^{(1)})$, where $p^-({\rm loc}_{\bar{v}}({BF}_\alpha))$ is the natural image of ${\rm loc}_{\bar{v}}({BF}_\alpha)$ in $\rH^1(K_{\bar{v}},T_f^-(\alpha)\widehat\otimes\Lambda_K)$;
\item[(b)] $\widetilde{\rm Col}_{\boldsymbol{g}}({\rm loc}_{v}({BF}_\alpha))=\Lcal_p(f(\alpha)/K,\Sigma^{(2')})$.
\end{itemize}
\end{thm}

\begin{proof} This is proved in \cite[\S5]{BST}, where it is deduced from results in \cite{explicit} 
in combination with results in \cite{Betina-Dimitrov-Pozzi}, though here we have reversed the roles of $v$ and $\bar v$. 
\end{proof}

Let $E_1$ be the optimal curve in the isogeny class associated with $f$, in the sense of \cite{stevens-invmath}, with optimal parametrization $\pi_1:X_1(N)\rightarrow E_1$. 
The inclusion $Y_1(N)\hookrightarrow X_1(N)$ identifies the maximal quotient of $\rH^1_{\et}(\overline{X_1(N)},\Q_p(1))$ on which the Hecke operators $T_n$ acts as $a_n$
with the similar quotient of $\rH^1_{\et}(\overline{Y_1(N)},\Q_p(1))$, that is, with $V_f$. Via this identification and the parametrization $\pi_1$, the Tate module $T_pE_1$ is identified with the image of $\rH^1_{\et}(\overline{X_1(N)},\Z_p(1))$ in $V_f$; this is a sublattice of $T_f$. 
Let $E_\bullet/\Q$ be the elliptic curve in the isogeny class associated with $f$ constructed in \cite{wuthrich-int}. By construction, there is a cyclic isogeny
$\phi:E_1\rightarrow E_\bullet$ that is \'etale in characteristic $p$ and for which the resulting parametrization $\pi_\bullet = \phi\circ\pi_1:X_1(N)\rightarrow E_\bullet$ 
identifies $T_pE_\bullet$ with $T_f$ in $V_f$:
\[
T_pE_\bullet = T_f.
\] 
(cf. Proposition~8 and Theorem~4 of  \emph{op.\,cit.}).


\begin{lemma}\label{lem:cong-deg}
We have $\deg(\pi_\bullet)I_f  = \Zp$.
\end{lemma}

\begin{proof} 
Let $\omega_{E_\bullet}$ (resp.~$\omega_{E_1}$) be a N\'eron differential for $E_\bullet$ (resp.~$E_1$). Then 
\[
\pi_{\bullet,*}\pi_\bullet^*\omega_{E_\bullet} = \deg(\pi_\bullet)\omega_{E_\bullet}.
\]
As the isogeny $\phi$ is \'etale, $\phi^*\omega_{E_\bullet} = u_1 \omega_{E_1}$ for some $u_1\in \Z_p^\times$. Similarly,
$\pi_1^* \omega_{E_1} = c_1\omega_f$ for some $c_1\in \Z_p^\times$, as $p\nmid 2N$ (see \cite[Cor.\,4.1]{mazur} and \cite[Prop.\,~3.3]{greenvats}). Hence, $\pi_{\bullet,*}\omega_f = u_1c_1\deg(\pi_\bullet)\omega_{E_\bullet}$
and so it suffices to show that $\pi_{\bullet,*}\omega_f = a \omega_{E_\bullet}$ for some $a\in\Zp$ such that $a I_f = \Zp$.

As $E_\bullet$ has ordinary reduction at the prime $p$, the induced map $\rH^1_{\et}(\overline{Y_1(N)},\Zp(1))\twoheadrightarrow T_f = T_pE_\bullet$ factors through 
projection to the ordinary summand $\mathcal{H}^{\ord} :=e_{\ord} \rH^1_{\et}(\overline{Y_1(N)},\Zp(1))$. There is a corresponding commutative diagram
\[
\begin{tikzcd}
\mathcal{H}^+ \arrow[r,hook] \arrow[d,two heads] & \mathcal{H}^{\ord} \arrow[r,two heads] \arrow[d,two heads] & \mathcal{H}^- \arrow[d,two heads] \\
T_{f}^+ \arrow[r,hook] & T_{f} \arrow[r,two heads] & T_f^-
\end{tikzcd}
\]
of $G_{\Qp}$-modules,
where $\mathcal{H}^+$ (resp.~$\mathcal{H}^-$) is the maximal submodule (resp.~quotient) on which the inertia group $I_p$ acts non-trivially (resp.~trivially); this non-trivial
action is via the cyclotomic character. The middle arrow is the defining projection, which induces the other two maps. 

Let $M(f)$ denote the maximal torsion-free quotient of $M_2(\Gamma_1(N),\Z_p)$ on which the Hecke operator $T_n$ acts as $a_n$ for all $n$. 
Since $p\nmid 2N$, the integral de Rham - \'{e}tale comparison isomorphisms (we need the log version for the open curve $Y_1(N)$)
induce compatible identifications 
\[
e_{\ord}\rH^0(\Omega_{X_1(N)/\Zp}(\log(\mathrm{cusps})) = e_{\ord} M_2(\Gamma_1(N))\simeq \mathcal{H}^-
\]
and
\[ 
M(f) \simeq T_f^-= T_pE_\bullet^- \simeq\rH^0(\Omega_{E_\bullet/\Zp}) = \Zp \omega_{E_\bullet}. 
\]
Via these, $\pi_{\bullet,*}\omega_f$ is identified with the image of $f$ in
$M(f)$. Since $\omega_{E_\bullet}$ is identified with a $\Zp$-generator of $M(f)$ and we have an isomorphism of rank one $\Zp$-modules $M(f) \simeq I_f$, $f\mapsto 1$,
it follows that $\pi_{\bullet,*}\omega_f = a \omega_{E_\bullet}$ for some $a\in \Zp$ such that $aI_f = \Zp$, as desired.
\end{proof}

%

Recall the $\alpha$-twisted versions of the two-variable $p$-adic $L$-functions $\Lcal_p^{\rm PR}(E_\bullet/K)$ and $\Lcal_p^{\rm Gr}(f/K)$ introduced in Definitions~\ref{def:PR} and \ref{def:Gr}, respectively.

\begin{cor}\label{cor:KLZ}
There are injective $\Lambda_K$-linear maps with pseudo-null cokernel 
\[
{\rm Col}_{E_\bullet}:\rH^1(K_{\bar{v}},(T_pE_\bullet)^-(\alpha)\widehat\otimes\Lambda_K)\to\Lambda_K,\quad
{\rm Col}_{\boldsymbol{g}}:\rH^1(K_{v},(T_pE_\bullet)^+(\alpha)\widehat\otimes\Lambda_K)\to\Lambda_K^{\rm ur},
\]
such that 
\[
{\rm Col}_{E_\bullet}(p^-({\rm loc}_{\bar v}(BF_{\alpha})))=\Lcal_p^{\rm PR}(E_\bullet(\alpha)/K),\quad{\rm Col}_{\boldsymbol{g}}({\rm loc}_{v}(BF_{\alpha}))=\Lcal_p^{\rm Gr}(f(\alpha)/K).
\]
\end{cor}

\begin{proof}
By \cite[Cor.~4.1]{mazur} and \cite[Prop.~3.3]{greenvats}, the Manin constant associated to the modular parametrization $\pi_\bullet:X_1(N)\rightarrow E_\bullet$ 
is a $p$-adic unit, so setting 
\[
{\rm Col}_{E_\bullet}:=\deg(\pi_{\bullet})\cdot H_p(f)\cdot\widetilde{\rm Col}_{f},
\]
the first part of the result follows from Theorem~\ref{thm:KLZ} and Lemma~\ref{lem:cong-deg}. On the other hand, as explained in the proof of Lemma~\ref{lem:integral}, the characteristic power series $H_{\bg}^{\rm cusp}$ of the CM family $\boldsymbol{g}$ divides $h_K\cdot\Lcal_v(K)^-$. Since on the other hand by \cite[Cor.~5.6]{hida-coates} $H(\boldsymbol{g})$ is divisible by $h_K\cdot\Lcal_v(K)^-$, it follows that the congruence ideal of $\boldsymbol{g}$ is generated by $h_K\cdot\Lcal_v(K)^-$. Thus setting 
\[
\quad{\rm Col}_{\boldsymbol{g}}:=h_K\cdot\Lcal_v(K)^-\cdot\widetilde{\rm Col}_{\boldsymbol{g}}
\] 
the second part of the result follows from 
Theorem~\ref{thm:KLZ}.
\end{proof}

\subsection{Iwasawa main conjectures}

One key application of Corollary~\ref{cor:KLZ} is in relating different instances of the main conjectures for the cyclotomic $\Z_p$-extension $K_\infty^+/K$. Similarly as in $\S\ref{subsec:Sel-str}$, for any elliptic curve $E/\Q$ we put 
\begin{align*}
\X_{\rm Gr}(E(\alpha)/K_\infty^+):=\X_{\rm rel,\rm str}(E(\alpha)/K_\infty^+),\quad
\mathfrak{S}_{\rm ord,\rm rel}(E(\alpha)/K_\infty^+)&=\rH^1_{\Fcal_{\rm ord,\rm rel}}(K,T_pE(\alpha)\otimes\Lambda_K^+),
\end{align*}
etc.. Write $BF_{\alpha}^+$ for the image of the class $BF_{\alpha}$ of Theorem~\ref{thm:KLZ} under the natural projection
\[
\rH^1_{\Fcal_{\rm ord,\rm rel}}(K,(T_pE_\bullet)(\alpha)\widehat\otimes\Lambda_K)\rightarrow
\rH^1_{\Fcal_{\rm ord,\rm rel}}(K,(T_pE_\bullet)(\alpha)\widehat\otimes\Lambda_K^+), 
\]
so we have $BF_{\alpha}^+\in\mathfrak{S}_{\rm ord,\rm rel}(E_\bullet(\alpha)/K_\infty^+)$.

\begin{prop}\label{prop:equiv}
Suppose $E_\bullet(K)[p]=0$ and that $\Lcal_p^{\rm PR}(E_\bullet(\alpha)/K)^+$ and $\Lcal_p^{\rm Gr}(f(\alpha)/K)^+$ are both nonzero. Then the following are equivalent:
\begin{itemize}
\item[(i)] $\mathfrak{S}_{\rm ord,\rm rel}(E_\bullet(\alpha)/K_\infty^+)$ has $\Lambda_K^+$-rank one, $\X_{\rm ord,\rm str}(E_\bullet(\alpha^{-1})/K_\infty^+)$ is $\Lambda_K^+$-torsion, and
\[
\ch_{\Lambda_K^+}\bigl(\X_{\rm ord,\rm str}(E_\bullet(\alpha^{-1})/K_\infty^+\bigr)\supset\ch_{\Lambda_K^+}\bigl(\mathfrak{S}_{\rm ord,\rm rel}(E_\bullet(\alpha)/K_\infty^+)/\Lambda_K^+BF_{\alpha}^+\bigr).
\]
\item[(ii)] $\mathfrak{S}_{\rm str,\rm rel}(E_\bullet(\alpha)/K_\infty^+)$ and $\X_{\rm Gr}(E_\bullet(\alpha)/K_\infty^+)$ are both $\Lambda_K^+$-torsion, and
\[
\ch_{\Lambda_K^+}\bigl(\X_{\rm Gr}(E_\bullet(\alpha)/K_\infty^+)\bigr)\Lambda_K^{+,\rm ur}\supset\bigl(\Lcal_p^{\rm Gr}(f(\alpha)/K)^+\bigr).
\]
\item[(iii)] $\mathfrak{S}_{\rm ord}(E_\bullet(\alpha)/K_\infty^+)$ and $\X_{\rm ord}(E_\bullet(\alpha)/K_\infty^+)$ are both $\Lambda_K^+$-torsion, and
\[
\ch_{\Lambda^+_K}\bigl(\X_{\rm ord}(E_\bullet(\alpha)/K_\infty^+)\bigr)\supset\bigl(\Lcal_p^{\rm PR}(E_\bullet(\alpha)/K)^+\bigr).
\]
\end{itemize}
The same result holds for the opposite divisibilities.
\end{prop}

\begin{proof}
This is a well-known consequence of Poitou--Tate duality and the reciprocity laws of Theorem~\ref{thm:KLZ}, but we provide the details for the convenience of the reader. Below we set $\mathfrak{S}_{\rm str,\rm rel}=\mathfrak{S}_{\rm str,\rm rel}(E_\bullet(\alpha)/K_\infty^+)$, $\X_{\rm Gr}=\X_{\Gr}(E_\bullet(\alpha)/K_\infty^+)$, etc. for the ease of notation; and similarly, $\X_{\rm ord}(\alpha^{-1})=\X_{\rm ord}(E_\bullet(\alpha^{-1})/K_\infty^+)$, $\X_{\rm Gr}(\alpha^{-1})=\X_{\Gr}(E_\bullet(\alpha^{-1})/K_\infty^+)$, etc.. Note that, since $\alpha^{\tau}=\alpha^{-1}$ and $(T_pE_{0}\otimes (\Lambda_K^+)^{\vee})^{\tau}\simeq T_pE_{0}\otimes (\Lambda_K^+)^{\vee}$, the action of complex conjugation gives rise to isomorphisms of $\Lambda_K^+$-modules:
\begin{equation}\label{eq:cc}
\X_{\Gr}(\alpha^{-1})\simeq \X_{\Gr}, \ \ \ \X_{\rm ord}(\alpha^{-1})\simeq \X_{\rm ord} \ \ \ \X_{\rm ord, str}(\alpha^{-1})\simeq \X_{\rm str, ord}.
\end{equation}

For the equivalence ${\rm (i)}\Leftrightarrow{\rm (ii)}$ consider the exact sequence
\begin{equation}\label{eq:PT1}
0\rightarrow\mathfrak{S}_{\rm str,\rm rel}\rightarrow\mathfrak{S}_{\rm ord,\rm rel}\rightarrow\rH^1_{\rm ord}(K_{v},T_pE_\bullet(\alpha)\otimes\Lambda_K^+)\rightarrow\X_{\rm Gr}(\alpha^{-1})\rightarrow\X_{\rm ord,\rm str}(\alpha^{-1})\rightarrow 0.
\end{equation}
In both cases, we see that $\mathfrak{S}_{\rm str,\rm rel}$ is $\Lambda_K^+$-torsion, hence trivial (by the assumption $E_\bullet(K)[p]=0$), and (\ref{eq:PT1}) yields the exact sequence
\[
0 \to\mathfrak{S}_{\rm ord,\rm rel}/\Lambda_K^+ BF^+_{\alpha} \to\rH^1_{\rm ord}(K_{v},T_pE_\bullet(\alpha)\otimes\Lambda_K^+)/\Lambda_K^+\loc_{v}(BF_{\alpha}^+)\to\X_{\rm Gr}(\alpha^{-1})\to \X_{\rm ord,\rm str}(\alpha^{-1}) \to 0.
\]
Since by Corollary~\ref{cor:KLZ} the second term in this exact sequence is pseudo-isomorphic---via the map ${\rm Col}_{\bg}$---to $\Lambda_K^{+,\rm ur}/(\Lcal_p^{\rm Gr}(f(\alpha)/K)^+)$, the equivalence ${\rm (i)}\Leftrightarrow{\rm (ii)}$ follows from this, the multiplicativity of characteristic ideals, and the isomorphisms \eqref{eq:cc}. On the other hand, for the equivalence ${\rm (i)}\Leftrightarrow{\rm (iii)}$ consider the exact sequence
\begin{equation}\label{eq:PT2}
0\rightarrow\mathfrak{S}_{\rm ord}\rightarrow\mathfrak{S}_{\rm ord,\rm rel}\rightarrow\rH_{/\rm ord}^1(K_{\bar v},T_pE_\bullet(\alpha)\otimes\Lambda_K^+)\rightarrow\X_{\rm ord}(\alpha^{-1})\rightarrow\X_{\rm ord,\rm str}(\alpha^{-1})\rightarrow 0,
\end{equation}
where 
\[
\rH^1_{/\rm ord }(K_{\bar v},T_pE_\bullet(\alpha)\otimes\Lambda_K^+):=\frac{\rH^1(K_{\bar v},T_pE_\bullet(\alpha)\otimes\Lambda_K^+)}{\rH^1_{\rm ord }(K_{\bar v},T_pE_\bullet(\alpha)\otimes\Lambda_K^+)}\simeq\rH^1(K_{\bar v},T_f^-(\alpha)\widehat\otimes\Lambda_K).
\]
Similarly as before, in both cases we find $\mathfrak{S}_{\rm ord}$ is $\Lambda_K^+$-torsion, so from (\ref{eq:PT2}) we obtain the exact sequence
\[
0\to \mathfrak{S}_{\rm ord,\rm rel}/\Lambda_K^+ BF_{\alpha}^+ \to \rH^1_{/\rm ord}(K_{\bar{v}},T_pE_\bullet\otimes\Lambda_K^+)/\Lambda_K^+p^-(\loc_{\bar{v}}(BF_{\alpha}^+))\to \X_{\rm ord}(\alpha^{-1})\to \X_{\rm ord,\rm str}(\alpha^{-1}) \to 0.
\]
Since by Corollary~\ref{thm:KLZ} the second term in this exact sequence is pseudo-isomorphic---via the map ${\rm Col}_{E_\bullet}$---to $\Lambda_K^+/(\Lcal_p^{\rm PR}(E_\bullet(\alpha)/K)^+)$, taking characteristic ideals and applying \eqref{eq:cc} the result follows.
\end{proof}

 
\subsection{The Beilinson--Flach Euler system divisibility}\label{sec:BF}

In the case $\alpha=1$, the ``upper bound'' divisibility in Conjecture~\ref{conj:IMC-K} for $\X_{\rm ord}(E_\bullet/K_\infty^+)$ can be deduced from Kato's work together with Proposition~\ref{prop:comp-Selcyc}. 
For our later arguments (especially for elliptic curves $E/\Q$ of rank $>1$), we will need a similar divisibility for a twist by a non-trivial character $\alpha$ of $\Gamma_K^-$, a result that we shall deduce from 
Proposition~\ref{prop:equiv} and the next result.


\begin{thm}\label{thm:BF-ES} 
Suppose $\alpha\neq 1$ is a non-trivial character of $\Gamma_K^-$ and 
$BF_{\alpha}^+\neq 0$. Then $\X_{\rm ord,\rm str}(E_\bullet(\alpha^{-1})/K_\infty^+)$ is $\Lambda_K^+$-torsion,  $\mathfrak{S}_{\rm ord,\rm rel}(E_\bullet(\alpha)/K_\infty^+)$ has $\Lambda_K^+$-rank one, and we have the divisibility
\[
\ch_{\Lambda_K^+}\bigl(\X_{\rm ord,\rm str}(E_\bullet(\alpha^{-1})/K_\infty^+\bigr)\supset\ch_{\Lambda_K^+}\bigl(\mathfrak{S}_{\rm ord,\rm rel}(E_\bullet(\alpha)/K_\infty^+)/\Lambda_K^+BF_{\alpha}^+\bigr)
\]
in $\Lambda_K^+\otimes\Q_p$.
\end{thm}

\begin{proof}
As a consequence of the cyclotomic Euler system constructed in \cite[Thm.\,8.1.3]{explicit} attached to the pair $(f,\boldsymbol{g})$, there exists an integer $r\ge 0$ such that $\varpi^rBF_\alpha^+$ 
extends to a system of cohomology classes
\[
BF_{\alpha,m}^+\in\rH^1(K(\mu_m),(T_pE_\bullet)(\alpha)\widehat\otimes\Lambda_K^+)
\]
indexed by integers $m\geq 1$ coprime to $pcND_K$ for an auxiliary integer $c>1$ coprime to $6ND_K$ with $BF_{\alpha,1}^+=\varpi^rBF_\alpha^+$ and satisfying the Euler system norm relations. (The $\varpi^r$ appears because the specialization at the character $\alpha$ of the Galois module
associated to $\boldsymbol{g}$ in \cite{explicit} may not equal ${\rm Ind}_K^\Q(\alpha)$ but only contain this lattice with finite index.) 
%
Thus by the results of \cite[\S{12}]{explicit}, giving in particular a refinement of the results of \cite{rubin-ES} for Euler systems with a non-trivial local condition at $p$, it suffices to verify that the $G_\Q$-module $T:=(T_pE_\bullet)\otimes_{\Zp}{\rm Ind}_K^\Q(\alpha)$ satisfies the following hypotheses:
\begin{itemize}
\item[(i)] { $V=T\otimes\Q_p$ is irreducible as a $\Phi[{\rm Gal}(\overline{\Q}/\Q(\mu_{p^\infty}))]$-module,}
\item[(ii)] There exists an element $\sigma\in{\rm Gal}(\overline{\Q}/\Q(\mu_{p^{\infty}}))$ such that ${\rm dim}_{\Phi}(V /(\sigma-1) V)=1$.
\end{itemize}
{Here, as before, $\Phi/\Q_p$ is a finite extension contaning the values of the character $\alpha$.}
(Note that \cite[Thm.~12.3.4]{explicit} gives a refinement of \cite[Thm.~II.3.3]{rubin-ES} under the big image hypothesis ``Hyp$(BI)$'' in \emph{op.\,cit.}; the same methods yield a corresponding refinement of \cite[Thm.~II.3.4]{rubin-ES} under the above hypotheses (i)-(ii).)
{
It remains to verify hypotheses (i) and (ii).  

We first explain that $E_\bullet$ does not have CM. Suppose to the contrary that $E_\bullet$ has CM by an order
in an imaginary quadratic field $\mathcal{K}$. We note that $p$ must be unramified in $\mathcal{K}$, as $E_\bullet$ is assumed to have good reduction at $p$:
$V_\ell E_\bullet = T_\ell E_\bullet\otimes_{\Z_\ell}{\Q_\ell}$ is the induction from $\Gal(\overline{\Q}/\mathcal{K})$ of the $\ell$-adic Galois-representation associated
with a Hecke character of $\mathcal{K}$ and so is ramified at any prime $\neq\ell$ that is ramified in $\mathcal{K}/\Q$ while also being unramified
at any prime $\neq\ell$ at which $E_\bullet$ has good reduction. 
Next, recall that $E_\bullet[p]^{ss} \simeq \mathbb{F}_p(\phi)\oplus \mathbb{F}_p(\psi)$.  As $E_\bullet$ is assumed to have CM by $\mathcal{K}$, $a_\ell(E_\bullet)=0$ for any prime $\ell$ 
of good reduction for $E_\bullet$ that is inert in $\mathcal{K}/\Q$. It follows that $\psi = \phi\eta_{\mathcal{K}}$, where $\eta_\mathcal{K}$ is the quadratic character of the extension
$\mathcal{K}/\Q$. But then $\omega = \phi\psi = \phi^2\eta_{\mathcal{K}}$ implies that both $\phi$ and $\psi$ are ramified at $p$, contradicting the assumption
that $E_\bullet$ has good ordinary reduction at $p$ and hence one of $\phi$ and $\psi$ is unramified at $p$.

Let $\rho_{E_\bullet}:G_\Q\rightarrow\Aut_{\Zp}(T_pE_\bullet) \simeq \GL_2(\Zp)$ be the usual Galois action on the $p$-adic Tate module of $E_\bullet$.  
Let $K_\infty^{cyc} = K(\mu_{p^\infty})$. As $E_\bullet$ does not have CM, it follows from Serre's open image theorem 
that $\rho_{E_\bullet}(\Gal(\overline{\Q}/K_\infty^{cyc}))$ is an open subgroup of $\SL_2(\Zp)$. In particular, 
$V_pE_\bullet = T_pE_\bullet\otimes_{\Zp}\Qp$ is an irreducible $\Gal(\overline{\Q}/K_\infty^{cyc})$-representation.
It follows that the restriction of $V = V_pE_\bullet \otimes_{\Zp}{\mathrm{Ind}}_K^\Q(\alpha)$ to $\Gal(\overline{\Q}/K_\infty^{cyc})$
decomposes as $V = V_pE_\bullet(\alpha) \oplus V_pE_\bullet(\alpha^{-1})$, with the summands being irreducible.
Let $\ell$ be a prime that ramifies in $K$. Then $\ell\neq p$ as $p$ splits in $K$ by \eqref{eq:intro-spl}, and $E_\bullet$ has good reduction at $\ell$ 
as any prime dividing the conductor of $E_\bullet$ is unramified in $K$ by \eqref{eq:intro-Heeg}.  
Let $\tau_\ell\in I_\ell \subset \Gal(\overline{\Q}/\Q(\mu_{p^\infty}))$, where $I_\ell$ denotes the inertia subgroup at any prime above $\ell$, be an element with non-trivial image in $\Gal(K/\Q)$. 
Then $\tau_\ell$ acts trivially on $V_pE_\bullet$ but maps $V_pE_\bullet(\alpha)$ isomorphically onto $V_pE_\bullet(\alpha^{-1})$ 
(and vice versa). 

Suppose $0\neq W \subset V$ is an irreducible $\Phi[\Gal(\overline{\Q}/\Q(\mu_{p^\infty}))]$-subrepresentation.  
Then $W \cap V_pE_\bullet(\alpha)$ is either $0$ or all of $V_pE_\bullet(\alpha)$. As $W$ is $\tau_\ell$-stable and $V_pE_\bullet(\alpha)$ is not, it must
be that $W\cap V_pE_\bullet(\alpha) = 0$. It follows that $W$ projects isomorphically onto $V_pE_\bullet(\alpha^{-1})$ as 
a $\Gal(\overline{\Q}/K_\infty^{cyc})$-representation. The same argument also shows that $W$ projects isomorphically onto $V_pE_\bullet(\alpha)$ as 
a $\Gal(\overline{\Q}/K_\infty^{cyc})$-representation.  But this means that $V_pE_\bullet(\alpha)\simeq V_pE_\bullet(\alpha^{-1})$ as 
$\Gal(\overline{\Q}/K_\infty^{cyc})$-representations. In turn this implies that $V_pE_\bullet \simeq V_pE_\bullet(\alpha^2)$ as 
$\Gal(\overline{\Q}/K_\infty^{cyc})$-representations. However, the determinant of $V_pE_\bullet$ is trivial on $\Gal(\overline{\Q}/K_\infty^{cyc})$
while $\alpha^4$ is not (as $\alpha$ is non-trivial and pro-$p$), so $V_pE_\bullet \not\simeq V_pE_\bullet(\alpha^2)$. This contradiction proves (i).

To establish (ii), we first choose $\sigma_1\in \Gal(\overline{\Q}/K_\infty^{cyc})$ so that 
$\rho_{E_\bullet}(\sigma_1) \simeq \left(\smallmatrix 1 & * \\ 0 & 1 \endsmallmatrix\right)$ with $*\neq 0$. This is possible
as $\rho_{E_\bullet}(\Gal(\overline{\Q}/K_\infty^{cyc}))$ is an open subgroup of $\SL_2(\Zp)$. The same is then 
true of $\rho_{E_\bullet}(\tau_\ell\sigma_1)$ as $\rho_{E_\bullet}(\tau_\ell) = 1$. 
Let $V_\alpha$ be the $\Phi$-space underlying $\mathrm{Ind}_K^\Q(\alpha)$ and let $\rho_\alpha: G_\Q\rightarrow \Aut_\Phi(V_\alpha) \simeq \GL_2(\Phi)$ denote
the action of $G_\Q$ on $V_\alpha$. Then $\det\rho_\alpha = \eta_K$, the quadratic character associated with the extension $K/\Q$.
As $\tau_\ell\sigma_1 \not\in G_K$, it follows that $\rho_\alpha(\tau_\ell\sigma_1) \simeq \left(\smallmatrix 1 & 0 \\ 0 & -1 \endsmallmatrix\right)$.
Clearly, the element $\sigma = \tau_\ell\sigma_1\in \Gal(\overline{\Q}/\Q(\mu_{p^\infty}))$ satisfies (ii).
}
\end{proof}

\section{Interlude: The rank one case and the general strategy}\label{sec:one}

In this section we give a proof of 
Theorem~\ref{thm:CYC} 
under the following two additional hypotheses:
\begin{itemize}
\item[(a)] ${\rm corank}_{\Z_p}{\rm Sel}_{p^\infty}(E/K)=1$.
\item[(b)] The restriction map
\[
{\rm Sel}_{p^\infty}(E/K)\xrightarrow{{\rm loc}_v}E(K_v)\otimes\Q_p/\Z_p
\]
is nonzero.
\end{itemize}
%
%
 
The short argument that follows, albeit independent from the discussion in the later sections, will allow us to motivate the more involved arguments needed for the proof of Theorem~\ref{thm:CYC} in general, and might provide some orientation to the reader.   

Recall that $E_\bullet$ denotes the elliptic curve in the isogeny of $E$ constructed in \cite{wuthrich-int}. 
\sk

\noindent\emph{Step 1}. Under the above additional hypotheses, by  \cite[Thm.~C]{eisenstein} the module $\X_{\rm Gr}^S(E_\bullet/K_\infty^-)$ is $\Lambda^-_K$-torsion, with 
\begin{equation}\label{eq:ac-IMC}
\ch_{\Lambda^-_K}\bigl(\X_{\rm Gr}(E_\bullet/K_\infty^-)\bigr)\Lambda_K^{-,\rm ur}=\bigl(\Lcal_p^{{\rm BDP}}(f/K)\bigr).
\end{equation}

Conditions (a) and (b) above correspond to conditions (a) and (b) in Lemma~\ref{lem:coinv} with $\alpha=1$, and so from \eqref{eq:ac-IMC} we conclude that $\rH^1_{\Fcal_{\Gr}}(K,E_\bullet[p^\infty])$ is finite, and letting $\Fcal_{\rm Gr}(E_\bullet/K_\infty^-)\in\Z_p\llbracket{T}\rrbracket$  be a characteristic power series for $\X_{\rm Gr}(E_\bullet/K_\infty^-)$ we have
\begin{equation}\label{eq:ac-0}
\Fcal_{\rm Gr}(E_\bullet/K_\infty^-)(0)\;\sim_p\;\Lcal_p^{\rm BDP}(f/K)(0)\neq 0,
\end{equation}
where $\sim_p$ denotes equality up to a $p$-adic unit.
\sk

\noindent\emph{Step 2}. From Kato's result \cite[Thm.~17.4]{kato-euler-systems} (as refined in \cite[Thm.~16]{wuthrich-int} to an integral divisibility in the $p$-Eisenstein case) applied to $E_\bullet$ and $E_\bullet^K$, together with Proposition~\ref{prop:comp-Lcyc} and Proposition~\ref{prop:comp-Selcyc}, we deduce that $\X_{\rm ord}(E_\bullet/K_\infty^+)$ is $\Lambda_K^+$-torsion, and that we have the divisibility
\begin{equation}\label{eq:cyc-ord}
\ch_{\Lambda_K^+}\bigl(\X_{\rm ord}(E_\bullet/K_\infty^+)\bigr)\supset\bigl(\Lcal_p^{\rm PR}(E_\bullet/K)^{+}\bigr)
\end{equation}
in $\Lambda_K^+$. Moreover, the nonvanishing of $\Lcal_p^{\rm PR}(E_\bullet/K)^+$ follows from Rohrlich's nonvanishing results \cite{rohrlich-cyc}, while that of $\Lcal_p^{\rm Gr}(f/K)^+$ follows from (\ref{eq:ac-0}) and Proposition~\ref{prop:comp-Lac}, noting that 
\[
\Lcal_p^{\rm Gr}(f/K)^-(0)=\Lcal_p^{\rm Gr}(f/K)^+(0).
\] 
Therefore, by Proposition~\ref{prop:equiv} with $\alpha=1$ we obtain that $\X_{\rm Gr}(E_\bullet/K_\infty^+)$ is $\Lambda_K^+$-torsion, with the divisibility 
\begin{equation}\label{eq:cyc-Gr}
\ch_{\Lambda_K^+}\bigl(\X_{\rm Gr}(E_\bullet/K_\infty^+)\bigr)\supset\bigl(\Lcal_p^{\rm Gr}(f/K)^{+}\bigr)
\end{equation}
in $\Lambda_K^{+,\rm ur}$.
\sk

\noindent\emph{Step 3}. For $\Fcal_{\rm Gr}(E_\bullet/K_\infty^-)\in\Z_p\llbracket{T}\rrbracket$ a characteristic power series for $\X_{\rm Gr}^S(E_\bullet/K_\infty^-)$, we have the chain of relations 
\[
\Fcal_{\rm Gr}(E_\bullet/K_\infty^+)(0)\;\sim_p\;
\Fcal_{\rm Gr}(E_\bullet/K_\infty^-)(0)\;\sim_p\;\Lcal_p^{\rm BDP}(f/K)(0)\;\sim_p\;\Lcal_p^{\rm Gr}(f/K)^{-}(0)=\Lcal_p^{\rm Gr}(f/K)^+(0),
\]
using Proposition~\ref{prop:euler-char} with $\alpha=1$ (resp. Proposition~\ref{prop:factor-L-Gr}) for the first (resp. third) equality up to a $p$-adic unit. 
%
We thus conclude that
\[
\Fcal_{\rm Gr}(E_\bullet/K_\infty^+)(0)\;\sim_p\;\Lcal_p^{\rm Gr}(f/K)^{+}(0)\neq 0,
\]
which by easy commutative algebra (see \cite[Lem.~3.2]{skinner-urban}) implies that equality holds in (\ref{eq:cyc-Gr}):
\[
\ch_{\Lambda_K^+}\bigl(\X_{\rm Gr}(E_\bullet/K_\infty^+)\bigr)=\bigl(\Lcal_p^{\rm Gr}(f/K)^{+}\bigr).
\]
By Proposition~\ref{prop:equiv}, it follows that $\X_{\rm ord}(E_\bullet/K_\infty^+)$ is $\Lambda_K^+$-torsion, with
\[
\ch_{\Lambda_K^+}\bigl(\X_{\rm ord}(E_\bullet/K_\infty^+)\bigr)=\bigl(\Lcal_p^{\rm PR}(E_\bullet/K)^+\bigr),
\]
and by Proposition~\ref{prop:isog-inv} the same conclusion holds with $E$ in place of $E_\bullet$. By Proposition~\ref{prop:comp-Lcyc} and Proposition~\ref{prop:comp-Selcyc}, this equality of characteristic ideals together with Kato's divisibility for $E$ yields Theorem~\ref{thm:CYC} in this case. 
\sk

In order to obtain Theorem~\ref{thm:CYC} in general:

\begin{itemize}
\item We shall prove Theorem~\ref{thm:AC} from the Introduction, removing the assumption ${\rm corank}_{\Z_p}{\rm Sel}_{p^\infty}(E/K)=1$ from  \cite[Thm.~C]{eisenstein}. 
Then, 
similarly as in \emph{Step 1} above, we shall obtain
\begin{equation}\label{eq:ac-alpha}
\Fcal_{\rm Gr}(E_\bullet(\alpha)/K_\infty^-)(0)\;\sim_p\;\Lcal_p^{\rm BDP}(f(\alpha)/K)(0)\neq 0,\nonumber
\end{equation}
for any character $\alpha$ of $\Gamma_K^-$ 
away from the zeros of $\Lcal_p^{\rm BDP}(f/K)$ (so necessarily $\alpha\neq 1$ if the $\Z_p$-corank of ${\rm Sel}_{p^\infty}(E/K)$ is greater than $1$).

\item By arguments similar to those in \emph{Steps 2} and \emph{3} above, but complicated by the need to apply certain congruences and the use of $S$-imprimitive Selmer groups, we will show that for $\alpha$ sufficiently close to $1$, the module $\X_{\rm ord}(E_\bullet(\alpha)/K_\infty^+)$ is $\Lambda_K^+$-torsion, with
\[
\ch_{\Lambda_K^+}\bigl(\X_{\rm ord}(E_\bullet(\alpha)/K_\infty^+)\bigr)=\bigl(\Lcal_p^{\rm PR}(E_\bullet(\alpha)/K)^+\bigr).
\]
%
%
\end{itemize}

From this last equality (taking $\alpha$ to be sufficiently close to $1$), we deduce that the original $\mathfrak{X}_{\rm ord}(E_\bullet/K_\infty^+)$ and $\Lcal_p^{\rm PR}(E_\bullet/K)^+$ have the same Iwasawa invariants, which together with Kato's work will yield the proof of Theorem~\ref{thm:CYC}. 

\section{Anticyclotomic main conjecture}\label{sec:anticyc}

The key new result in this section is Theorem~\ref{thm:Zp-twisted}. The result is a Kolyvagin system bound complementing \cite[Thm.~3.2.1]{eisenstein} for characters $\alpha$ of $\Gamma_K^-$ that are close to $1$. We then use this result to obtain a version of \cite[Thm.~C]{eisenstein} and its corollaries removing the assumption that ${\rm corank}_{\Z_p}{\rm Sel}_{p^\infty}(E/K)=1$ (i.e., assumption (Sel) in \emph{loc.\,cit.}).  

Throughout this section, we let $E/\Q$ be an elliptic curve of conductor $N$, $p\nmid 2N$ be a prime of good ordinary reduction for $E$, and $K$ be an imaginary quadratic field of discriminant $D_K$ prime to $Np$. We assume that
\begin{equation}\label{eq:h1}
E(K)[p]=0.\tag{h1}
\end{equation}

\subsection{A Kolyvagin-style bound for $\alpha\equiv 1\;({\rm mod}\,\varpi^m)$}
Let $R$ be the ring of integers of some finite extension $\Phi/\Qp$ and let $\varpi\in R$ be a uniformizer.
Let $\fm = (\varpi)$ be the maximal ideal of $R$.
Let $\alpha:\Gamma_K^- \to R^{\times}$ be an anticyclotomic character and $m$ a positive integer such that 
\begin{equation}\label{eq:pm}
\alpha\equiv 1\;({\rm mod}\,\varpi^m).
\end{equation}

Denote by $\rho_E:G_\Q\rightarrow{\rm Aut}_{\Z_p}(T_pE)$ the representation on the $p$-adic Tate module of $E$, and consider the $G_K$-modules 
\begin{equation}
T:=T_pE\otimes_{\Z_p}R(\alpha),\quad
V:=T\otimes_R \Phi,\quad A:=T\otimes_R\Phi/R \simeq V/T,\nonumber
\end{equation}
where $R(\alpha)$ is the free $R$-module of rank one on which $G_K$ acts via the composition of the projection $G_K\twoheadrightarrow \Gamma_K^-$ with
$\alpha$,  and the $G_K$-action on $T$ is via $\rho = \rho_E\otimes \alpha$.

Let $\mathscr{L}=\mathscr{L}_E$ be the set of primes $\ell\nmid N$ that are inert in $K$ and satisfy $a_\ell\equiv\ell+1\equiv 0\pmod{p}$, where $a_\ell=\ell+1-\vert\tilde{E}(\mathbb{F}_\ell)\vert$, and let $\mathscr{N}$ denote the set of square-free products of primes $\ell\in\mathscr{L}$. For each $\ell\in\mathscr{L}$, let $I_{\ell}\subset\Z_p$ be the smallest ideal containing $\ell+1$ for which the Frobenius element $\frob_{\lambda} \in G_{K_{\lambda}}$ acts trivially on $T / I_{\ell} T$, where $\lambda\mid \ell \in \mathscr{L}$. 
For any $k\geqslant 0$, let 
\[
T^{(k)}=T/\varpi^kT,\quad\cL^{(k)}=\{\ell\in\cL\colon I_\ell\subset p^k\Z_p\},
\]
and let $\mathscr{N}^{(k)}$ be the set of square-free products of primes $\ell\in\cL^{(k)}$. We refer the reader to \cite[\S{3.1}]{eisenstein} for the definition of the module of Kolyvagin systems $\mathbf{KS}(T,\Fcal_{\rm ord},\cL)$ associated to the triple $(T,\Fcal_{\rm ord},\cL)$ (here and until $\S\ref{subsec:ac-IMC}$, $\Fcal_{\rm ord}$ denotes the ordinary Selmer structure introduced in \cite[p.\,548]{eisenstein}, which is compatible with the discrete coefficients analogue defined in $\S$\ref{sec:seltwist} and Definition \ref{defilocalconds}). 

\begin{thm}\label{thm:Zp-twisted}  
There exist non-negative integers $\CM$ and $\CE$ depending only on $T_pE$ and $\operatorname{rank}_{\Z_p}(R)$ such that if
$m\geq \CM$ and if there is a Kolyvagin system $\kappa=\{\kappa_{n}\}_{n\in\cN}\in\mathbf{KS}(T,\Fcal_{\rm ord},\cL)$ with $\kappa_{1}\neq 0$,
then $\rH^1_{\Fcal_{\rm ord}}(K,T)$ has $R$-rank one and there is a finite $R$-module $M$ such that
\[
{\rm H}^1_{\Fcal_{\rm ord}}(K,A)\simeq \Phi/R \oplus M\oplus M
\]
with 
\[
{\rm length}_R(M)\leqslant {\rm length}_R\bigl({\rm H}^1_{\Fcal_{\rm ord}}(K,T)/R\cdot\kappa_{1}\bigr)+\CE.
\]
\end{thm}
\noindent The `error term' $\CE$ in this theorem is independent of $m$, but that comes at the expense of the result applying only to characters $\alpha$ that are sufficiently close to $1$ (as measured by $\CM$). The reader may wish to compare this theorem
with \cite[Thm.~3.2.1]{eisenstein} whose error term $E_\alpha$ is at least as large as $m$ but which also applies to $\alpha$ that are relatively far from $1$. Both results are crucial for the proof of Theorem \ref{thm:AC}.

\subsection{Structure of Selmer groups}

In the following we use $\Fcal$ to denote $\Fcal_{\rm ord}$ for simplicity. For any $k\geq 1$, let $R^{(k)}= R/\fm^k$.
We recall the following structure results:
\begin{lemma}
\label{lemmamod} 
For every $n\in\cN^{(k)}$ and $0\leqslant i\leqslant k$ there are natural isomorphisms
\begin{equation}\label{eq2}
\rH^1_{\Fcal(n)}(K,T^{(k)}/\fm^iT^{(k)})\xrightarrow{\sim}\rH^1_{\Fcal(n)}(K,T^{(k)}[\fm^i])\xrightarrow{\sim}{\rm H}_{\Fcal(n)}^1(K,T^{(k)})[\fm^i]\nonumber
\end{equation}
induced by the maps $T^{(k)}/\fm^iT^{(k)}\xrightarrow{\pi^{k-i}}T^{(k)}[\fm^i]\rightarrow T^{(k)}$.
\end{lemma}

\begin{proof} See \cite[Lem.~3.3.1]{eisenstein}.
\end{proof}

\begin{prop}
\label{propstructure} There is an integer $\epsilon\in\{0,1\}$ such that for all $k$ and every
every $n\in\cN^{(k)}$ there is an $R^{(k)}$-module $M^{(k)}(n)$ such that
\begin{equation}\label{eq1}
\rH^1_{\CF(n)}(K,T^{(k)}) \simeq (R/\fm^k)^\epsilon\oplus M^{(k)}(n)\oplus M^{(k)}(n).\nonumber
\end{equation}
\end{prop}

\begin{proof} See \cite[Prop.~3.3.2]{eisenstein}.
\end{proof}


By Lemma \ref{lemmamod} and \eqref{eq:pm}, if $k\geq m$ there is an isomorphism
\begin{equation}\label{eq:pmtorsion}
\rH^1_{\Fcal(n)}(K,T_E^{(m)})\simeq \rH^1_{\Fcal(n)}(K,T^{(k)})[\fm^{m}],
\end{equation}
where $T_E^{(m)}=T_p(E)\otimes_{\Z_p} R/\fm^m$.  We can then exploit the action of complex conjugation on the left hand side
(both $T_E^{(m)}$ and the Selmer structure $\Fcal(n)$ are stable under this action).   We make use of this in our subsequent analysis
of the structure of the $R$-modules $M^{(k)}(n)$ in terms of Kolyvagin classes.


\subsection{The \v{C}ebotarev argument}\label{seccheb}
We recall the definitions of the error terms $C_1,C_2$ of \cite[$\S$3.3.1]{eisenstein}. For $U = \Z_p^\times\cap \mathrm{im}(\rho_E)$ let
\[
C_1 := \min\{v_p(u-1)\colon u\in U\}.
\]
As $U$ is an open subgroup, $C_1<\infty$.
Recall also that $\End_{\Z_p}(T_pE)/\rho_E(\Z_p[G_{\Q}])$ is a torsion $\Z_p$-module and let
\[
C_2:=\min\bigl\{ n\geqslant 0 \colon p^n\End_{\Z_p}(T_pE)\subset\rho_E(\Z_p[G_{\Q}])\bigr\}.
\]
Let $r=\operatorname{rank}_{\Z_p}R$ and 
\[
e:=r(C_1+C_2).
\]

For any finitely-generated torsion $R$-module $M$ and $x\in M$, let
$$
\ord(x):=\min\{m\geqslant 0: \varpi^m\cdot x =0\} \ \ \text{and } \ \ \ \exp(M):=\max\{\ord(x):x\in M\}.
$$
The following result 
is one of the main tools for our proof of Theorem \ref{thm:Zp-twisted}. 
\begin{prop}\label{prop:prime2} 
Let $c^\pm \in \rH^1(K,T_E^{(m)})^\pm$.  Let $k\geq m$. 
Then there exist infinitely many primes $\ell\in \cL^{(k)}$ such that 
$$
\ord(\loc_\ell(c^\pm)) \geqslant \ord(c^\pm) - e.
$$
In particular, $R\cdot \loc_\ell(c^+)+ R\cdot \loc_\ell(c^-)$ has an $R$-submodule isomorphic to  
$$
R/\fm^{\max\{0,\ord(c^+)-e\}}\oplus  R/\fm^{\max\{0,\ord(c^-)-e\}}.
$$
\end{prop}


\begin{proof}
The proof of this proposition follows along the lines of that of \cite[Prop. 3.3.6]{eisenstein}. 

Let $u\in  \Z_p^\times\cap \mathrm{im}(\rho_E|_{G_{K_\infty}})= \Z_p^\times\cap \mathrm{im}(\rho_E)\subset \Z_p^\times\cap \mathrm{im}(\rho_E \otimes \alpha)$ (see \cite[Lemma 3.3.3]{eisenstein}) be such that $v_p(u-1) = C_1$.   Let $L$ be 
the composite of the fixed field of the action
of $G_K$ on $T_E/p^kT_E$ and the field $K_\alpha$ trivialising $\alpha \mod p^k$.  Then there is some $h$ in the center of $\Gal(L/K)$ 
such that $h$ acts on $T/p^kT$, and hence on $T^{(m)}\simeq T_E^{(m)}$, as multiplication by $u$. The kernel of the 
restriction map $\rH^1(K,T^{(m)})\rightarrow \rH^1(L,T^{(m)})$ is $\rH^1(L/K,T^{(m)})$ and it follows from the existence
of $h$ that the latter is annihilated by $u-1$ (this is essentially Sah's Lemma: if $c:\Gal(L/K)\rightarrow T^{(m)}$ is 
a $1$-cocycle, then $c(gh) = c(hg)$ and so $(h-1)c(g) = (g-1)c(h)$ and hence $p^{C_1}c$ is a coboundary). It follows that 
\begin{equation}\label{eq:kerm}
p^{C_1}\cdot \ker \bigl( \rH^1(K,T_E^{(m)})\to \rH^1(L, T_E^{(m)})\bigr) = 0.
\end{equation}


Let $d^\pm:= \ord(c^{\pm}) - r(C_1 + C_2)$. If $d^+=d^-\leq 0$, then there is nothing to prove. So assume at least one of $d^\pm$ is positive. 
By (\ref{eq:kerm}), the kernel of the restriction map
$$
\rH^1(K,T_E^{(m)}) \stackrel{res}{\rightarrow} \rH^1(L,T_E^{(m)}) = \Hom_{G_K}(G_L,T_E^{(m)})
$$
is annihilated by $p^{C_1}=\varpi^{rC_1}$. Let $f^\pm\in \Hom_{G_K}(G_L,T_E^{(m)})$ be the image of $c^\pm$. We then have 
$$
\ord(f^\pm) \geq \ord(c^\pm) - rC_1.
$$
As $f^\pm(G_L)$ is a $G_K$-submodule, $f^\pm(G_L) = \im(\rho_E)\cdot f^\pm(G_L)$ and so, by the definition of $C_2$,
the image of $f^\pm$ contains $p^{C_2}\operatorname{End}(T_p(E))\cdot f^\pm(G_L)$. 
Since $\ord(f^\pm) \geq \ord(c^\pm) - rC_1$, it follows that the $R$-span of the image of $f^\pm$ contains $\varpi^{m-\ord(c^\pm) + r(C_1+C_2)} T_E^{(m)}$. 
Since at least one of $d^+$ and $d^-$ is positive, 
it follows that at least one of $f^+$ and $f^-$ is non-trivial. 

Let $H\subset G_L$ be the intersection of the kernels of $f^+$ and of $f^-$, and let $Z = G_L/H$. Note that $H\neq G_L$ since 
some $f^\pm$ is non-trivial, so $Z$ is a non-trivial torsion $R$-module. Note also that $Z$ is stable under the action of complex conjugation since each $f^\pm$ is.  In particular, $Z$ decomposes into eigenspaces under the action of complex conjugation: $Z= Z^+\oplus Z^-$. 

Let  $g^\pm$ be the projection of $f^\pm$ to the summand $(T_E^{(m)})^{\pm} \cong R/\fm^m$.
Then the $R$-span of the image of $g^\pm$ contains an $R$-submodule isomorphic to $R/\fm^{\max\{0,d^\pm\}}$. 
We have $g^\pm(Z^-)=0$ since $f^\pm\in \Hom(G_L,E[p^m])^{\pm}$, which means $f^\pm(Z^-)\subset E[p^m]^{\mp}$. So we find $g^{\pm}(Z) = g^{\pm}(Z^+)$ and
that the $R$-span of $g^\pm(Z^+)$ contains a submodule isomorphic to $R/\fm^{\max\{0,d^\pm\}}$. It follows
that $Z^+$ is non-trivial.

If $d^\pm>0$, let $W_\pm \subset Z^+$ be the proper subgroup such that $g^\pm(W_\pm) = \varpi^{m-(d^\pm-1)}(T_E^{(m)})^\pm$. 
If $d^\pm \leq 0$, let $W_\pm  = 0$. Then both $W_+$ and $W_-$ are proper subgroups of $Z^+$ (since there exists some $z\in Z^+$ such that $g^\pm(z)\in \varpi^{m-d^\pm}(T_E^{(m)})^\pm$). It follows that
$W_+\cup W_-\neq Z^+$.  Let $z\in Z^+$, $z\not\in W_+\cup W_-$. By the definition of $W^\pm$, we have
\begin{equation}\label{eq:ordnek}
\ord(g^\pm(z)) \geq d^\pm.
\end{equation}

Let $M =  \overline{\Q}^{H}$, so $\operatorname{Gal}(M/L) = Z$. Let $g = \tau z \in G_\Q$, and let 
$\ell\nmid Np$ be any prime such that both $c^+$ and $c^-$ are unramified at $\ell$ and $\frob_\ell = g$ in $\operatorname{Gal}(M/\Q)$. 
The \v{C}ebotarev density theorem implies there are infinitely many such primes.
 Since $Z$ fixes $L$ and since $L$ contains the fixed field of the $G_K$-action on $E[p^k]$,
$\frob_\ell$ acts as $\tau$ on both $E[p^k]$ and $K$.
This means that $a_\ell(E) \equiv \ell +1 \equiv 0 \mod p^k$ and $\ell$ is inert in $K$. 
Since $L$ also contain the fixed field of $\alpha\mod p^k$, for $\lambda\mid \ell$ a prime of $K$, it follows that  $\frob_\lambda$
acts trivially on $T/p^kT$ and hence that $\ell\in \mathscr{L}^{(k)}$. 

Since $\ell$ is inert in $K$, the Frobenius element of $\ell$ in $\operatorname{Gal}(\bar{\Q}/K)$ is $\frob_\ell^2$. Consider the restriction of
$c^\pm$ to $K_\ell$. Since $c^\pm$ is unramified at $\ell$, $\loc_\ell(c^\pm)$ is completely determined by the image $c^\pm(\frob_\ell^2)$ in $T_E^{(m)}/(\frob_\ell^2-1)T_E^{(m)}$. By the choice of $\ell$, $\frob_\ell^2$ acts trivially on
$T_E^{(m)}$, so  $T_E^{(m)}/(\frob_\ell^2-1)T_E^{(m)} = T_E^{(m)}$. Moreover, $\frob_\ell^2 = g^2  = z^2 \in \operatorname{Gal}(M/L)$,
so $c^\pm(\frob_\ell^2) = f^\pm(z^2) = 2g^\pm(z)$, where the second equality follows from the fact that the projection of $f^\pm$ to $(T_E^{(m)})^\mp$ maps $z\in Z^+$ to zero. Since $p$ is odd, (\ref{eq:ordnek}) yields $\ord(\loc_{\ell}(c^\pm)) = 
\ord(c^\pm(\frob_\ell^2)) = \ord(2g^\pm(z)) = \ord(g^\pm(z)) \geq d^\pm$.
\end{proof}

\begin{rmk}\label{rmkcalpha}
The primary difference between Proposition \ref{prop:prime2} and \cite[Prop.~3.3.6]{eisenstein} is  that here we have restricted ourselves to the $\fm^m$-torsion of the Selmer groups of $T^{(k)}$  (see (\ref{eq:pmtorsion})) and so we can directly work with the eigenspaces of complex conjugation.
For the proof of {\em loc.~cit.} we worked over an extension trivialising the character $\alpha \mod \varpi^k$ and then used ``some quadratic forms'' to 
estimate the linear independence of the images of the localisations of the classes.  The upshot is that our error term no longer involves
the $C_{\alpha}$ of \cite{eisenstein}. This is crucial for removing the corank one assumption in the proof of the anticyclotomic Iwasawa main conjecture in {\em op.~cit.}. However it causes some complications in the proof of Theorem \ref{thm:Zp-twisted}: we use the $\fm^m$-Selmer groups to control the 
image of the localisation at Kolyvagin primes of classes in the $\fm^k$-Selmer groups, and the resulting control is not as tight as in \cite{eisenstein}.
 \end{rmk}

\subsection{Proof of Theorem \ref{thm:Zp-twisted}}\label{secproofkoly}
The (co-)rank one claim in the theorem follows from \cite[Thm.~3.3.8]{eisenstein}: 
$$
\rH^1_\CF(K,T)\simeq R \ \ \text{and} \ \ \rH^1_\CF(K,A) \simeq \Phi/R \oplus M, \ \ M \simeq M_0\oplus M_0
$$
for some finitely-generated torsion $R$-module $M_0$ such that $M_0\simeq M^{(k)}(1)$ for all $k\gg 0$.
In fact, the proof of  \cite[Thm.~3.3.8]{eisenstein} shows that $M_0\simeq M^{(k)}(1)$ if $k>{\ind(\kappa_{1})+3r(C_1+C_2+m)}$.
In the current setting, the error term $C_{\alpha}$ of \emph{op. cit.} is equal to $m$; it is essentially this fact that prevents the arguments in {\em op.~cit.} from applying to prove the theorem in the current setting and is the reason we take a different approach below
to establish the bound 
\begin{equation}\label{eq:bound}
\tag{B}
s_1+\CE\geq {\rm length}_{R}(M^{(k)}(1)),
\end{equation}
where $s_1=\ind(\kappa_{1},\rH^1_{\Fcal}(K,T))$ and $\CE$ does not depend on $m$, provided $m$ is sufficiently large. 
As $\mathrm{\length}_R(M) = \mathrm{length}_R(M^{(k)}(1))$ for $k\gg 0$, the bound \eqref{eq:bound} implies 
the bound in Theorem \ref{thm:Zp-twisted}.

We now focus on the proof of \eqref{eq:bound}. 
%
A finite torsion $R$-module $X$ is isomorphic to a sum of cyclic $R$-modules:  $X \simeq \oplus_{i=1}^{s(X)} R/\fm^{d_i}$ for some uniquely-determined
integers $d_i\geq 0$. For an integer $t\geq 0$ we let $\rho_t(X)= \# \{i: d_i > t\}$.  In particular, for $n\in\cN^{(m)}$ we let
\[
\rho_t(n): = \rho_t(\rH^1_{\F(n)}(K, T^{(m)})^+) + \rho_t(\rH^1_{\F(n)}(K,T^{(m)})^-).
\]
Note that it $t<m$ then $\rho_t(n)\geq 1$ since the $\epsilon$ of Proposition \ref{propstructure} is $1$ (the latter fact is implicit in the above rank one result).
We also let
$$
\rho: = 2(\rho_e(1)-1).
$$
Note that $\rho< 2\dim_\mathbb{F}(\rH^1_{\Fcal}(K,T^{(1)}) = 2\dim_{\mathbb F_p}(\rH^1_\Fcal(K,E[p])$, where $\mathbb{F}= R/\fm$ is the residue field of $R$, and hence $\rho$ is bounded by a constant independent of $k$, $m$, and $\alpha$.


\begin{proof}[Proof of \eqref{eq:bound}]
Let $s(n)= \dim_{\mathbb{F}}\rH^1_{\F(n)}(K,T^{(1)})-1 = \dim_\mathbb{F}M(n)[\fm]$, and let
\[
\CE = (\rho + (s(1) +1+ 2\rho)(5\rho+1))e \ \ \text{and} \ \ \CM = (1+5\rho)e.
\]
Note that $\CE$ and $\CM$ are bounded by constants that are independent of $m$ (and depend on $\alpha$ only through the $\Zp$-rank $r$ of $R$). 
Let $k$ be a fixed integer such that 
\begin{equation}\label{eq:kbig}
k > \mathrm{length}_{R}(M(1)) + \ind(\kap_1)+m + (6 s(1)+ 2) e.
\end{equation}
We will show that \eqref{eq:bound} holds provided 
$$
m > \CM.
$$

%

If $\rho = 0$, then the exponent of $M(1)$ is at most $e$ and therefore
\[
\tfrac{1}{2} \mathrm{length}_{R}(M(1)) \leq \tfrac{1}{2} (s(1)+1)e \leq \ind(\kappa_1) + \CE.
\]
So we may assume $\rho>0$. 
%
%
%

We will find sequences of integers $1=n_0, n_1,..., n_\rho \in \cN^{(k)}$ and $1=x_0,x_1,...,x_\rho$ along with integers $b(n_i)\geq 0$ that satisfy
\begin{itemize}
\item[(a)] $s(n_{i+1})-2 \leq s(n_{i})\leq s(n_{i+1})+2$;
\item[(b)] $\tfrac{1}{2}\mathrm{length}_{R}(M(n_i))\geq  \tfrac{1}{2}\mathrm{length}_{R}(M(n_{i-1}))- b(n_i)$ ;
\item[(c)] $\tfrac{1}{2}\mathrm{length}_{R}(M(n_i)) \leq \tfrac{1}{2}\mathrm{length}_{R}(M(n_{i-1})) + e$;
\item[(d)] $\ord(\kappa_{n_i})\geq \ord(\kappa_{n_{i-1}})-e$;
\item[(e)] $\ind(\kappa_{n_{i-1}})+e\geq \ind(\kappa_{n_i})+b(n_i)$;
\item[(f)]  $x_{i-1}+2\leq x_i\leq x_{i-1} +5$, $\rho_{x_i e}(n_i) \leq \rho_{x_{i-1}e}(n_{i-1})$, and
$\rho_{x_i e}(n_i) = \rho_{x_{i-1}e}(n_{i-1})>1$ only if $\rho_{(x_i+1)e}(\rH^1_{\F(n_{i})}(K,T^{(m)})^\pm)\geq 1$;
\item[(g)] if $\rho_{x_{i-2}e}(n_{i-2})>1$ then 
$\rho_{x_{i} e}(n_{i}) < \rho_{x_{i-2}e}(n_{i-2})$.
\end{itemize}
Before explaining the existence of such sequences, we demonstrate that \eqref{eq:bound} follows from (a)--(g).


From repeated appeals to (a) we obtain 
$$
s(n_\rho) \leq s(1) + 2\rho. 
$$
From repeated appeals to (g) we see that 
either $\rho_{x_ie}(n_i) = 1$ for some $1\leq i<\rho$, in which case $1\leq \rho_{x_\rho e}(n_\rho) \leq \rho_{x_ie} (n_i) = 1$, 
or $1\leq \rho_{x_\rho e} (n_\rho) \leq \rho_{x_{\rho-2} e} (n_{\rho-2})-1 \leq \cdots \leq \rho_{x_0 e}(n_0) - \tfrac{1}{2}\rho = \rho_e(1) - \tfrac{1}{2}\rho = 1$.
In either case, we see
$$
\rho_{x_\rho e}(n_\rho) = 1.
$$
As $x_\rho \leq x_0 + 5\rho = 1 +5\rho$, it then follows that 
\begin{equation}\label{eq:length1}
\mathrm{length}_R(M(n_\rho)) \leq s(n_\rho)x_\rho e  \leq s(n_\rho)(5\rho+1) e \leq (s(1) + 2\rho)(5\rho+1) e.
\end{equation}
By repeatedly applying (b) and (e) we obtain
\begin{equation*}\begin{split}
\ind(\kap_1) + \rho e &  \geq \ind(\kap_{n_\rho}) + b(n_1) + b(n_2) + \cdots + b(n_\rho)  \\
&\geq \ind(\kap_{n_\rho}) + \tfrac{1}{2}\mathrm{length}_R(M(1)) - \tfrac{1}{2}\mathrm{length}_R(M(n_\rho)).
\end{split}
\end{equation*}
Combined with \eqref{eq:length1} this gives
\begin{equation*}\begin{split} 
\ind(\kap_1) + \CE & \geq \ind(\kap_{n_\rho}) + \tfrac{1}{2}\mathrm{length}_R(M(1)) - \tfrac{1}{2}\mathrm{length}_R(M(n_\rho)) + (s(1)+2\rho)(5\rho+1)e \\
& \geq \ind(\kap_{n_\rho}) + \tfrac{1}{2}\mathrm{length}_R(M(1)) + \tfrac{1}{2}\mathrm{length}_R(M(n_\rho))  \\
& \geq \tfrac{1}{2}\mathrm{length}_R(M(1)),
\end{split}
\end{equation*}
which is the bound \eqref{eq:bound}.


We will now define the sequence $n_0=1, n_1, \dots, n_T \in \mathscr{N}^{(k)}$ (and subsequently the $b(n_i)$ and $x_i$) by making repeated use of Proposition \ref{prop:prime2} to choose suitable primes in $\mathscr{L}^{(k)}$. 

Suppose $1=n_0, n_1, \dots, n_j \in \mathscr{N}^{(k)}$, $1=x_0, x_1, \dots, x_j$, and $0=b(n_0), b(n_1), \dots, b(n_j)$, $j<\rho$, are such that (a)--(f) hold for all $1\leq i\leq j$ (note that if $j=0$ then (a)--(f) are vacuously true).  
We will explain how to choose a prime $\ell\in \mathscr{L}^{(k)}$ such that $n_0,\dots,n_j,n_{j+1}=n_j\ell$
satisfy (a)--(f) for all $1\leq i\leq j+1$. Repeating this process yields the desired sequence $n_0,\dots,n_\rho$.

Let $c_0\in \rH^1_{\F(n_j)}(K,T^{(k)})$ generate an $R/\fm^{k}$-summand complementary to $M(n_{j})$, so
$\rH^1_{\F(n_j)}(K,T^{(k)}) = R \cdot c_0 \oplus M(n_j) \simeq R/\fm^k\oplus M(n_j)$.
Let $\nu \in \{\pm \}$ such that $\ord((1+\nu\tau)\varpi^{k-m}c_0)=m$, that is, the order of
$c^{\nu}= (1+\nu\tau)(\varpi^{k-m}c_0) \in \rH^1_{\F(n_j)}(K,T^{(m)})^{\nu}$ is $m$. Note that the order of $\varpi^{k-m}c_0$ is $m$, so there exists at least one $\nu \in \{\pm \}$ satisfying the desired condition.
%
Let $N: = \operatorname{exp} ( \rH^1_{\Fcal(n_j)}(K,T^{(m)})^{-\nu})$ and 
let $c^{-\nu} \in \rH^1_{\Fcal(n_j)}(K,T^{(m)})^{-\nu}$ have order $N$. 
%
%
We apply Proposition \ref{prop:prime2} to the classes $c^{\nu}$ and  $c^{-\nu}$ and obtain a prime $\ell\in \mathscr{L}^{(k)}$ such that
\[
\ord(\loc_{\ell}(\varpi^{k-m}c_0))\geq  \ord(\loc_{\ell}(c^{\nu}))\geq m-e \ \ \text{and} \ \ \ord(\loc_{\ell}(c^{-\nu}))\geq N-e,
\]
and $\loc_{\ell}( \rH^1_{\Fcal(n_j)}(K,T^{(m)}))$ has an $R$-submodule isomorphic to $R/\fm^{m-e}\oplus R/\fm^{N-e}$.
%
As $m>e$, it follows that $\ord(\loc_{\ell}(c_0))\geq k-e$, and
hence $\loc_\ell(\rH^1_{\F(n_j)}(K,T^{(k)})) \simeq R/\fm^{k-a}\oplus R/\fm^{b}$ for some $a\leq e$ and $b\geq N-e$.
In particular, there is a short exact sequence
\begin{equation}\label{eq:imageloc}
0\to H \to \rH^1_{\Fcal(n_j)}(K,T^{(k)}) \xrightarrow{\loc_{\ell}} R/\fm^{k-a}\oplus R/\fm^{b}\to 0, \ \ \ a\leq e, \ b\geq N-e,
\end{equation}
where $H :=\rH^1_{\Fcal(n_j)_{\ell}}(K,T^{(k)}))$ is the kernel of the localisation at $\ell$. 
Global duality then implies that there is an exact sequence
\begin{equation}\label{eq:dualexactseq}
0 \to H \to  \rH^1_{\Fcal(n_j\ell)}(K,T^{(k)})\simeq R/\fm^k \oplus M(n_j\ell)\xrightarrow{\loc_{\ell}} R/\fm^{k-b'}\oplus R/\fm^{a'}\to 0, \ \ \ a\geq a', b'\geq b.
\end{equation}
Here we have used that the arithmetic dual of $T^{(k)} = T^{(k)}_\alpha$ is $T^{(k)}_{\alpha^{-1}}$ and that the complex conjugation $\tau$ induces an isomorphism $\rH^1_{\CF(n)}(K,T^{(k)}_{\alpha^{-1}})\simeq \rH^1_{\CF(n)}(K,T^{(k)}_{\alpha})$.

We now show (a)-(e) hold for $i=j+1$ with $n_{j+1}=n_j\ell$ and $b(n_{j+1}) = b'$. 

Let $h: = \dim_{\mathbb{F}} H[\fm]$. From \eqref{eq:imageloc} it follows that $h \leq 1+s(n_j) \leq h+2$, and from \eqref{eq:dualexactseq} it follows that
$h\leq 1+s(n_{j+1})\leq h+2$. Hence $s(n_j) -2 \leq h-1 \leq s(n_{j+1}) \leq h+1 \leq s(n_j) +2$, which shows that (a) holds. 

From\eqref{eq:dualexactseq} and  \eqref{eq:imageloc}  we find
\[
\length_R(M(n_{j+1})) = \length_R(H) - b' + a' = \length_R(M(n_j)) - (b+b') + (a+a').
\]
As $(b+b') \leq 2b' = 2b(n_{j+1})$, it follows that (b) holds. And as $(a+a') \leq 2e$, it follows that (c) holds.
%
%



To verify that (d) holds for $i=j+1$, we first observe that 
$$
\ord(\kap_{n_{j+1}}) = \ord(\kap_{n_{j}\ell}) \geqslant \ord(\loc_{\ell}(\kap_{n_{j}\ell}))= \ord(\loc_\ell(\kap_{n_{j}})),
$$
the last equality following from the finite-singular relations of the Kolyvagin system.
So (d) holds if $\ord(\loc_\ell(\kap_{n_{j}})) \geqslant \ord(\kap_{n_{j}})-e$.
To see that this last inequality holds, we note that
$\ord(\kap_{n_{j}}) \geq \ord(\kap_{n_0}) - je$ by (d) for $1\leqslant i \leqslant j$. 
But $\ord(\kap_{n_0}) =\ord(\kap_1)= k-\ind(\kap_1)$ by the choice of $k$ (and the fact that $\rH^1_{\CF}(K,T)$ is torsion-free by \eqref{eq:h1}), 
and so by \eqref{eq:kbig} and repeated application of (c) for $1\leq i\leq j$ we have
\begin{equation}\label{eq:ordnj}\begin{split}
\ord(\kap_{n_j}) \geq k-\ind(\kap_1) - je & >  \length_R(M(1))+m+(6s(1)+2-j)e \\
& \geq \length_R(M(n_j))+m+(6s(1)+2-3j)e   \\
& \geq  \length_R(M(n_j)) + m + 2e  \\
& \geq  \length_R(M(n_j)) + 2e.
\end{split}\end{equation}
Write $\kap_{n_{j}} = x c_0 + y$ with $x\in R$ and $y \in M(n_j)$. 
Since $\ord(\kap_{n_j}) > \exp(M(n_j))$ by \eqref{eq:ordnj}, it follows that 
$x = \varpi^t u$ for $t = k-\ord(\kap_{n_{j}})$ and some $u\in R^\times$. 
It follows that
$$
\varpi^{\exp(M(n_{j}))}\loc_\ell(\kap_{n_j}) = \varpi^{\exp(M(n_{j}))+t} u\loc_\ell(c_0).
$$
By the choice of $\ell$, $\ord(\loc_\ell(c_0)) \geqslant k - e = t+\ord(\kap_{n_j}) - e > t+{\exp(M(n_{j}))}+e$, where the last inequality follows by
\eqref{eq:ordnj}. We then deduce that
$$
\ord(\loc_\ell(\kap_{n_j})) = \ord(\loc_\ell(c_0)) - t \geq k - e -t = \ord(\kap_{n_j})-e,
$$
which -- as noted at the start of this paragraph -- implies that (d) holds.

Next we verify (e) for $i={j+1}$. Let $c_1\in \rH^1_{\CF(n_{j+1})}(K,T^{(k)})$ be a generator of an $R/\fm^k$-summand
complementary to $M(n_{j+1})$, so $\rH^1_{\CF(n_{j+1})}(K,T^{(k)}) = R\cdot c_1\oplus M(n_{j+1})\simeq R/\fm^k\oplus M(n_{j+1})$. 
Write $\kap_{n_j} = u\pi^g c_1 + y$ and $\kap_{n_{j+1}} = v\pi^h c + y'$, where $u,v\in R^\times$, $y\in M(n_{j})$ and $y'\in M(n_{j+1})$. 
Arguing as in the preceding proof that (d) holds for $i={j+1}$ shows that 
$\ord(\kap_{n_i})> \exp(M(n_i))+2e$ for $1\leqslant i\leqslant j+1$ and in particular for $i=j$ and $i=j+1$. 
Hence  $g = k - \ord(\kap_{n_j})$ and $h = k - \ord(\kap_{n_{j+1}})$.
Arguing further as in the proof that (d) holds also yields
$$
\ord(\loc_\ell(\kap_{n_j}))  = \ord(\loc_\ell(c_0))-g \ \ \text{and} \ \  \ord(\loc_\ell(\kap_{n_{j+1}})) = \ord(\loc_\ell(c_1))-h.
$$
From the finite-singular relations for the Kolyvagin system the left-hand sides of both equalities are equal and therefore
$$
h-g = \ord(\loc_\ell(c_1))-\ord(\loc_\ell(c_0)).
$$
We refer again to the short exact sequences \eqref{eq:imageloc} and \eqref{eq:dualexactseq}.
By the choice of $\ell$, 
$\ord(\loc_\ell(c_0))\geq k-e> \exp(M(n_j))\geq b$, the last inequality by \cite[Lem.~3.3.10(ii)]{eisenstein}. Hence we must have 
$\ord(\loc_\ell(c_0))=k-a$. Similarly,  we also must have $\ord(\loc_\ell(c_1))=k-b'$. Thus we find
$$
h-g =  (k-b')-(k-a) = a-b' \leq e -b'.
$$
Since $h-g = \ord(\kap_{n_j})-\ord(\kap_{n_{j+1}})$, this proves $\ord(\kap_{n_j})+b'\leq \ord(\kap_{n_{j+1}})+e$ and hence, since we have shown $\ord(\kap_{n_j})=k-\ind(\kap_{n_j})$ and $\ord(\kap_{n_{j+1}})=k-\ind(\kap_{n_{j+1}})$, (e) holds.


So far, our arguments have not wandered far from the lanes of \cite[\S 3.3.3]{eisenstein}.
However, at this point we cannot continue along the same path and deduce -- in the notation of {\em op.~cit.} (see also below) --
that $d_t(M(n_{j+1})) \geqslant d_{t+2}(M(n_{j}))$, $t=1,\dots,\dim_{\mathbb{F}}(M(n_{i})[\fm])-2$. This is because 
knowing the localisation of $\rH^1_{\Fcal(n_j)}(K,T^{(k)})[\fm^m]^{-\nu}$ is not sufficient to determine which class(es) in $\rH^1_{\Fcal(n_j)}(K,T^{(k)})$ 
have localistion generating the summand $R/\fm^b$ in \eqref{eq:imageloc}.  Instead, to conclude the proof of \eqref{eq:bound}, we establish (f) and (g). This roughly  tells us that -- even without being able to control the individual $d_t(M(n_{j}))$s -- the number of large exponents decreases. Actually, in general we can only show that this number does not increase, but if we are in the unfortunate situation where it is stable (which happens essentially if all the ``big summands'' are in the same eigenspace), then at the step $n_{j+2}$ this number decreases.

It remains to define $x_{j+1}$ and verify (f) and (g).  To this end we introduce some more notation.  A finitely-generated torsion $R$-module $X$ can be written 
as sum of cyclic $R$-modules $X\simeq \oplus_{i=1}^{s(X)} R/\fm^{d_i(X)}$, with the exponents $d_i(X)$ uniquely determined.  We shall always suppose that
the $d_i(X)$ have been labeled so that $d_1(X)\geq d_2(X) \geq \cdots \geq d_{s(X)}(X)$.  Note that $s(X) = \dim_{\mathbb{F}} X[\fm]$.
We will adopt the convention that $d_i(X)= 0$ if $i>s(X)$, extending the $d_i(X)$ to all positive $i$.

Taking the $\fm^m$-torsion of the exact sequence \eqref{eq:imageloc} we obtain two short exact sequences:
\begin{equation}\label{eq:se-1}
0 \to H[\fm^m]^{\nu} \to \rH^1_{\F(n_j)}(K,T^{(m)})^{\nu} \xrightarrow{\loc_\ell} R/\fm^{m-a_{\nu}}\to 0, \ \ 0\leq a_{\nu}\leq e, 
\end{equation}
and
\begin{equation}\label{eq:se-2}
0 \to H[\fm^m]^{-\nu} \to \rH^1_{\F(n_j)}(K,T^{(m)})^{-\nu} \xrightarrow{\loc_\ell} R/\fm^{b_{\nu}}\to 0, \ \ N-e\leq b_{\nu}\leq N,
\end{equation}
The bounds on the exponents for the modules on the right come from the choice of $\ell$ with respect to the classes
$c^{\nu} = (1+\nu\tau)\varpi^{k-m}c_0$ and $c^{-\nu}$.  Global duality then yields two additional short exact sequences
\begin{equation}\label{eq:se-3}
0 \to H[\fm^m]^{\nu} \to \rH^1_{\F(n_{j+1})}(K,T^{(m)})^{\nu} \xrightarrow{\loc_\ell} R/\fm^{a'_{\nu}}\to 0, \ \ 0\leq a_{\nu}'\leq e, 
\end{equation}
and
\begin{equation}\label{eq:se-4}
0 \to H[\fm^m]^{-\nu} \to \rH^1_{\F(n_{j+1})}(K,T^{(m)})^{-\nu} \xrightarrow{\loc_\ell} R/\fm^{m-b'_{\nu}}\to 0, \ \ b_{\nu}'\leq b_{\nu}.
\end{equation}

As $\ord(c^{\nu}) = m$ and $\ord(\loc_\ell(c^{\nu})) = m-a$ with $a\leq e$, it follows from \eqref{eq:se-1}
that $d_i(H[\fm^m]^{\nu}) \leq d_{i+1}(\rH^1_{\F(n_j)}(K,T^{(m)})^{\nu})+e$.  From \eqref{eq:se-3} we deduce
that $d_i(\rH^1_{\F(n_{j+1})}(K,T^{(m)})^{\nu})\leq d_i(H[\fm^m]^{\nu}) +e$, and so
$$
d_i(\rH^1_{\F(n_{j+1})}(K,T^{(m)})^{\nu}) \leq d_{i+1}(\rH^1_{\F(n_j)}(K,T^{(m)})^{\nu})+2e.
$$
Let $i_0=\rho_{x_je}(\rH^1_{\F(n_j)}(K, T^{(m)})^{\nu})$. It follows that
$d_{i_0}(\rH^1_{\F(n_{j+1})}(K,T^{(m)})^{\nu}) \leq d_{i_0+1}(\rH^1_{\F(n_j)}(K,T^{(m)})^{\nu})+2e \leq x_j e + 2e$, 
so 
\begin{equation}\label{eq:rhopluspart}
\rho_{(x_j+2)e}(\rH^1_{\F(n_{j+1})}(K, T^{(m)})^{\nu}) \leq i_0-1  = \rho_{x_je}(\rH^1_{\F(n_j)}(K, T^{(m)})^{\nu})-1.
\end{equation}

Next we consider the exact sequence \eqref{eq:se-4}. 
One of the cyclic summands of $\rH^1_{\cF(n_{j+1})}(K,T^{(m)})^{-\nu}$, say one isomorphic to $R/\fm^{d_t}$ for $d_t = d_t(\rH^1_{\cF(n_{j+1})},T^{(m)})^{-\nu})$,
surjects under $\loc_\ell$ onto $R/\fm^{m-b_{\nu}'}$.  There is therefore an $R$-module surjection $H[\fm^m]^{-\nu}\twoheadrightarrow \rH^1_{\cF(n_{j+1})}(K,T^{(m)})^{-\nu}/(R/\fm^{d_t})$, and so -- upon taking Pontryagin duals -- an $R$-module injection
$\rH^1_{\cF(n_{j+1})}(K,T^{(m)})^{-\nu}/(R/\fm^{d_t})\hookrightarrow H[\fm^m]^{-\nu}$. From this together with the injection in \eqref{eq:se-2} we conclude
that
\begin{equation}\label{eq:missingd}
d_i(\rH^1_{\cF(n_{j+1})}(K,,T^{(m)})^{-\nu}) \leq \begin{cases} d_i(\rH^1_{\cF(n_{j})}(K,T^{(m)})^{-\nu}) & i< t \\ 
 d_{i-1}(\rH^1_{\cF(n_{j})}(K,T^{(m)})^{-\nu}) & i>t.\end{cases}
 \end{equation}
 It then follows that
 \begin{equation*}\label{eq:minusineq}
  \rho_{x_je}(\rH^1_{\cF(n_{j+1})}(K,T^{(m)})^{-\nu}) \leq \rho_{x_je}(\rH^1_{\cF(n_{j})}(K,T^{(m)})^{-\nu})+1
 \end{equation*}
 Together with \eqref{eq:rhopluspart} this implies
 \begin{equation}\label{eq:rhofirstineq}
 \rho_{(x_j+2)e}(n_{j+1}) \leq \rho_{x_je}(n_j),
 \end{equation}
 which proves the first inequality in (f) for any choice of $x_{j+1}\geq x_j+2$.


Suppose 
\begin{equation}\label{eq:case1} 
\tag{$\spadesuit$}
x_j e +e < N = \operatorname{exp} (\rH^1_{\F(n_j)}(K,T^{(m)})^{-\nu}).
\end{equation}
Considering the exact sequence \eqref{eq:se-2}, we see that one of the cyclic summands of $\rH^1_{\F(n_j)}(K,T^{(m)})^{-\nu}$, say one isomorphic
to $R/\fm^{d_h}$ for $d_h = d_h(\rH^1_{\F(n_j)}(K,T^{(m)})^{-\nu})$, surjects onto $R/\fm^{b_{\nu}}$. As $b_{\nu}\geq N - e > x_je$, it follows that
$d_h> x_je$, and so $\rho_{x_je}(\rH^1_{\F(n_j)}(K,T^{(m)})^{-\nu}/(R/\fm^{d_h})) = \rho_{x_je}(\rH^1_{\F(n_j)}(K,T^{(m)})^{-\nu}) - 1$.
On the other hand, since $d_h \leq N$, the $R$-module surjection $H[\fm^m]^{-\nu}\twoheadrightarrow \rH^1_{\cF(n_{j})}(K,T^{(m)})^{-\nu}/(R/\fm^{d_h})$ 
has kernel annihilated by $\varpi^e$ and so $\rho_{x_je+e}(H[\fm^m]^{-\nu}) \leq \rho_{x_je}(\rH^1_{\F(n_j)}(K,T^{(m)})^{-\nu}/(R/\fm^{d_h}))$.
Combining these inequalities yields
$$
\rho_{x_je+2e}(H[\fm^m]^{-\nu})\leq \rho_{x_je+e}(H[\fm^m]^{-\nu}) \leq \rho_{x_je}(\rH^1_{\F(n_j)}(K,T^{(m)})^{-\nu}) - 1.
$$
From the previously noted injection $\rH^1_{\cF(n_{j+1})}(K,T^{(m)})^{-\nu}/(R/\fm^{d_t})\hookrightarrow H[\fm^m]^{-\nu}$ we deduce that
$$
\rho_{x_je+2e}(\rH^1_{\cF(n_{j+1})}(K,T^{(m)})^{-\nu}) - 1 \leq \rho_{x_je+2e}(\rH^1_{\cF(n_{j+1})}(K,T^{(m)})^{-\nu}/(R/\fm^{d_t})) \leq
\rho_{x_je+2e}(H[\fm^m]^{-\nu}).
$$
Together the two displayed equations yield
$$
\rho_{x_je+2e}(\rH^1_{\cF(n_{j+1})}(K,T^{(m)})^{-\nu})\leq \rho_{x_je}(\rH^1_{\F(n_j)}(K,T^{(m)})^{-\nu}).
$$
Combining this with \eqref{eq:rhopluspart} we find 
\begin{equation}\label{eq:rhocase1}
\text{\eqref{eq:case1}} \implies \rho_{(x_j+2)e}(n_{j+1}) < \rho_{x_je}(n_j).
\end{equation}

If \eqref{eq:case1} does not hold, then $N=\operatorname{exp} ( \rH^1_{\Fcal(n_j)}(K,T^{(m)})^{-\nu})\leq (x_j+1)e$.
From \eqref{eq:missingd} we see that each $d_i(\rH^1_{\F(n_{j+1})}(K,T^{(m)})^{-\nu})\leq (x_j+1)e$ except possibly for $i=t$. 
So either $\rho_{(x_j+1)e}(\rH^1_{\F(n_{j+1})}(K,T^{(m)})^{-\nu})=0$ or $t=1$ and $\rho_{(x_j+1)e}(\rH^1_{\F(n_{j+1})}(K,T^{(m)})^{-\nu})=1$.
If $\rho_{(x_j+1)e}(\rH^1_{\F(n_{j+1})}(K,T^{(m)})^{-\nu})=0$, then it follows from \eqref{eq:rhopluspart} that 
$\rho_{(x_j+2)e}(n_{j+1}) < \rho_{x_je}(n_j)$. 
If $d_t \leq (x_j+3)e$, then we also have $\rho_{(x_j+3)e}(n_{j+1}) < \rho_{x_je}(n_j)$ by similar reasoning, where as above $d_t = d_t(\rH^1_{\cF(n_{j+1})},T^{(m)})^{-\nu})$ is the length of the summand
surjecting under $\loc_\ell$ onto $R/\fm^{m-b_{\nu}'}$ in \eqref{eq:se-4}.

Suppose $\rho_{(x_j+3)e}(n_j)=1$. The proof of \eqref{eq:rhofirstineq} also holds for $x_j$ replaced with $x_j+3$, which shows
$\rho_{(x_j+5)e}(n_{j+1}) \leq \rho_{(x_{j}+3)e}(n_{j})=1$.

To summarize, we have shown that 
\begin{equation}\label{eq:equalityconds}
\rho_{(x_j+5)e}(n_{j+1}) = \rho_{x_je}(n_j)>1 \implies
\left\{
\begin{array}{c} \text{\eqref{eq:case1} does not hold}, \\ \rho_{(x_j+1)e}(\rH^1_{\F(n_{j+1})}(K,T^{(m)})^{-\nu})=1, \\  
d_t > (x_j+3)e, \\ \rho_{(x_j+3)e}(n_j)>1.
\end{array} 
\right.
\end{equation}

Suppose $\rho_{(x_j+5)e}(n_{j+1}) = \rho_{x_je}(n_j)>1$.  Since $\rho_{(x_j+5)e}(n_{j+1}) \leq \rho_{(x_j+3)e}(n_{j+1}) \leq \rho_{(x_j+2)e}(n_{j+1}) \leq \rho_{x_je}(n_j)$,
it follows that we also have $\rho_{(x_j+3)e}(n_{j+1})=\rho_{(x_j+2)e}(n_{j+1}) = \rho_{x_je}(n_j)$. 
Furthermore, all the conditions on the right-hand side of
\eqref{eq:equalityconds} hold. As $d_t> (x_j+3)e$, $\mathrm{exp}(\rH^1_{\F(n_{j+1})}(K,T^{(m)})^{-\nu} )> (x_j+3)e$,
so $\rho_{(x_j+3)e}(\rH^1_{\F(n_{j+1})}(K,T^{(m)})^{-\nu} )\geq 1$. 
As $\rho_{(x_j+3)e}(\rH^1_{\F(n_{j+1})}(K,T^{(m)})^{-\nu}) \leq \rho_{(x_j+1)e}(\rH^1_{\F(n_{j+1})}(K,T^{(m)})^{-\nu})=1$, it follows
that $\rho_{(x_j+3)e}(\rH^1_{\F(n_{j+1})}(K,T^{(m)})^{-\nu}) =1$. 
Since $\rho_{x_je}(n_j) \geq 2$,
$\rho_{(x_j+3)e}(\rH^1_{\F(n_{j+1})}(K,T^{(m)})^{\nu}) =  \rho_{(x_j+3)e}(n_{j+1})  - \rho_{(x_j+3)e}(\rH^1_{\F(n_{j+1})}(K,T^{(m)})^{-\nu})  
= \rho_{x_je}(n_j) -1 \geq 1$. This shows
\begin{equation}\label{eq:equalityconds2}
\rho_{(x_j+5)e}(n_{j+1}) = \rho_{x_je}(n_j)>1 \implies
\left\{
\begin{array}{c} \rho_{(x_j+2)e}(n_{j+1}) = \rho_{x_je}(n_j), \\ \rho_{(x_j+3)e}(\rH^1_{\F(n_{j+1})}(K,T^{(m)})^{\pm})\geq 1.
\end{array}\right.
\end{equation}

We can now complete our definition of $x_{j+1}$ and the verification of (f):
\begin{itemize}
\item[(i)] If $\rho_{(x_j+5)e}(n_{j+1}) = 1$ or $\rho_{(x_j+5)e}(n_{j+1}) < \rho_{x_je}(n_j)$, then 
$x_{j+1}:= x_j+5$.
\item[(ii)] If $\rho_{(x_j+5)e}(n_{j+1}) = \rho_{x_je}(n_j)>1$, then $x_{j+1}: = x_j +2$ and 
\eqref{eq:equalityconds2} shows that 
$$\rho_{(x_{j+1}+1)e}(\rH^1_{\F(n_{j+1})}(K,T^{(m)})^{\pm}) \geq 1.$$
\end{itemize}
Hence (f) holds for $i=j+1$.

%

Finally, we verify (g) for $i=j+1$. Suppose $\rho_{x_{j-1}e}(n_{j-1})>1$. 
If $\rho_{x_j e}(n_j) < \rho_{x_{j-1}e}(n_{j-1})$, then $\rho_{x_{j+1}e}(n_{j+1})\leq \rho_{x_j e}(n_j) < \rho_{x_{j-1}e}(n_{j-1})$. Suppose then that
$\rho_{x_j e}(n_j) = \rho_{x_{j-1}e}(n_{j-1})$. It follows from (f) in the case $i=j$ (which holds by induction) 
that $\rho_{(x_j+1)e}(\rH^1_{\F(n_j)}(K,T^{(m)})^{\pm}) \geq 1$. This implies
that \eqref{eq:case1} holds, and so, by \eqref{eq:equalityconds} above, $\rho_{(x_j+5)e}(n_{j+1}) < \rho_{x_j e}(n_j)= \rho_{x_{j-1}e}(n_{j-1})$.
As $x_{j+1} = x_j+5$ in this case (see (i) above), this shows that (g) holds.
\end{proof}

This completes the proof of Theorem \ref{thm:Zp-twisted}.

\begin{rmk}\label{rmkkolyvaginargument}
The strategy to produce the $n_i\in \mathscr{N}^{(k)}$ as in the preceding proof is similar to the one employed in \cite[$\S$3.3.3]{eisenstein} but technically more delicate. In particular, comparing condition (b) here and condition (b) in \emph{op.\,cit.}, one will notice that the one stated herein is weaker. 
The issue is exactly the one hinted at previously: we have removed the dependence of the error term on $\alpha$, but at the cost of having to work with the $\fm^m$-torsion. This prevents us from proving that we can take $b_{M(n_{t-1})}(n_t)=d_1(M(n_{t-1}))-e$ (notation as in \cite{eisenstein}) -- and so being able to work with the stronger condition (b). This results in the need for the additional conditions (f) and (g), which can be thought of as an induction step on ``the number of summands of the $\fm^m$-torsion of the Selmer groups that are not bounded by (controlled) multiples of $e$''. Philosophically, this is what is done in the proof of \cite[Lemma 1.6.4]{howard}, where however, as $e=0$, one can work with the $\fm$-torsion and has an equality for the order in Proposition \ref{prop:prime2}. 
\end{rmk}

\subsection{The anticyclotomic Iwasawa main conjectures}\label{subsec:ac-IMC}

With Theorem \ref{thm:Zp-twisted} in hand, we can state and prove a strengthening of \cite[Thm.~3.4,1]{eisenstein} and consequent strengthenings 
of \cite[Thm.~4.1.2, Thm.~4.2.2]{eisenstein} as well as \cite[Cor.~4.2.3]{eisenstein}. Let $\Lambda=\Lambda_K^-$, 
%
and note that the modules 
\[
\mathcal{X}=\rH^1_{\Fcal_\Lambda}(K,M_E)^\vee,\quad\rH^1_{\Fcal_\Lambda}(K,\mathbf{T})
\] 
in \cite[\S{3.4}]{eisenstein} are the same as the modules $\X_{\rm ord}(E/K_\infty^-)$ and $\mathfrak{S}_{\rm ord}(E/K_\infty^-)$ in $\S\ref{subsec:Sel-str}$, respectively.

\begin{thm}\label{thm:howard}
Assume $E(K)[p]=0$ and suppose there is a Kolyvagin system $\kappa\in\mathbf{KS}(\mathbf{T},\Fcal_\Lambda,\cL_E)$ with $\kappa_1\neq 0$. Then $\mathfrak{S}_{\rm ord}(E/K_\infty^-)$ has $\Lambda$-rank one, and there is a finitely generated torsion $\Lambda$-module $M$ such that
\begin{itemize}
\item[(i)] $\X_{\rm ord}(E/K_\infty^-)\sim\Lambda\oplus M\oplus M$,
\item[(ii)] ${\rm char}_\Lambda(M)$ divides ${\rm char}_\Lambda\bigl(\mathfrak{S}_{\rm ord}(E/K_\infty^-)/\Lambda\kappa_1\bigr)$ in $\Lambda[1/p]$.
\end{itemize}
\end{thm}
\begin{proof}
The proof is the same as that of \cite[Theorem 3.4.1]{eisenstein}. However, the height one prime $(\gamma^- -1)\subset \Lambda$ was excluded from the analysis
in {\em loc.~cit.} because the error term in \cite[Thm.~3.2.1]{eisenstein} increases as the characters $\alpha$ get $p$-adically closer to $1$. 
Replacing the appeal to {\em op.~cit.} with one to Theorem~\ref{thm:Zp-twisted} 
for the case of the prime $(\gamma^- -1)$,
yields the theorem.
\end{proof}

Applying Theorem~\ref{thm:howard} to the Kolyvagin system $\kappa^{\rm Hg}=\{\kappa_n^{\rm Hg}\}_{n\in\mathscr{N}}$ of \cite[Thm.~4.1.1]{eisenstein}, we thus obtain the following. 

\begin{thm}\label{thm:howard-HP}
Assume $E(K)[p]=0$. Then $\mathfrak{S}_{\rm ord}(E/K_\infty^-)$ has $\Lambda$-rank one, and there is a finitely generated torsion $\Lambda$-module $M$ such that
\begin{itemize}
\item[(i)] $\X_{\rm ord}(E/K_\infty^-)\sim\Lambda\oplus M\oplus M$,
\item[(ii)] ${\rm char}_\Lambda(M)$ divides ${\rm char}_\Lambda\bigl(\mathfrak{S}_{\rm ord}(E/K_\infty^-)/\Lambda\kappa_1^{\rm Hg}\bigr)$ in $\Lambda[1/p]$.
\end{itemize}
\end{thm}

Using this, 
we conclude just as for \cite[Thm.~4.2.2]{eisenstein}:

\begin{thm}\label{thm:BDP-IMC}
Suppose $K$ satisfies hypotheses {\rm (\ref{eq:intro-Heeg})}, {\rm (\ref{eq:intro-spl})}, and {\rm (\ref{eq:intro-disc})}, and that
$E[p]^{ss}=\mathbb{F}_p(\phi)\oplus\mathbb{F}_p(\psi)$ as $G_\bQ$-modules, with $\phi\vert_{G_p}\neq\mathds{1},\omega$.
Then $\X_{\rm Gr}(E/K_\infty^-)$ is $\Lambda$-torsion, and
\[
{\rm char}_\Lambda(\X_{\rm Gr}(E/K_\infty^-))\Lambda^{\rm ur}=(\Lcal_p^{\rm BDP}(f/K))
\]
as ideals in $\Lambda^{\rm ur}$. Hence the anticyclotomic Iwasawa--Greenberg main conjecture in Conjecture~\ref{conj:IMC} holds.
\end{thm}

And just as for \cite[Cor.~4.2.3]{eisenstein} (noting that the ambiguity by powers of $p$ in \emph{loc.\,cit.} can be removed), we then have Theorem~\ref{thm:AC} in the Introduction:

\begin{cor}\label{cor:PR} 
Suppose $K$ satisfies hypotheses {\rm (\ref{eq:intro-Heeg})}, {\rm (\ref{eq:intro-spl})},  {\rm (\ref{eq:intro-disc})}, and that $E[p]^{ss}=\mathbb{F}_p(\phi)\oplus\mathbb{F}_p(\psi)$ as $G_\bQ$-modules, with $\phi\vert_{G_p}\neq\mathds{1},\omega$. 
Then both ${\rm H}^1_{\Fcal_\Lambda}(K,\mathbf{T})$ and 
$\rH^1_{\Fcal_\Lambda}(K,M_E)^\vee$ have $\Lambda$-rank one, and
\[
\mathrm{char}_\Lambda\bigl(\rH^1_{\Fcal_\Lambda}(K,M_E)^\vee_{\rm tors}\bigr)=\mathrm{char}_\Lambda\bigl({\rm H}^1_{\Fcal_\Lambda}(K,\mathbf{T})/\Lambda\kappa_\infty\bigr)^2.
\]
\end{cor}

\section{Mazur's main conjecture}\label{sec:MC}

In this section we put everything together to deduce the proof of Theorem~\ref{thm:CYC} in the Introduction.

\subsection{The main result}

The following is the main result of this paper.

\begin{thm}\label{thm:A}
Let $E/\Q$ be an elliptic curve, and $p>2$ a prime of good reduction for $E$ such that $E[p]^{ss}=\mathbb{F}_p(\phi)\oplus\mathbb{F}_p(\psi)$ 
with $\phi\vert_{G_{\Q_p}}\neq 1,\omega$. Then the module $\X_{\rm ord}(E/\Q_\infty)$ is $\Lambda_\Q$-torsion, with
\[
\ch_{\Lambda_\Q}\bigl(\X_{\rm ord}(E/\Q_\infty)\bigr)=\bigl(\Lcal_p^{\rm MSD}(E/\Q)\bigr).
\]
In other words, Mazur's main conjecture for $E$ holds.
\end{thm}

\subsection{Proof of Mazur's main conjecture}

We divide it into three steps, similarly as we did in $\S\ref{sec:one}$. Choose an imaginary quadratic field $K$ satisfying hypotheses (\ref{eq:Heeg}), (\ref{eq:spl}), and (\ref{eq:disc}). As usual, we let $E_\bullet$ denote the elliptic curve in the isogeny class of $E$ constructed in \cite{wuthrich-int}, and put $S=\Sigma\smallsetminus\{p,\infty\}$.
\sk

\noindent\emph{Step 1}. The $p$-adic $L$-functions $\Lcal_p^{\rm PR}(E_\bullet/K)^+\in\Lambda_K^+$ and $\Lcal_p^{\rm BDP}(f/K)\in\Lambda_K^{-,{\rm ur}}$ are nonzero: For $\Lcal_p^{\rm PR}(E_\bullet/K)^+$ this follows from  Proposition~\ref{prop:comp-Lcyc}, Theorem~\ref{thm:MSD}, and Rohrlich's nonvanishing result \cite{rohrlich-cyc}; and for $\Lcal_p^{\rm BDP}(f/K)$ this is part of Theorem~\ref{thm:BDP}. Fix an integer $m>0$ such that 
\begin{equation}\label{eq:p^m}
\Lcal_p^{\rm PR}(E_\bullet/K)^+\neq 0\in\Lambda_K^+/p^m\Lambda_K^+,
\end{equation}
and take a crystalline character $\alpha:\Gamma_K^-\rightarrow R^\times$ with $\alpha\equiv 1\pmod{\varpi^m}$ such that  $\Lcal_p^{\rm BDP}(f(\alpha)/K)(0)\neq 0$. 

By Theorem~\ref{thm:BDP-IMC} 
we then have that $\X_{\rm Gr}(E_\bullet(\alpha)/K_\infty^-)$ is $\Lambda_K^-$-torsion, with
\begin{equation}\label{eq:ac-alpha}
\Fcal_{\rm Gr}(E_\bullet(\alpha)/K_\infty^-)(0)\;\sim_p\;\Lcal_p^{\rm BDP}(f(\alpha)/K)(0)\neq 0,
\end{equation}
where $\Fcal_{\rm Gr}(E_\bullet/K_\infty^-)\in R\llbracket{T}\rrbracket$ is any characteristic power series for $\X_{\rm Gr}(E_\bullet(\alpha)/K_\infty^-)$.
\sk

\noindent\emph{Step 2}. Since $\alpha$ is anticyclotomic, for $\alpha\neq 1$ the nonvanishing of $\Lcal_p^{\rm PR}(E_\bullet(\alpha)/K)^+$ and $\Lcal_p^{\rm Gr}(f(\alpha)/K)^+$ is not automatic, but for our choice of $\alpha$ we can show this easily.

\begin{lemma}\label{lem:nonzeroLs}
With $\alpha$ chosen as above, the $p$-adic $L$-functions $\Lcal_p^{\rm PR}(E_\bullet(\alpha)/K)^+$ and $\Lcal_p^{\rm Gr}(f(\alpha)/K)^+$ are both nonzero.
\end{lemma}

\begin{proof}
For $\Lcal_p^{\rm PR}(E_\bullet(\alpha)/K)^+$, this is clear from (\ref{eq:p^m}) and the congruence of Lemma~\ref{lem:cong-L}; and for $\Lcal_p^{\rm Gr}(f(\alpha)/K)^+$ it follows from the relations
\begin{equation}\label{eq:sim-p}
\Lcal_p^{\rm Gr}(f(\alpha)/K)^+(0)=\Lcal_p^{\rm Gr}(f(\alpha)/K)^-(0)\,\sim_p\,\Lcal_p^{\rm BDP}(f(\alpha)/K)(0),
\end{equation}
using Proposition~\ref{prop:comp-Lac} for the last equality up to a $p$-adic unit.
\end{proof}

In light of this nonvanishing, by Corollary~\ref{cor:KLZ} the class $BF_{\alpha}^+\in\mathfrak{S}_{\rm ord,\rm rel}(E_\bullet(\alpha)/K_\infty^+)$ is nonzero, and so by Theorem~\ref{thm:BF-ES} the Selmer group $\X_{\rm ord,\rm str}(E_\bullet(\alpha^{-1})/K_\infty^+)$ is $\Lambda_K^+$-torsion, with
\begin{equation}\label{eq:BF-div-alpha}
\ch_{\Lambda_K^+}\bigl(\X_{\rm ord,\rm str}(E_\bullet(\alpha^{-1})/K_\infty^+)\bigr)\supset\ch_{\Lambda_K^+}\bigl(\mathfrak{S}_{\rm ord,\rm rel}(E_\bullet(\alpha)/K_\infty^+)/\Lambda_K^+BF_{\alpha}^+\bigr)
\end{equation}
in $\Lambda_K^+\otimes\Q_p$. Note the need to invert $p$ in this divisibility, an ambiguity that we shall remove in the next result.


\begin{lemma}\label{lem:int-div-PR}
The module $\X_{\rm ord}(E_\bullet(\alpha)/K_\infty^+)$ is $\Lambda_K^+$-torsion, with
\begin{equation}\label{eq:cyc-ord-alpha}
\ch_{\Lambda_K^+}\bigl(\X_{\rm ord}(E_\bullet(\alpha)/K_\infty^+)\bigr)\supset\bigl(\Lcal_p^{\rm PR}(E_\bullet(\alpha)/K)^+\bigr) 
\end{equation}
in $\Lambda_K^+$.
\end{lemma}

\begin{proof}
The combination of Proposition~\ref{prop:equiv}, Lemma~\ref{lem:nonzeroLs}, and (\ref{eq:BF-div-alpha}) shows that $\X_{\rm ord}(E_\bullet(\alpha)/K_\infty^+)$ is $\Lambda_K^+$-torsion, with the claimed divisibility holding in $\Lambda_K^+\otimes\Q_p$. Let $S=\Sigma\smallsetminus\{p,\infty\}$. By Corollary~\ref{cor:imp-IMC}, it follows that $\mathfrak{X}_{\rm ord}^S(E_\bullet(\alpha)/K_\infty^+)$ is also $\Lambda_K^+$-torsion, with
\begin{equation}\label{eq:ord-cyc-Qp}
\ch_{\Lambda_K^+}\bigl(\X_{\rm ord}^S(E_\bullet(\alpha)/K_\infty^+)\bigr)\supset\bigl(\Lcal_p^{\rm PR}(E_\bullet(\alpha)/K)^{+,S}\bigr) 
\end{equation}
in $\Lambda_K^+\otimes\Q_p$. Denote by $\Fcal_{\rm ord}^S(E_\bullet(\alpha)/K_\infty^+)\in\Lambda_K^+$ a characteristic power series for $\X_{\rm ord}^S(E_\bullet(\alpha)/K_\infty^+)$, so from the above we have
\begin{equation}\label{eq:div-h}
\Fcal_{\rm ord}^S(E_\bullet(\alpha)/K_\infty^+)\cdot h=\varpi^k\cdot\Lcal_p^{\rm PR}(E_\bullet(\alpha)/K)^{+,S}
\end{equation}
for some $h\in\Lambda_K^+$ and $k\in\Z$. If $k<0$ there is nothing to show, so assume $k\geq 0$. From Kato's divisibility \cite[Thm.~17.4]{kato-euler-systems} (refined to an integral statement as in \cite[Thm.~16]{wuthrich-int}), Proposition~\ref{prop:comp-Lcyc}, and Proposition~\ref{prop:comp-Selcyc} we have that the untwisted Selmer group $\X_{\rm ord}(E_\bullet/K_\infty^+)$ is $\Lambda_K^+$-torsion, with the integral divisibility
\begin{equation}\label{eq:final-cyc-ord-1}
\ch_{\Lambda_K^+}\bigl(\X_{\rm ord}(E_\bullet/K_\infty^+)\bigr)\supset\bigl(\Lcal_p^{\rm PR}(E_\bullet/K)^+\bigr)
\end{equation}
in $\Lambda_K^+$, and so from Proposition~\ref{prop:prim-imprim-ord} we get that $\X_{\rm ord}^S(E_\bullet/K_\infty^+)$ is also $\Lambda_K^+$-torsion, and we have the integral divisibility
\begin{equation}\label{eq:final-cyc-ord-1-imp}
\ch_{\Lambda_K^+}\bigl(\X_{\rm ord}^S(E_\bullet/K_\infty^+)\bigr)\supset\bigl(\Lcal_p^{\rm PR}(E_\bullet/K)^{+,S}\bigr)
\end{equation}
in $\Lambda_K^+$. By the congruences of Proposition~\ref{prop:cong-Sel} and Lemma~\ref{lem:cong-L}, it follows from (\ref{eq:final-cyc-ord-1-imp}) that for $\alpha$ sufficiently close to $1$ (i.e. taking $m>0$ in \eqref{eq:p^m} sufficiently large) the $\mu$-invariant of $\Fcal_{\rm ord}^S(E_\bullet(\alpha)/K_\infty^+)$ is at most that of $\Lcal_p^{\rm PR}(E_\bullet(\alpha)/K)^{+,S}$. Thus from (\ref{eq:div-h}) we see that $h$ is divisible by $\varpi^k$, and therefore the divisibility (\ref{eq:ord-cyc-Qp}) holds in $\Lambda_K^+$.  Together with Corollary~\ref{cor:imp-IMC}, this yields the result.
\end{proof}

From the preceding two lemmas and Proposition~\ref{prop:equiv}, we deduce that $\X_{\rm Gr}(E_\bullet(\alpha)/K_\infty^+)$ is $\Lambda_K^+$-torsion, with
\begin{equation}\label{eq:cyc-alpha-final}
\bigl(\Fcal_{\rm Gr}(E_\bullet(\alpha)/K_\infty^+)\bigr)\supset\bigl(\Lcal_p^{\rm Gr}(f(\alpha)/K)^+\bigr)
\end{equation}
in $\Lambda_K^{+,\rm ur}$, where $\Fcal_{\rm Gr}(E_\bullet(\alpha)/K_\infty^+)\in\Lambda_K^+$ is any characteristic power series for $\X_{\rm Gr}(E_\bullet(\alpha)/K_\infty^+)$. Together with the relations
\[
\Fcal_{\Gr}(E_\bullet(\alpha)/K_\infty^+)(0)\,\sim_p\,
\Fcal_{\Gr}(E_\bullet(\alpha)/K_\infty^-)(0)\,\sim_p\,\Lcal_p^{\rm BDP}(f(\alpha)/K)(0)\,\sim_p\,\Lcal_p^{\rm Gr}(f(\alpha)/K)^-(0)\neq 0
\]
following from Proposition~\ref{prop:euler-char}, relation (\ref{eq:ac-alpha}), and Proposition~\ref{prop:comp-Lac}, and noting that 
\[
\Lcal_p^{\rm Gr}(f(\alpha)/K)^-(0)=\Lcal_p^{\rm Gr}(f(\alpha)/K)^+(0),
\]
by easy commutative algebra (see \cite[Lem.\,3.2]{skinner-urban}) it follows that equality holds in (\ref{eq:cyc-alpha-final}), so we have
\[
\bigl(\Fcal_{\rm Gr}(E_\bullet(\alpha)/K_\infty^+)\bigr)=\bigl(\Lcal_p^{\rm Gr}(f(\alpha)/K)^+\bigr).
\]
(Note that conditions (a) and (b) in Lemma~\ref{lem:coinv} needed for the above application of Proposition~\ref{prop:euler-char} follow from our choice of $\alpha$ with $\Lcal_p^{\rm BDP}(f(\alpha)/K)(0)\neq 0$ together with the reciprocity law of \cite[Thm.\,5.7]{cas-hsieh1} and the result of \cite[Thm.\,3.2.1]{eisenstein} applied to the Heegner point Kolyvagin system in \cite[\S{4.1}]{eisenstein}.) 

In particular, for $\alpha$ as above sufficiently close to $1$, we conclude by Proposition~\ref{prop:equiv} that $\X_{\rm ord}(E_\bullet(\alpha)/K_\infty^+)$ is $\Lambda_K^+$-torsion, with
\begin{equation}\label{eq:alpha-equality}
\ch_{\Lambda_K^+}\bigl(\X_{\rm ord}(E_\bullet(\alpha)/K_\infty^+)\bigr)=\bigl(\Lcal_p^{\rm PR}(E_\bullet(\alpha)/K)^+\bigr).
\end{equation}

\noindent\emph{Step 3}. 
We are now in a position to prove Conjecture~\ref{conj:IMC-K} for $\X_{\rm ord}(E/K_\infty^+)$.

\begin{thm}\label{thm:PR-IMC}
Let $E/\Q$ be an elliptic curve, and $p>2$ a prime of good reduction for $E$ such that $E[p]^{ss}=\mathbb{F}_p(\phi)\oplus\mathbb{F}_p(\psi)$ 
with $\phi\vert_{G_{\Q_p}}\neq 1,\omega$. 
Then module $\X_{\rm ord}(E/K_\infty^+)$ is $\Lambda_K^+$-torsion, with
\[
\ch_{\Lambda_K^+}(\X_{\rm ord}(E/K_\infty^+))=(\Lcal_p^{\rm PR}(E/K)^+).
\]
\end{thm}

\begin{proof}
By Proposition~\ref{prop:isog-inv}, it suffices to prove the result of $E_\bullet$. Let $S=\Sigma\smallsetminus\{p,\infty\}$. As shown above, from Kato's work we can deduce that $\X_{\rm ord}^S(E_\bullet/K_\infty^+)$ is $\Lambda_K^+$-torsion, and we have the integral divisibility
\[
\ch_{\Lambda_K^+}\bigl(\X_{\rm ord}^S(E_\bullet/K_\infty^+)\bigr)\supset\bigl(\Lcal_p^{\rm PR}(E_\bullet/K)^{+,S}\bigr)
\]
in $\Lambda_K^+$. Take a character $\alpha:\Gamma_K^-\rightarrow R^\times$ with 
\[
\alpha\equiv 1\;({\rm mod}\,\varpi^m)
\] 
for some $m\gg 0$ so that the equality (\ref{eq:alpha-equality}) holds. (Note that the argument in \emph{Step 1} and \emph{Step 2} leading to that equality only excludes finitely many $\alpha$.) By Corollary~\ref{cor:imp-IMC} it follows that $\X_{\rm ord}^S(E_\bullet/K_\infty^+)$ is also $\Lambda_K^+$-torsion, and denoting by $\Fcal_{\rm ord}^S(E_\bullet/K_\infty^+)$ and $\Fcal^S_{\rm ord}(E_\bullet(\alpha)/K_\infty^+)\in\Lambda_K^+$ characteristic power series for $\X^S_{\rm ord}(E_\bullet/K_\infty^+)$ and $\X^S_{\rm ord}(E_\bullet(\alpha)/K_\infty^+)$, respectively, we have 
\begin{equation}\label{eq:key-cong}
\Fcal^S_{\rm ord}(E_\bullet/K_\infty^+)\equiv\Fcal^S_{\rm ord}(E_\bullet(\alpha)/K_\infty^+)\equiv\Lcal_p^{\rm PR}(E_\bullet(\alpha)/K)^{+,S}\equiv\Lcal_p^{\rm PR}(E_\bullet/K)^{+,S}\;({\rm mod}\,\varpi^m),
\end{equation}
as a consequence of Proposition~\ref{prop:cong-Sel}, the combination of (\ref{eq:alpha-equality}) and Corollary~\ref{cor:imp-IMC}, and Lemma~\ref{lem:cong-L}, respectively. Taking $m\gg 0$, it follows from the congruence (\ref{eq:key-cong}) that $\Fcal_{\rm ord}^S(E_\bullet/K_\infty^+)$ and $\Lcal_p^{\rm PR}(E_\bullet/K)^{+,S}$ have \emph{the same} Iwasawa $\lambda$- and $\mu$-invariants, and so equality holds in (\ref{eq:final-cyc-ord-1-imp}). By Corollary~\ref{cor:imp-IMC}, this yields the proof of the theorem.
\end{proof}

The proof of Mazur's main conjecture for $E$ now follows easily.

\begin{proof}[Proof of Theorem~\ref{thm:A}] 
Choose an imaginary quadratic field $K$ satisfying hypotheses (\ref{eq:disc}), (\ref{eq:Heeg}), and (\ref{eq:spl}). As before, from \cite{kato-euler-systems} and \cite{wuthrich-int} we have the divisibilities
\[
\ch_{\Lambda_\Q}\bigl(\X_{\rm ord}(E/\Q_\infty)\bigr)\supset\bigl(\Lcal_p^{\rm MSD}(E/\Q)\bigr),\quad\ch_{\Lambda_\Q}\bigl(\X_{\rm ord}(E^K/\Q_\infty)\bigr)\supset\bigl(\Lcal_p^{\rm MSD}(E^K/\Q)\bigr)
\]
in $\Lambda_\Q$. If $\ch_{\Lambda_\Q}(\X_{\rm ord}(E/\Q_\infty))\neq(\Lcal_p^{\rm MSD}(E/\Q))$ then from Proposition~\ref{prop:comp-Lcyc} and Proposition~\ref{prop:comp-Selcyc} we conclude that 
\[
\bigl(\Lcal_p^{\rm PR}(E/K)^+\bigr)\subsetneq
\ch_{\Lambda_K^+}(\X_{\rm ord}(E/K_\infty^+)\bigr),
\]
but this contradicts Theorem~\ref{thm:PR-IMC}. Thus $\ch_{\Lambda_\Q}(\X_{\rm ord}(E/\Q_\infty))=(\Lcal_p^{\rm MSD}(E/\Q))$, concluding the proof.
\end{proof}

\subsection*{Conflict of interest statement} On behalf of all authors, the corresponding author states that there is no conflict of interest and that the manuscript complies to the Ethical Rules applicable for this journal.

\bibliography{references}
\bibliographystyle{alpha}

\end{document}